\documentclass[11pt]{amsart}

\usepackage{amsmath, amsfonts, amssymb, graphicx, amsthm, mathtools, enumitem, hyperref, setspace, tikz-cd, tabu, dsfont, caption, subcaption, bm}
\usepackage[left=1in, right=1in, top=1in, bottom=1in]{geometry}
\usepackage[ruled,vlined]{algorithm2e}
\RestyleAlgo{ruled}
\usepackage{lipsum}
\makeatletter
\g@addto@macro{\endabstract}{\@setabstract}
\newcommand{\authorfootnotes}{\renewcommand\thefootnote{\@fnsymbol\c@footnote}}%
\makeatother
\usepackage[capitalise]{cleveref}
\usepackage{booktabs}
\usepackage{caption}
\usepackage{tablefootnote}
\usepackage{subcaption}
\usepackage{breakcites}
\DeclareMathOperator*{\argmax}{arg\,max}
\DeclareMathOperator*{\argmin}{arg\,min}
\usepackage{setspace}
\usepackage{etoolbox}
\AtBeginEnvironment{algorithm}{\setstretch{1.5}}
\usepackage{url}
\usepackage{natbib}
\usepackage{hyperref}
\usepackage[capitalise]{cleveref}
\usepackage{graphicx}
\usepackage{caption}
\usepackage{bm}
\usepackage[ruled,vlined]{algorithm2e}
\RestyleAlgo{ruled}
\usepackage{booktabs}
\usepackage{bbm}
\usepackage{caption}
\usepackage{tablefootnote}
\usepackage{subcaption}
\usepackage{makecell}
\usepackage{array}
\newcommand{\PreserveBackslash}[1]{\let\temp=\\#1\let\\=\temp}
\newcolumntype{C}[1]{>{\PreserveBackslash\centering}p{#1}}
\newcolumntype{R}[1]{>{\PreserveBackslash\raggedleft}p{#1}}
\newcolumntype{L}[1]{>{\PreserveBackslash\raggedright}p{#1}}
\usepackage{multicol}

\theoremstyle{definition}
\newtheorem{theorem}{Theorem}

\newtheorem{proposition}{Proposition}
\newtheorem{assumption}{Assumption}
\newtheorem{lemma}{Lemma}
\newtheorem{observation}{Observation}
\newtheorem{definition}{Definition}

\theoremstyle{definition}

\newtheorem{main problem}[theorem]{Main problem}

\theoremstyle{remark}

\begin{document}
\begin{center}
	\LARGE
	Best of Many in Both Worlds: Online Resource Allocation with Predictions
	under Unknown Arrival Model \par \bigskip
	
	\normalsize
	\authorfootnotes
	Lin An, Andrew A. Li, Benjamin Moseley, and Gabriel Visotsky \par 
	Tepper School of Business, Carnegie Mellon University, Pittsburgh, Pennsylvania 15213 \par 
	linan, aali1, moseleyb, gvisotsk@andrew.cmu.edu \par \bigskip
	
	\today
\end{center}

	\begin{abstract}
Online decision-makers often obtain predictions on future variables, such as arrivals, demands, inventories, and so on. These predictions can be generated from simple forecasting algorithms for univariate time-series, all the way to state-of-the-art machine learning models that
leverage multiple time-series and additional feature information. However, the prediction accuracy is unknown to decision-makers a priori, hence blindly following the predictions can be harmful. In this paper, we address this problem by developing algorithms that utilize predictions in a manner that is robust to the unknown prediction accuracy.

We consider the Online Resource Allocation Problem, a generic model for online decision-making, in which a limited amount of resources may be used to satisfy a sequence of arriving requests. Prior work has characterized the best achievable performances when the arrivals are either generated stochastically (i.i.d.) or completely adversarially, and shown that algorithms exist which match these bounds under both arrival models, without ``knowing'' the underlying model. To this backdrop, we introduce predictions in the form of shadow prices on each type of resource.
Prediction accuracy is naturally defined to be the distance between the predictions and the actual shadow prices. 

We tightly characterize, via a formal lower bound, the extent to which any algorithm can optimally leverage predictions (that is, to ``follow'' the predictions when accurate, and ``ignore'' them when inaccurate) without knowing the prediction accuracy or the underlying arrival model. Our main contribution is then an algorithm which achieves this lower bound.
Finally, we empirically validate our algorithm with a large-scale experiment on real data from the retailer {\em H\&M}.
\end{abstract}

\newpage
\section{Introduction}

\label{Section intro}
Allocating a limited set of resources to satisfy different requests as they arrive is a key process in many operations problems. For example, airlines need to decide whether or not to accept a certain offer for a seat at a given price, while the total number of seats is limited \citep{talluri2006theory,ball2009toward}; online retailers must choose which products to display to a browsing customer, taking into account inventory levels \citep{gallego2004managing,luce2012individual}; internet search engines auction off impressions to advertisers with limited budgets \citep{edelman2007internet,mehta2007adwords}.
The \textit{Online Resource Allocation Problem} is a generic model for all of these settings. In the problem, requests arrive sequentially, each request consisting of multiple actions to choose from, and each action generating some reward and consuming some subset of resources. Actions are selected online, i.e. without knowing future requests. Resources are limited, and the objective is to maximize the total reward received across all time periods. While the Online Resource Allocation Problem is arguably ubiquitous in practice today, it may be worth highlighting a few motivating examples:

\begin{itemize}
\item {\bf Network Revenue Management:} The canonical example of network revenue management is airlines, for whom the resource to be allocated is the seats on future flights. This problem can be challenging as requests may involve multiple seats (e.g. group bookings, or even individuals flying multiple flight legs on a single itinerary), and can have highly varying prices due to the ever-growing number of fare classes.


\item {\bf Assortment Optimization:} Consider an online retailer. At various moments during a customer's browsing session, the retailer chooses a set of products to display (e.g. when the customer has placed a search query, or in-cart recommendations).
The customer then selects each product with some probability, based on their personal preferences and the assortment itself. In this assortment optimization problem, each opportunity to display an assortment is a request, the reward of an action is the (expected) profit earned by displaying a certain assortment, and the resources are the product inventories. 

\item  {\bf Online Matching (AdWords, Online Auctions):} Online matching is itself a general model formulating various two-sided markets, such as AdWords and online auctions. As a special case of the Online Resource Allocation Problem, the  online nodes (impressions in AdWords) can be viewed as the arrivals, and the capacities of the offline nodes (budgets of bidders) can be viewed as resources. 
\end{itemize}

At present, there are by and large two approaches to the Online Resource Allocation Problem. The traditional approach is to assume a model for the arriving requests, and develop algorithms that have the best worst-case guarantees. The two most popular arrival models are {\em stochastic} and {\em adversarial}, where the former assumes each arrival is drawn independently from an unknown underlying distribution, and the latter assumes nothing about the arrivals -- they can be as bad as possible. 
A result of \cite{balseiro2023best} states that the best possible (worst-case) performance can be achieved simultaneously under both arrival models {\em without knowing the actual arrival model}. This is quite nice -- in practice, if we think of the stochastic and adversarial models as broadly representing stationary and nonstationary processes, respectively, then their algorithm is able to leverage the ability to ``learn on the fly'' in stationary settings, while remaining robust to arbitrary nonstationarities.
Still, the optimality here is with respect to worst-case guarantees, which might be overly pessimistic.

The second, arguably more modern approach, is to utilize some sort of {\em predictions}  on the future arrivals. Here we use the term ``prediction'' in the broadest possible sense, ranging from simple time-series forecasting models, to state-of-the-art machine learning algorithms based on large amounts of data, to human judgement, and even combinations of all of the above. 
The de facto approach in practice is to take these predictions as fact (in a way we will make formal momentarily). 
Naturally, the performance of this approach relies heavily on the accuracy of the predictions, which is not guaranteed: Figure \ref{fig:first-example}, taken from \cite{an2023nonstationary}, shows this for the relatively simple task of forecasting daily visits to two  stores.

\begin{figure}[h]
	\centering
	\includegraphics[width=5in]{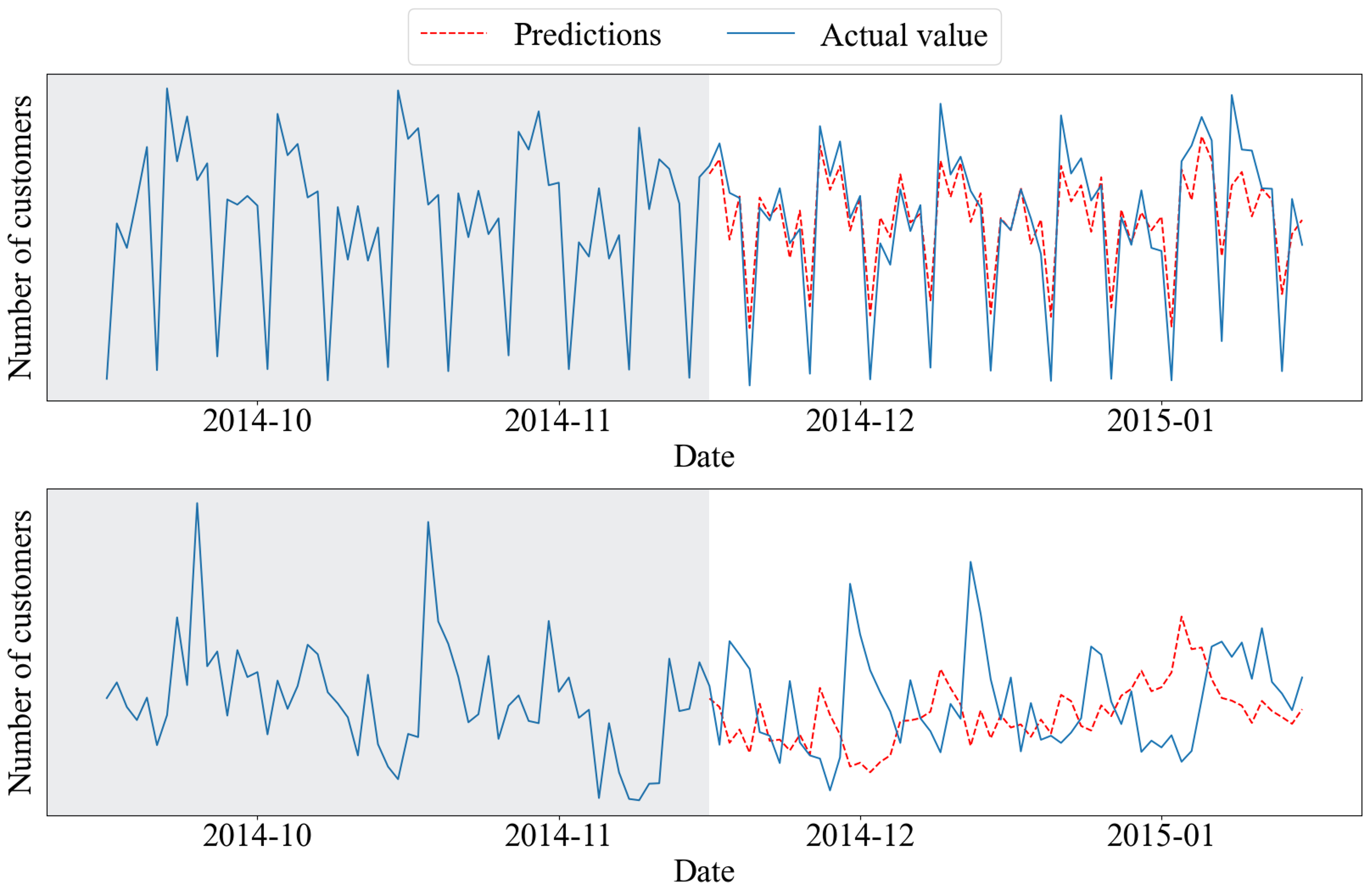}
	\caption{(Figure and caption from \cite{an2023nonstationary}) Daily number of customers (in blue), from September 2014 to January 2015, at two different stores in the Rossmann drug store chain. Predictions (in red), starting November  2014, are generated using Exponential Smoothing with the same fitting process. The store in the upper sub-figure has substantially more accurate predictions ($R^2=0.88$) than that of the lower sub-figure ($R^2=0.11$). } 
	\label{fig:first-example}
\end{figure}

To summarize so far, the Online Resource Allocation Problem admits algorithms with optimal worst-case guarantees (for both stochastic and adversarial arrival models, simultaneously), and these algorithms can be significantly better or worse than following predictions, depending on the prediction quality. This suggests the opportunity to design an algorithm that leverages predictions {\em optimally}, in the sense that the predictions are utilized when accurate, and ignored when inaccurate. Ideally, such an algorithm should operate without knowledge of (a) the accuracy of the predictions and (b) the method with which they are generated. {\em This is precisely what we seek to accomplish in this paper. }

\subsection{Online Resource Allocation with Predictions}
The primary purpose of this paper is to develop an algorithm that optimally incorporates {\em predictions} (defined in the most generic sense possible) into the Online Resource Allocation Problem. 
Without predictions, the {\em Online Resource Allocation Problem} consists of a finite horizon of $T$ time periods and a limited number of $m$ types of resources. At each time period, a decision must be made which will consume a certain set of resources and yield a certain reward. The form of these individual decision problems changes over time and is unknown in advance.

Following \cite{balseiro2023best}, we consider both the stochastic and adversarial arrival models. 
Under the stochastic model, we measure the performance of any algorithm via its {\em regret}, which is the difference in the total reward earned by an optimal offline algorithm (i.e. one that ``knows'' the entire arrival sequence beforehand) versus the reward earned by the algorithm. At minimum we aim to design algorithms that achieve {\em sub-linear} (i.e.~$o(T)$) regret, as such an algorithm would earn a per-period reward that is on average no worse than the optimal, as $T$ grows. Under the adversarial  model, sub-linear regret is impossible to achieve in the worst case, so instead we measure the performance of any algorithm via its {\em competitive ratio}, which is the ratio between the total reward earned by an optimal offline algorithm and the algorithm's reward. In other words, if an algorithm is $\alpha$-competitive, then it can always obtain a total reward that is no less than $1/\alpha$ times the reward of the optimal algorithm. Without predictions, \cite{arlotto2019uniformly} proved that under the stochastic model, any algorithm incurs at least $\Omega(T^{\frac{1}{2}})$ regret. Similarly \cite{balseiro2019learning} proved that under the adversarial model, any algorithm has at least an $\alpha^*$ competitive ratio, where $\alpha^*$ depends on simple problem parameters (it is these two bounds which \cite{balseiro2023best} matches simultaneously). We seek to design algorithms that go beyond these worst-case bounds using accurate predictions, but also enjoy the same guarantees using inaccurate predictions.

To that end, we introduce the notion of {\em predictions}. Our prediction is of the form of an $m$-dimensional vector $\hat{\mu}$ whose coordinates represent a predicted {\em shadow price} for each of the $m$ resources.
We will show that this form of prediction satisfies certain nice properties including that it (a) immediately translates to a decision policy, and (b) there always exists ``perfect'' predictions which achieve near-optimal reward.

We measure the quality of any prediction $\hat{\mu}$ by its $\ell_1$ distance to the closest perfect prediction $\mu^*$\footnote{The choice of the $\ell_1$ norm follows naturally from our analysis, though any $\ell_p$ norm where $p\in[1,2]$ yields similar performance guarantees for our algorithms.}. Specifically, we use an \textit{accuracy parameter} $a \ge 0$, defined as the largest $a$ such that $||\hat{\mu}-\mu^*||_1\in O(T^{-a})$. Notice that when $a=0$ the prediction is effectively useless, and as $a$ increases the prediction becomes more accurate. 
We call this problem \textit{Online Resource Allocation with Predictions}. Our primary challenge will be to design algorithms with performances that are robust in the prediction quality {\em without} having access to $a$.

\subsection{A Simple Example}
\begin{figure}[ht]
	\captionsetup[subfigure]
	{font=scriptsize,labelfont=scriptsize}
	\centering
	\begin{subfigure}{0.2\textwidth}
		\centering
		\includegraphics[width=\linewidth]{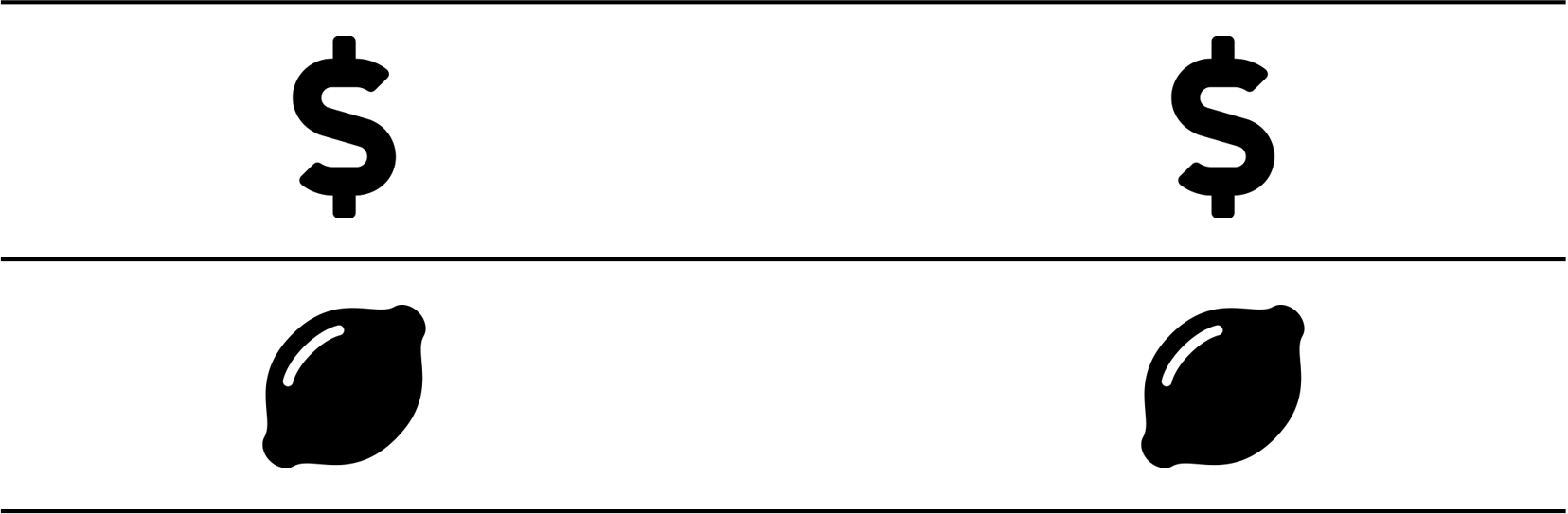}
	\end{subfigure}
	\qquad \qquad \quad
	\begin{subfigure}{0.2\textwidth}
		\centering
		\includegraphics[width=\linewidth]{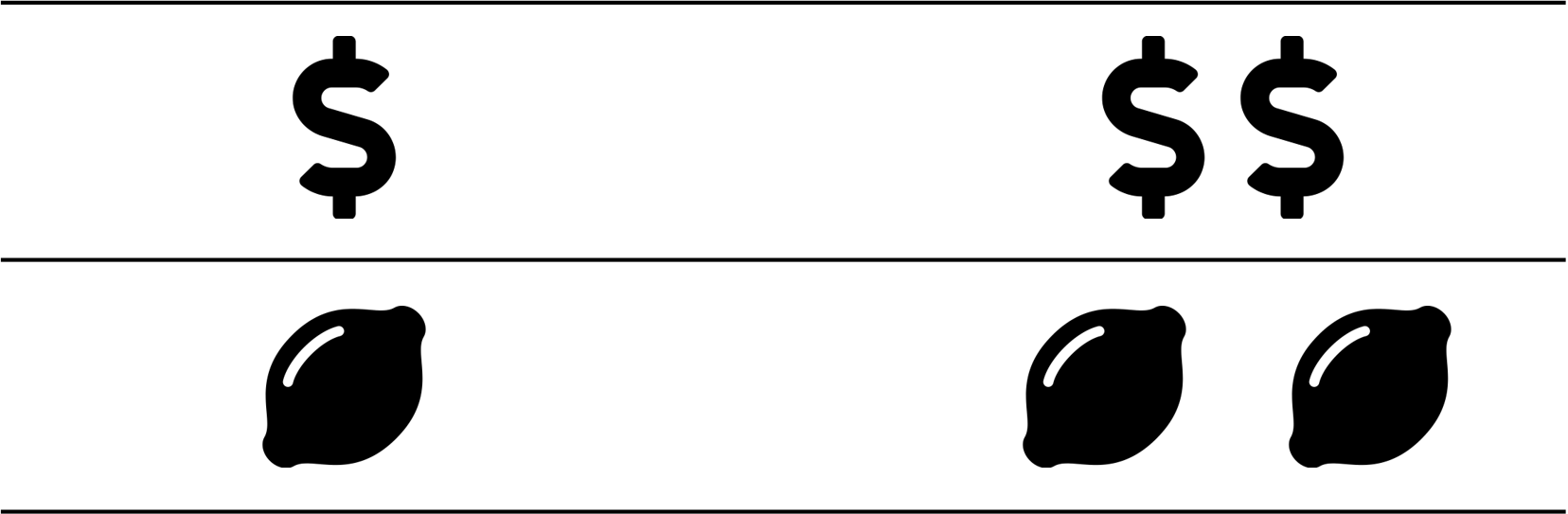}
	\end{subfigure}
	\caption{Two potential arrival sequences for an online resource allocation problem with a single resource (two lemons) and two time periods. The left (right) sequence falls under the stochastic (adversarial) arrival model.}
	\label{fig: Lemon}
\end{figure}
\vspace{-1.5em}

Before outlining our contributions, it is worth describing a simple example to illustrate the challenge we face in incorporating predictions of unknown accuracy. Consider the example in Figure \ref{fig: Lemon}, which depicts two potential arrival sequences for an online resource allocation problem with a single resource (two lemons) and two time periods. In both sequences, a single lemon may be sold for \$1 in the first time period. This same offer occurs in the second time period for the left sequence, but the right sequence offers \$2 for two lemons (this offer may not be split). Note that the left (right) sequence falls under the stochastic (adversarial) arrival model, and critically, an algorithm can not distinguish between the two sequences until the second time period. Still, a simple algorithm (accept all offers when feasible) achieves zero regret under the left sequence, and a competitive ratio of $1/2$ under the right sequence (incidentally, $\alpha^*=2$ for this problem instance).

However, suppose now we introduce a prediction, whose implication is that the first offer should be rejected. Under the right sequence, this constitutes a ``good'' prediction, and so an algorithm ideally would follow this prediction and collect the optimal \$2. Under the left sequence, this constitutes a ``bad'' prediction, and so an algorithm ideally would ignore this prediction but still achieve good regret as the arrival model is stochastic. It is of course impossible to do both of these. 

More generally, there are essentially four ``worlds'' we must consider, depending on whether the arrival model is stochastic or adversarial, and whether the predictions are accurate or inaccurate. This example demonstrates that we can not hope to achieve the best of all four worlds simultaneously.\footnote{The language here is indeed in reference to the ``best of both worlds'' literature, e.g.~\cite{balseiro2023best}.} Instead, we will find that just as the accuracy of the predictions is best characterized continuously between ``perfect'' and ``bad'' (via our accuracy parameter $a$), the arrival model is best characterized continuously along a carefully-defined interpolation between the stochastic and adversarial models.

\subsection{Our Contributions}
Our primary contributions can be summarized as follows.

\noindent
{\color{black}{\bfseries\sffamily 1. A Nonstationary Arrival Model and a Lower Bound:}} We define a parameterized class of arrival models that interpolates between the stochastic and adversarial models. In particular, we define a precise measure of the stationarity of an arrival sequence (Definition \ref{def:stationary}), in terms of two values $\lambda$ and $\delta$, such that $(\lambda,\delta)=(0,0)$ (loosely) corresponds to the stochastic model, and $(\lambda,\delta)=(1,1)$ corresponds to the adversarial model (the two values are in general distinct, and have nice time-series interpretations in terms of trend and seasonality). Notably, this stationarity measure is defined for deterministic arrival sequences, and thus the corresponding (nonstationary) arrival models can be defined without positing a stochastic generative model.

The primary value of this new measure of stationarity is that it tightly characterizes the extent to which we can expect an algorithm to leverage predictions of unknown quality. Specifically, we prove the following lower bound:
\setcounter{proposition}{5}
\begin{proposition}[Lower Bound, Informal]
	For any $0\leq \lambda\leq 1$ and $0<\delta\leq 1$, and any algorithm,
	at least one of the following holds:
	
	\begin{enumerate}
		\item[(1)] Under the stochastic arrival model, the algorithm incurs $\Omega(T)$ worst-case regret;
		\item[(2)] Under the adversarial arrival model, with $(\lambda,\delta)$-stationary arrivals, the algorithm's worst-case reward is at most 
		
		$$(1-\lambda)\max\left\{\frac{1}{\alpha^*}\textup{OPT}, \textup{PRD}\right\}  - \Omega(\delta T).$$ 
	\end{enumerate}
\end{proposition}
Here recall that $\alpha^*$ is the best competitive ratio (without predictions) which we will specify later. $\text{OPT}$ denotes the optimal offline reward. 
Following the actions induced by the predictions also yields a certain amount of reward, which we denote by $\text{PRD}$. 

Now if $(\lambda,\delta)=(1,1)$, i.e. the adversarial model with no restrictions, then Proposition \ref{prop:lower bound of full} implies that we can not simultaneously achieve sub-linear regret under the stochastic arrival model and a meaningful reward under the adversarial model (our lemon example was already evidence of this). However, for smaller values of $\lambda$ and $\delta$, we can hope for sub-linear regret and an adversarial reward close to the best value of $\max\left\{\frac{1}{\alpha^*}\textup{OPT}, \textup{PRD}\right\}$.

\noindent
{\color{black}{\bfseries\sffamily 2. An Optimal Algorithm:}} We construct an algorithm that {\em optimally} leverages predictions, in the sense that it achieves the lower bound of Proposition \ref{prop:lower bound of full}, without knowing the underlying arrival model (stochastic or adversarial) and without knowing the prediction accuracy. In particular, our main theoretical result is the following:

\begin{theorem}[Upper Bound, Informal]
	Given a prediction $\hat{\mu}$ with (unknown) accuracy parameter $a$ and given $0<\delta\leq 1$, there exists an algorithm such that, under mild (and tight) assumptions, both of the following hold: 
	\begin{enumerate}
		\item[(1)] Under the stochastic arrival model, the algorithm incurs $\tilde{O}(T^{\frac{1}{2}-a})\footnote{The $\tilde{O}(\cdot)$ notation hides logarithmic factors. Technically the regret should be $\tilde{O}(\max\{T^{\frac{1}{2}-a},1\})$, since if $a>\frac{1}{2}$ the regret bound should be a constant. For the simplicity of exposition we drop the obvious regret bound of $O(1)$ in the introduction section.}$ worst-case regret;
		\item[(2)]Under the adversarial arrival model, with $(\lambda,\delta)$-stationary arrivals, the algorithm's worst-case reward is at least 
		
		$$(1-\lambda)\max\left\{\frac{1}{\alpha^*}\textup{OPT}, \textup{PRD}\right\}  - O(\delta T).$$ 
	\end{enumerate}
\end{theorem}
\addtocounter{theorem}{-1}

Our theoretical results are summarized in bold in Table \ref{results summary}, with a comparison to the problem with no predictions and the problem with predictions of known accuracy. 

\begin{table}[h!]
	\centering
	\small
	\begin{tabular}{@{}lp{22mm}p{40mm}p{56mm}@{}}
		\toprule
		Arrival Model & \hfil Without & \hfil With Predictions of & \hfil  With Predictions of\\ &  \hfil Predictions & \hfil Known Accuracy & \hfil \text{ } Unknown Accuracy \vspace{0.5em}\\ 
		\hline & \\[-1.5ex]
		\vspace{0.7em}
		Stochastic (regret)& \hfil$O(T^{\frac{1}{2}})$ & \hfil$O(T^{\frac{1}{2}-a})$ & \hfil$\bm{\tilde{O}(T^{\frac{1}{2}-a})}$ \\
		\vspace{0.5em}
		Adversarial (reward) & \hfil$\frac{1}{\alpha^*}\text{OPT}$ & \hfil$\max\{\frac{1}{\alpha^*}\text{OPT},\text{PRD}\}$ & \hfil$\bm{(1-\lambda)\max\{\frac{1}{\alpha^*}\textbf{OPT},\textbf{PRD}\}-\delta T}$
		\\
		\bottomrule
	\end{tabular}
	\vspace{0.5em}
	\caption{Summary of our main theoretical results (in bold). Each entry has a corresponding algorithm that, without knowing the underlying arrival model, achieves the stated performance simultaneously for both stochastic arrivals and adversarial arrivals. Each entry also has a matching lower bound.}
	\label{results summary}
\end{table}

\noindent
{\bfseries\sffamily 3. A Large-Scale Experiment:} 
We demonstrate the practical value of our model (namely Online Resource Allocation with Predictions) and our algorithm via empirical results on an H\&M (Hennes \& Mauritz AB) dataset, which contains two years of transactions for 105,542 products. The experiment we conducted corresponds to the assortment problem we motivated above. For each experiment, which runs for three simulated months, we applied our algorithm and compared its performance against the two most-natural baseline algorithms: the optimal algorithm \textbf{without predictions,} and the simple policy which always utilizes the predictions (these correspond to the two ``existing approaches'' described previously). On any given experimental instance, the maximum (minimum) of the rewards gained by these two baselines can be viewed as the best (worst) we can hope for. Thus we measure performance in terms of the proportion of the gap between these two rewards gained by our algorithm, so if this ``optimality gap'' is close to 1, our algorithm performs almost as well as the better one of the two baselines. 

We used three popular forecasting algorithms to generate predictions of various quality. We find that with Prophet forecasts, the average optimality gap is 0.68; with ARIMA forecasts, the average optimality gap is 0.58; with Exponential Smoothing forecasts, the average optimality gap is 0.53. This demonstrates that our algorithm performs well,
irrespective of the quality of the predictions.

\vspace{1em}
The remainder of this paper is organized as follows. The current section concludes with a literature review. In Section \ref{Section formulation} we introduce our model of the Online Resource Allocation with Predictions. In Section \ref{Section main results} we present preliminary results of the problem without predictions as well as our main results. We then introduce our algorithms and proofs of main results in Section \ref{Section main algorithm}, which solves the Online Resource Allocation with Predictions under both arrival models without knowing the underlying arrival model. Section \ref{Section experiments} contains our experimental results, and Section \ref{Section conclusion} concludes the paper. 

\subsection{Literature Review}
\textbf{Online Resource Allocation.}
Allocating scarce resources to satisfy requests arriving online has been extensively studied under various models. Related works assuming the arrivals are stochastic (i.i.d. or random order) including \cite{devanur2009adwords,feldman2010online,devanur2011near,agrawal2014dynamic,kesselheim2014primal}, and \cite{gupta2016experts}, where the objective is to achieve sub-linear worst-case regret. Another popular arrival model is the adversarial arrival model, under which is usually impossible to achieve sub-linear worst-case regret. Instead, the objective is to obtain a certain factor of the rewards of the offline optimum, which is called competitive analsis. For example, \cite{mehta2007adwords} and \cite{buchbinder2007online} studied the AdWords problem where they obtained a $(1-1/e)$-fraction of the optimal allocation in hindsight.

Apart from considering different arrival models separately, there has been a recent line of work in developing algorithms that achieve good performance under various arrival models simultaneously without knowing the underlying arrival model. \cite{mirrokni2012simultaneous} considered the AdWords problem and gave an algorithm with the optimal competitive ratio under adversarial arrivals and improved competitive ratios (though not asymptotic optimality) under stochastic arrivals. \cite{balseiro2023best} studied the Online Resource Allocation Problem and provided a mirror descent algorithm that achieves the optimal worst-case regret under stochastic arrivals and the optimal competitive ratio under adversarial arrivals. The main algorithm in our paper also attains the optimal performance under both stochastic and adversarial arrivals.

\vspace{1em}
\noindent\textbf{Algorithms with Predictions.} With the ubiquity of large data-sets and machine-learning models, theory and practice of augmenting online algorithms with machine-learned predictions have been emerging. This framework has lead to new models of algorithm analysis for
going beyond worst-case analysis. Some applications on  optimization problems including revenue optimization \citep{munoz2017revenue, balseiro2022single,golrezaei2023online}, caching \citep{lykouris2021competitive,rohatgi2020near}, online matching \citep{lavastida2021using,jin2022online}, online scheduling \citep{purohit2018improving,lattanzi2020online}, the secretary problem \citep{antoniadis2020secretary,dutting2021secretaries,dutting2023prophet}, and the nonstationary newsvendor problem \citep{an2023nonstationary}. Most of the related works analyzed the algorithms' performances using competitive analysis and obtain optimal \textit{consistency-robustness (consistency-competitiveness)} trade-off, where \textit{consistency} is the competitive ratio of the algorithm when the prediction is accurate, and \textit{robustness (competitiveness)} is the competitive ratio of the algorithm regardless the prediction's accuracy. In contrast, under the stochastic arrivals we do regret analysis on our algorithm and prove our algorithm has near-optimal worst-case regret without knowing the prediction quality. Other papers that do regret analysis under the prediction model include \cite{munoz2017revenue} (revenue optimization in auctions), \cite{an2023nonstationary} (nonstationary newsvendor), and \cite{hu2024constrained} (constrained online two-stage stochastic optimization). 

Finally, the closest works to our own are \cite{balseiro2022single} and \cite{golrezaei2023online}, both of which are limited in the following two ways. First, the ``base'' problems they analyze (i.e.~without predictions) are strict special cases of the Online Resource Allocation Problem we study. Second, they treat prediction quality as binary: predictions are either entirely accurate or entirely inaccurate. Under this assumption, they successfully designed algorithms that achieve the optimal consistency-robustness (consistency-competitiveness) tradeoff. On the other hand, as stated earlier, we will {\em quantify} prediction quality, and provide tight guarantees for predictions of any quality.

\section{Model: The Online Resource Allocation with Predictions}\label{Section formulation}
In this section, we first formally define the \textbf{Online Resource Allocation with Predictions} problem, and then describe two standard arrival models (stochastic and adversarial) as well as their respective performance metrics.

\subsection{Problem Formulation} 
\noindent
{\bfseries\sffamily Online Resource Allocation:}
Consider a problem over $T$ time periods labeled $t=1,\dots, T$. Assume there are $m$ different types of resources. The total number of resources available is denoted by  $\rho T$, where $\rho\in\mathbb{R}^m_+$ is a non-negative $m$-dimensional vector. 
At each time period $t$, the decision-maker receives an {\em arrival} $\gamma_t=(r_t,g_t,\mathcal{X}_t)\in \mathcal{S}$. Here, $r_t:\mathcal{X}_t\to\mathbb{R}_+$ is a non-negative reward function, $g_t:\mathcal{X}_t\to\mathbb{R}^m_+$ is a non-negative resource consumption function, $\mathcal{X}_t\subset \mathbb{R}^d_+$ is a compact action space, and $\mathcal{S}$ denotes the set of all possible arrivals.\footnote{We assume throughout the paper that the reward functions and the resource consumption functions are deterministic for any given action, but our algorithms also apply when the rewards and/or consumed resources are random.} Note that we impose no convexity assumptions: $r_t(\cdot)$ can be non-concave, $g_t(\cdot)$ can be non-convex, and $\mathcal{X}_t$ can be non-convex or discrete. At each arrival $\gamma_t$, without knowing any of the future arrivals, an action $x_t\in\mathcal{X}_t$ must be selected, which yields $r_t(x_t)$ reward and consumes $g_t(x_t)$ resources. The objective is to maximize the total reward subject to the resource constraint. Finally, we assume that $\mathcal{X}_t$ always contains a $0$ vector representing a ``void'' action that consumes no resources and yields no rewards: $r(0)=0$ and $g_t(0)=0$. This ensures that there is always a feasible action available. 
This is the problem we will refer to as {\bf Online Resource Allocation (without predictions)}. 

We introduce some notations that will appear in our results later on (though our algorithms will not depend on these parameters). We denote by $\underline{\rho}=\min_{j\in[m]}\rho_j>0$ the lowest resource parameter and $\bar{\rho}=\max_{j\in[m]}\rho_j=||\rho||_\infty$ the highest resource parameter. Similarly, let $\bar{r} \ge 0$ be a constant which satisfies $\max_{x \in \mathcal{X}}r(x) \le \bar{r}$ for every $(r,g,\mathcal{X}) \in S$, and let $\bar{g} \ge \underline{g}>0$ be constants satisfying $\underline{g}\leq||g(x)||_\infty\leq \bar{g}$ for every $x\in\mathcal{X}$ and $x\neq 0$.

\vspace{1em}
\noindent
{\bfseries\sffamily Primal and Dual:} For any arrival sequence $\gamma=(\gamma_1,\dots,\gamma_T)$, we use $\text{OPT}(\gamma)$ to denote the \textit{offline/hindsight optimum}, which is the reward of the optimal solution when $\gamma$ is known in advance: 
\begin{equation}\label{eqn:offline opt}
	\text{OPT}(\gamma):=\max_{x_t\in\mathcal{X}_t} \sum_{t=1}^T r_t(x_t) \quad\text{ s.t. } \quad \sum_{t=1}^T g_t(x_t)\leq \rho T.\end{equation} 
As we will describe momentarily, it will be natural to consider predictions in terms of the dual space, so the Lagrangian dual problem of \cref{eqn:offline opt} plays a key role. 
Let $\mu\in\mathbb{R}^m_+$ be the vector of dual variables, where each $\mu_j$ can be thought of as the shadow price of resource $j$. We define
\begin{equation}\label{eqn:conjugate}
	r^*_t(\mu):=\sup_{x\in\mathcal{X}_t}\{r_t(x)-\mu^\top g_t(x)\}
\end{equation}
as the optimal opportunity-cost-adjusted reward of request $\gamma_t$, where the opportunity cost is calculated according to the shadow prices $\mu$.\footnote{We will assume that the primal optimization problems in \cref{eqn:conjugate} admit an optimal solution. This is to simplify the exposition -- our results still holds if we have an approximate of the optimal solution (see \cite{balseiro2023best}).
} Note that $r_t^*(\mu)$ is a generalization of the convex conjugate of $r_t(x)$ that takes the resource consumption function $g_t(x)$ and the action space $\mathcal{X}_t$ into account. In particular, when $g_t(x)=x$ and $\mathcal{X}_t$ is the whole space, $r_t^*(\mu)$ becomes the standard convex conjugate. For fixed arrivals $\gamma$, we define the Lagrangian dual function $D(\mu\mid\gamma):\mathbb{R}_m^+\to \mathbb{R}$ to be
\begin{equation}\label{eqn:dual}
	D(\mu\mid\gamma):=\sum_{t=1}^T r_t^*(\mu)+\mu^\top \rho T.
\end{equation} This allows us to move the constraints of \cref{eqn:offline opt} to the objective, which is easier to work with. We equip the primal space of the resource constraints $\mathbb{R}^m$ with the $\ell_\infty$ norm $||\cdot||_\infty$, and the Lagrangian dual space with the $\ell_1$ norm $||\cdot||_1$. Such choices of norms come naturally from our analysis. Similar performance guarantees of our algorithms with the dependence on the number of resources\footnote{The number of resources $m$ is viewed as a constant throughout the paper.} can be obtained using the $\ell_p$ norm for the primal space and the $\ell_q$ norm for the primal space with $1/p+1/q=1$ and $p\in[2,\infty]$. 

\vspace{1em}
\noindent
{\bfseries\sffamily Predictions:} So far, we have presented the problem of Online Resource Allocation  without predictions. As described in the introduction, it is often the case that when this problem is faced in practice, some notion of a ``prediction'' can be made which might guide us in selecting actions. Such predictions can come from a diverse set of sources ranging from simple human judgement, to forecasting algorithms built on previous demand data, to more-sophisticated machine learning algorithms trained on feature information. The process of sourcing or constructing such predictions is orthogonal to our work. Instead, we treat these predictions as  given to us endogenously (and in particular, we make no assumption on the accuracy of these predictions), and attempt to use these predictions optimally.

Notice that from \cref{eqn:conjugate}, at each time $t$, given a dual variable $\mu \in \mathbb{R}_+^m$ there is a natural action to take, namely the action 
$$x^\mu_t\in\argmax_{ x\in \mathcal{X}_t,\sum_{s=1}^{t-1} g_s(x_s)+g_t(x)\leq \rho T}\left\{r_t(x)-\mu^{\top} g_t(x)\right\}.\footnote{We again assume this optimization problem and other similar-style optimization problems in this paper admit an optimal solution to simplify the exposition. When the right hand side contains more than one action, we naturally choose the action that has the highest reward.}$$ In words, $x^\mu_t$ is the ``greedy'' action that, subject to the resource constraint, maximizes the opportunity-cost-adjusted reward according to the shadow prices $\mu$. Therefore each dual variable $\mu$ essentially corresponds to an algorithm, which is simply taking the ``greedy'' action $x^\mu_t$ at each time period. Below we formally define this algorithm for any dual variable $\mu$, which we call the {\em Dual-Adjusted Greedy Algorithm} ($\textup{GRD}_{\mu}$): 
\begin{algorithm} 
	\small
	\caption{Dual-Adjusted Greedy Algorithm $\textup{GRD}_{\mu}$}
	{\bf Inputs:} Dual variable $\mu$, total time periods $T$, initial resources $G_1=\rho T$\;
	\For{$t = 1,\ldots,T$ }{
		Receive request $\left(r_t, g_t, \mathcal{X}_t\right)$\;
		Make the primal decision $x_t$ and update the remaining resources $G_t$:\\
		\centerline{$x_t\in\arg \max _{ x\in \mathcal{X}_t,g_t(x)\leq G_t}\left\{r_t(x)-\mu^{\top} g_t(x)\right\}$;}
		\centerline{$G_{t+1}\gets G_t-g_t\left(x_t\right).$}}
	
	\label{alg:prediction}
\end{algorithm}

Let $R(\textup{GRD}_{\mu}\mid\gamma)$ denote the reward obtained by $\textup{GRD}_{\mu}$ with arrival sequence $\gamma$ and dual variable $\mu$.\footnote{Later we will formally extend the notion of $R(\textup{ALG}\mid\gamma)$ to any algorithm ALG.} We say a dual variable $\mu^*$ is a \textit{``perfect'' dual variable} (of some arrival sequence $\gamma$) if $\textup{GRD}_{\mu^*}$ yields rewards that is at most a constant away from $\text{OPT}$ (hence essentially optimal). It can be shown that there always exists a ``perfect'' dual variable:
\setcounter{proposition}{0}
\begin{proposition}[``Perfect'' Dual Variable]\label{corollary:perfect dual}
	For any arrival sequence $\gamma$,
	$$\max_{\mu\in\mathbb{R}_+^m}R(\textup{GRD}_{\mu}\mid\gamma)+(\bar{g}/\underline{g}+1)(m+1)\bar{r}\geq \textup{OPT}(\gamma).$$
\end{proposition} 
The proof of \cref{corollary:perfect dual} appears in Appendix \ref{Appendix A}, and utilizes the Shapley-Folkman Theorem. In words, \cref{corollary:perfect dual} shows that there exists a dual variable $\mu^*$ that is essentially optimal to follow. 

With the understanding of the key role that dual variables play in our problem, we formally introduce the notion of predictions. Because dual variables induce actions, they are natural quantities to predict. We assume that before the first time period, the decision-maker receives a {\bf prediction} $\hat{\mu}\in\mathbb{R}^m_+$ of the dual variable $\mu$. 
We measure the \textbf{prediction error} of $\hat{\mu}$ by $||\hat{\mu}-\mu^*||_1$, its $\ell_1$ distance from $\mu^*$. \footnote{In the case that multiple perfect dual variables exist, we take $\mu^*$ to be the perfect dual variable that is closest to $\hat{\mu}$.} 
We quantify the prediction error using the {\bf accuracy parameter}, which is the smallest $a \in [0,\infty]$ such that 
\begin{equation*}
	||\hat{\mu}-\mu^*||_1\leq\kappa T^{-a}.
\end{equation*}
Here $\kappa>0$ is a scaling constant that we can choose, and any $\kappa$ that ensures $a \in [0,\infty]$ can be chosen without affecting our performance bound asymptotically. \footnote{As a technical aside, there is a natural choice of $\kappa$: Proposition 2 in \cite{balseiro2023best} implies that it is enough to only consider dual variables that lie in the $m$-dimensional rectangle $[0,\mu^{\max}_1]\times\cdots\times[0,\mu^{\max}_m]\in\mathbb{R}_+^m$ where $\mu^{\max}_j=\bar{r}/\rho_j+1$. Therefore we may assume without loss of generality that the prediction $\hat{\mu}$ we receive lies inside this rectangle (otherwise we could project $\hat{\mu}$ onto this rectangle). Thus setting $\kappa=||\mu^{\max}||_1$ ensures $a\in[0,\infty]$.} The two extreme cases of the prediction error are (1) $a=0$, in which case $\hat{\mu}$ is almost a constant away from $\mu^*$, so the prediction is effectively useless; and (2) $a=\infty$, in which case $\hat{\mu}=\mu^*$, so the prediction is perfect. We will always assume that $a$ is unknown to the decision-maker. 

In reality, a prediction $\hat{\mu}$ is unlikely to be completely useless. We make the following technical assumption on the prediction quality:

\begin{assumption}[Non-trivial Prediction]\label{assumption: Adversarial Arrival Model(a)}
	There exists a (known) function $\epsilon(T)= o(1)$ such that $||\hat{\mu}-\mu^*||_1=o(\epsilon(T))$.
	
\end{assumption}

Note that \cref{assumption: Adversarial Arrival Model(a)} does not eliminate the case $a=0$. In practice $\epsilon(T)$ can be chosen to be a function close to 1 without hurting the algorithm's performance guarantee. 

\subsection{Arrival Models and Performance Metrics}

An online algorithm $\text{ALG}$, at each time period $t$, takes an action $x_t$ (potentially randomized, but deterministic here to save on notation) based on the prediction $\hat{\mu}$, the current request $\left(r_t, g_t, \mathcal{X}_t\right)$ and the previous history $\mathcal{H}_{t-1}:=\left\{r_s, g_s, \mathcal{X}_s, x_s\right\}_{s=1}^{t-1}$, i.e., $x_t=\text{ALG}\left(r_t, g_t, \mathcal{X}_t \mid \hat{\mu}, \mathcal{H}_{t-1}\right)$. We denote the reward received by an algorithm on an arrival sequence $\gamma$ as
$$
R(\text{ALG} \mid \gamma)=\sum_{t=1}^T r_t\left(x_t\right).
$$

	

This notation is in compliant with the notation $R(\textup{GRD}_{\mu}\mid\gamma)$ which we defined for the Dual-Adjusted Greedy Algorithm. For the prediction $\hat{\mu}$, we use the \textit{Prediction Algorithm} to represent the special case of the Dual-Adjusted Greedy Algorithm where the dual variable is $\hat{\mu}$, and we let $\textup{PRD}(\gamma)=R(\textup{GRD}_{\hat{\mu}}\mid\gamma)$. As stated in \cref{corollary:perfect dual}, if $a=\infty$, i.e., if $\hat{\mu}=\mu^*$, then $\text{PRD}(\gamma)+(\bar{g}/\underline{g}+1)(m+1)\overline{r}\geq \textup{OPT}(\gamma)$, which shows the Prediction Algorithm is essentially optimal if we have a perfect prediction.

Note that for any sequence of dual variables $\mu_1,\dots,\mu_T$, following $\mu_t$ at time period $t$ gives a series of actions $x_1,\dots,x_t$. We define the \textit{depletion time} of resource $j$ by following $\mu_1,\dots,\mu_T$ to be the first time period such that the remaining amount of resources $j$ is less than $\underline{g}$, that is, after this time period no actions that consumes resource $j$ is feasible (if this never happens we set the depletion time to be $T$). We will use the depletion time to quantify the behavior of $\hat{\mu}$. Intuitively, dual variables close to $\hat{\mu}$ induce similar actions in most time periods as long as $\hat{\mu}$ is not always on the ``boundary'' of decisions, and hence their depletion time should be similar. We make this idea formal using the following assumption on the depletion time.

\begin{assumption}[Non-degenerate Prediction]\label{assumption: Adversarial Arrival Model(b)}
	There exists a constant $\zeta>0$ that satisfies the following: for any sequence of dual variables $\mu_1,\dots,\mu_T$ where $\mu_t\in\mathbb{R}^m_+$ and $||\hat{\mu}-\mu_t||_1\leq \zeta$ for all $t$, the difference between the depletion time of resource $j$ by following $\mu_1,\dots,\mu_T$ and by following $\hat{\mu}$ is in $o(T)$ for every resource $j$.
\end{assumption}

\cref{assumption: Adversarial Arrival Model(b)} roughly states that, for a sequence of dual variables that is close to the prediction, the action induced by the sequence of dual variables and the action induced by the prediction deplete resources at similar times. This assumption is reasonable and mild for the following reasons: in reality, most actions sets are discrete (such as $\{\text{accept,reject}\}$, $\mathbb{N}$, etc.). Therefore for most $\hat{\mu}$, as long as it is not at the ``boundary'' (which is often a measure-zero set), dual variables close to $\hat{\mu}$ all induce the same action. Moreover, in practice it is also unlikely for the ``boundaries'' at each time period to be the same across a majority of time periods since $r_t(\cdot)$ and $g_t(\cdot)$ vary over time, in which case \cref{assumption: Adversarial Arrival Model(b)} is satisfied with any prediction $\hat{\mu}$. Finally, perturbing each input with some small noise also turns a degenerate prediction into a non-degenerate one.

There are two primary arrival models when studying online problems: the \textit{stochastic (i.i.d) arrival model} and the \textit{adversarial arrival model}, both which we  define formally below. Our goal is to design algorithms that have good performances under both arrival models, and for predictions of different qualities. Additionally, our algorithms should be \textit{oblivious} to the arrival model and the prediction quality, i.e., they should have good performance without knowing the arrival model and the prediction quality.

\vspace{1em}
\noindent
{\bfseries\sffamily Stochastic (i.i.d.) Arrival Model:} The arrivals are drawn independently from an (unknown) underlying probability distribution $\mathcal{P}\in\Delta(\mathcal{S})$, where $\Delta(\mathcal{S})$ is the space of all probability distributions over $\mathcal{S}$. We measure the performance of an algorithm by its \textbf{regret}. Given an underlying arrival distribution $\mathcal{P}$, the regret incurred by an algorithm $\text{ALG}$ under $\mathcal{P}$ is defined as
$\mathbb{E}_{\gamma\sim \mathcal{P}^T}[\text{OPT}(\gamma)-R(\text{ALG} \mid \gamma)]$. We will be concerned with the {\bf worst-case regret} over all distributions in $\Delta(\mathcal{S})$: we define the regret of $\text{ALG}$ to be
$$\text{Regret}(\text{ALG})=\sup_{\mathcal{P}\in\Delta(\mathcal{S})}\mathbb{E}_{\gamma\sim \mathcal{P}^T}[\text{OPT}(\gamma)-R(\text{ALG} \mid \gamma)].$$ Note that if the regret $\text{Regret}(\text{ALG})$ 
is sub-linear in $T$, then algorithm $\text{ALG}$ is essentially optimal on average as $T$ goes to infinity.

\vspace{1em}
\noindent
{\bfseries\sffamily Adversarial Arrival Model:} The arrivals are arbitrary and chosen adversarially. Unlike
the stochastic arrival model, regret here can be shown to grow linearly with $T$ for any algorithm, so it is less meaningful to study the order of regret over $T$. Instead, we use \textbf{competitive ratio} as the performance metric. We say that an algorithm $\text{ALG}$ is asymptotically \bm{$\alpha$}\textbf{-competitive} if $\alpha\geq 1$ is the smallest number such that 
$$\limsup_{T\to\infty}\sup_{\gamma\in\mathcal{S}^T}\left\{\frac{1}{T}\left( \frac{1}{\alpha}\text{OPT}(\gamma) -R(\text{ALG}\mid \gamma)\right)\right\}\leq 0.$$ In words, if an algorithm is asymptotically $\alpha$-competitive, then it can obtain at least $1/\alpha$ fraction of the optimal reward in hindsight as $T$ goes to infinity.\footnote{We assume the arrival sequence $ \gamma=(\gamma_1,\dots,\gamma_T)$ is fixed in advance. Our results still hold if the arrival sequence is chosen by a non-oblivious or adaptive adversary who does not know the internal randomization of the algorithm.} 

\cite{balseiro2019learning} proved that, without predictions, the lowest competitive ratio that any algorithm can achieve is $\alpha^*=\max\{\sup_{(r,g,\mathcal{X})\in\mathcal{S}}\sup_{j\in[m],x\in\mathcal{X}}g(x)_j/\rho_j,1\}$. \cite{balseiro2023best} gave a mirror descent algorithm that achieves this competitive ratio. This is, loosely speaking, the best we might hope to achieve with ``bad'' predictions. On the other hand, we can always obtain $\text{PRD}(\gamma)$ by following the prediction, which may exceed $\text{OPT}(\gamma)/\alpha^*$ with ``good'' predictions (indeed, as we have seen in \cref{corollary:perfect dual}, $\text{PRD}(\gamma)$ can be as large as $\text{OPT}(\gamma)$).

If we knew the prediction quality beforehand, we could obtain the maximum of the two by simply choosing the better approach (this is in fact the best we can hope for). Using this as the benchmark, we will compare an algorithm's reward to this maximum. That is, for an algorithm $\text{ALG}$, we will analyze the following quantity:
$$\limsup_{T\to\infty}\sup_{\gamma\in\mathcal{S}^T}\left\{\frac{1}{T}\left(\max\left\{\frac{1}{\alpha^*}\text{OPT}(\gamma),\text{PRD}(\gamma)\right\} - R(\text{ALG}\mid \gamma)\right)\right\}.$$

\subsection{A Measure of Stationarity} 
Ideally, one would hope to develop an algorithm that achieves the ``best'' performance under both stochastic and adversarial arrivals respectively without knowing the prediction quality and the underlying arrival model. However, we will show in the next section that this is provably not achievable by any algorithm. 

For an arrival sequence $\gamma$, its \textit{stationarity} is closely related to the ``difficulty'' of solving the instance it created. As examples, an arrival sequence generated independently from the same underlying distribution can be considered as completely stationary, an arrival sequence that has certain seasonality/periodicity with small trend (e.g. generated from time series models) is less stationary, and an arrival sequence that is adversarially chosen (e.g. the lower bound instance) is completely nonstationary. Intuitively, an arrival sequence is more stationary if certain parts of the sequence with the same length are ``similar'' to each other. In this subsection we formalize this idea and develop a measure of arrival sequences' stationarity. We then use it to quantify  algorithms' performances.

For a time interval from time periods $s$ to time period $t$, let $\gamma_{s:t}=(\gamma_s,\dots,\gamma_t)$ denote the arrival sequence from time period $s$ to time period $t$. We define the \textit{$\gamma_{s:t}$-subproblem} to be the problem instance where the arrival sequence is $\gamma_{s:t}$ and the total amount of resources is $\rho (t-s+1)$, i.e., scaled down proportionally. In  particular,
the (offline) optimum of the \textit{$\gamma_{s:t}$-subproblem} is: \begin{equation*}
	\text{OPT}(\gamma_{s:t}):=\max_{x_{t'}\in\mathcal{X}_{t'}} \sum_{t'=s}^t r_{t'}(x_{t'}) \quad\text{ s.t. } \quad \sum_{t'=s}^t g_{t'}(x_{t'})\leq \rho (t-s+1).\end{equation*} 
Similarly, we use $R(\textup{GRD}_{\mu}\mid\gamma_{s:t})$ to denote the amount of reward obtained by the Dual-Adjusted Greedy Algorithm with dual variable $\mu$ on the \textit{$\gamma_{s:t}$-subproblem}.

\begin{definition}[Measure of Stationarity]\label{def:stationary}
	Given the total number of available resources $\rho T$, an arrival sequence $\gamma=(\gamma_1,\dots,\gamma_T)$ is \textit{$(\lambda,\delta)$-stationary} for some $0<\delta\leq 1$ and $0\leq \lambda\leq 1$ if for every $\mu\in\mathbb{R}^m_+$: \begin{equation*}
		\min_{k=1,\dots,\lfloor\frac{1}{\delta}\rfloor}(R(\textup{GRD}_{\mu}\mid\gamma_{1:k\delta T})+R(\textup{GRD}_{\mu}\mid\gamma_{k\delta T+1:T}))\geq (1-\lambda)R(\textup{GRD}_{\mu}\mid\gamma).
	\end{equation*}
\end{definition}

Intuitively, $\gamma$ being $(\lambda,\delta)$-stationary (roughly) means when we break $\gamma$ into two subproblems at any time period that is a multiple of $\delta T$, the rewards obtained by these two subproblems sum up to be at least $1-\lambda$ portion of the reward obtained by $\gamma$.

\noindent A few remarks are in order:
\begin{itemize}
	\item $\delta$ measures the number of possible partition time periods that makes the subproblems similar to the original problem. On the extremes, $\delta$ close to 0 means that every subproblem is similar to the original problem, and $\delta=1$ imposes no restrictions on $\gamma$. As examples, if $\gamma$ is generated i.i.d. from some underlying distribution, $\delta$ can be arbitrarily close to 0, and if $\gamma$ is periodic with small period, $\delta$ can be small.
	\item $\lambda$ measures the loss in the partition, which can be viewed as the similarity of the subproblems to the original problem. On the extremes, $\lambda=0$ means $\gamma$ can be partitioned at time periods $k\delta T$ without losing much rewards, and $\lambda=1$ imposes no restrictions on $\gamma$. As an example, $\gamma$ having small ``trend'' (i.e., the infinity-norm of the vector of possible rewards is similar across all time periods) implies small $\lambda$.
	\item Smaller $\delta$ and smaller $\lambda$ both represent more stationarity. Note that there is no single fixed $(\lambda,\delta)$ pair for an arrival sequence, but rather each choice of $\delta$ gives a corresponding $\lambda$, and smaller $\delta$ yields larger $\lambda$, i.e.~the subproblems become less similar as the partition becomes more granular. The role of $\delta$ and $\lambda$ will become clear when we state our main theorem, and we will not need to know the value of $\lambda$ in our algorithm.
	\item Unlike usual stochastic definitions of stationarity, here it is defined for deterministic arrival sequences. We show in the proposition below that if the arrivals are stochastic (i.i.d.), then the arrival sequence is $(\delta,0)$-stationary for any $\delta>0$ with high probability. This shows our definition of stationarity is compatible with stochastic definitions of stationarity.
	\begin{proposition}\label{prop:iid stationary}
		If an arrival sequence $\gamma$ is generated under the stochastic (i.i.d.) arrival model, then $\gamma$ is $(\delta,\lambda)$-stationary for any constants $\delta,\lambda>0$ with probability at least $1-O(T^{-2})$.
	\end{proposition} The proof of Proposition \ref{prop:iid stationary} appears in Appendix \ref{Appendix A}.
	\item Unlike usual definitions of stationarity, our definition is problem- (i.e. resource-) dependent. For example, $\rho=0$ and $\rho$ sufficiently large both imply $\delta$ can be arbitrarily small and $\lambda=0$, since any partition of the arrival sequence gives the same reward. 
\end{itemize}

\section{Main Results}\label{Section main results}
In this section, we first present previous results for the Online Resource Allocation problem {\em without} predictions, and then give our main results on the full problem ({\em with} predictions) along with matching lower bounds.

\subsection{Prior Results: Online Resource Allocation without Predictions}
\cite{balseiro2023best} studied the no-prediction version of our problem and gave a mirror descent algorithm which achieves the ``best'' achievable performance under both arrival models without knowing the underlying arrival model. We discuss their algorithm in detail in Appendix \ref{Appendix A.2.5}. They proved the following performance guarantee for their \textit{Mirror Descent Algorithm} (MDA):

\begin{proposition}[Theorem 1 and Theorem 2 in \cite{balseiro2023best}]\label{prop:balseiro}
	Consider the Mirror Descent Algorithm \textup{(MDA)}.
	It holds that: \begin{itemize}
		\item[(1)] If the arrivals are stochastic, $$\textup{Regret}(\textup{MDA})=O(T^{\frac{1}{2}});$$
		\item[(2)] If the arrivals are adversarial,
		$$\limsup_{T\to\infty}\sup_{\gamma\in\mathcal{S}^T}\left\{\frac{1}{T}\left(\frac{1}{\alpha^*}\textup{OPT}(\gamma) - R(\textup{MDA}\mid \gamma)\right)\right\}\leq 0.$$
	\end{itemize} 
\end{proposition}
\noindent \cref{prop:balseiro} shows that the Mirror Descent Algorithm achieves $O(T^{\frac{1}{2}})$ regret and is $\alpha^*$-competitive, which are both optimal \citep{arlotto2019uniformly, balseiro2019learning}. 

\subsection{Prior Results: Lower Bounds}
As a final step before describing our results, we present previous lower bounds for the full problem with predictions and {\em known} arrival model.

\vspace{1em}
\noindent
{\bfseries\sffamily Stochastic Arrivals:} Without predictions, the best achievable regret by any algorithm is $O(T^{\frac{1}{2}})$ \citep{arlotto2019uniformly}. 
With predictions, \cite{orabona2013dimension} gave the following lower bound on the best achievable regret with known accuracy parameter $a$:

\begin{proposition}[Corollary of Theorem 2 in \cite{orabona2013dimension}]\label{prop:stochastic lower}
	Under stochastic arrival model, given a prediction $\hat{\mu}$ with accuracy parameter $a$, no algorithm can achieve regret better than $O(\max\{T^{\frac{1}{2}-a},1\})$, even if $a$ is known.	
\end{proposition}



\noindent
{\bfseries\sffamily Adversarial Arrivals:} Without predictions, the best achievable reward by any algorithm (taken the worst-case $\gamma$ across all problem instances) is $\frac{1}{\alpha^*}\text{OPT}(\gamma)$ \citep{balseiro2019learning}.
On the other hand, simply following the actions induced by the prediction at each time yields reward $\text{PRD}(\gamma)$. As we have seen in \cref{corollary:perfect dual}, for good predictions $\text{PRD}(\gamma)$ can be as high as $\text{OPT}(\gamma)$. Hence we have the following lower bound under adversarial arrivals:

\begin{proposition}[Corollary of Theorem 3.1 in \cite{balseiro2019learning}]\label{prop:adv lower}
	Under adversarial arrival model, given a prediction  with accuracy parameter $a$, no algorithm can achieve (worst-case $\gamma$ across all problem instances) reward higher than $\max\{\frac{1}{\alpha^*}\textup{OPT}(\gamma),\textup{PRD}(\gamma)\}$, even if $a$ is known.
\end{proposition}

\subsection{Our Results: Online Resource Allocation with Predictions}
We are finally prepared to state our main result, which  to develop a single algorithm that achieves the optimal performance using predictions, without knowing the underlying arrival model and the prediction accuracy.

\begin{theorem}[Upper Bound]
	Assume that Assumptions \ref{assumption: Adversarial Arrival Model(a)} and \ref{assumption: Adversarial Arrival Model(b)} hold. Given a prediction $\hat{\mu}$ with (unknown) accuracy parameter $a$ and given $0<\delta\leq 1$, there exists an algorithm \textup{(MainALG)} such that: 
	\begin{enumerate}
		\item[(1)] If the arrivals are stochastic,
		$$\textup{Regret}(\textup{MainALG})=\tilde{O}(\max\{T^{\frac{1}{2}-a},1\});$$
		\item[(2)] If the arrivals are adversarial and $(\lambda,\delta)$-stationary, $$\limsup_{T\to\infty}\sup_{\gamma\in\mathcal{S}^T}\left\{\frac{1}{T}\left( (1-\lambda)\max\left\{\frac{1}{\alpha^*}\textup{OPT}(\gamma),\textup{PRD}( \gamma)\right\} - R(\textup{MainALG}\mid \gamma)\right)\right\}\leq \delta\bar{r}.$$
	\end{enumerate}
\end{theorem}
\noindent A few remarks are in order:
\begin{itemize}
	\item The performance guarantee under adversarial arrivals is better for smaller $\lambda$ and $\delta$, which matches the intuition that one can hope to achieve better performance with more stationary arrivals.
	\item If the arrivals are known to be adversarial, which we will discuss in the next section (Algorithm \ref{alg:adversarial} and \cref{thm:adversarial}), there exists an algorithm that achieves
	$$\limsup_{T\to\infty}\sup_{\gamma\in\mathcal{S}^T}\left\{\frac{1}{T}\left(\max\left\{\frac{1}{\alpha^*}\textup{OPT}(\gamma),\textup{PRD}( \gamma)\right\} - R(\textup{ALG}\mid \gamma)\right)\right\}\leq 0$$ That is, we are able to not suffer from nonstationarity. This is because the performance requirement is much higher for stochastic arrivals (sub-linear regret), which requires a conservative consumption of resources and hence obtains less rewards when the arrivals are highly nonstationary. This idea is elaborated in the lower bound construction below.
\end{itemize}

We provide a lower bound which shows \cref{thm:main} is tight in the sense that for any algorithm that achieves sub-linear regret under stochastic arrivals, one cannot replace $\lambda$ with any number smaller and still get meaningful guarantees under adversarial arrivals. The proof of \cref{prop:lower bound of full} appears in Appendix \ref{Appendix A.5}. The lower bound construction consists of two instances, stochastic with bad prediction and adversarial with good prediction, that are provably indistinguishable for a certain period of time.

\begin{proposition}[Lower Bound]\label{prop:lower bound of full}
	For any $0\leq \lambda'<\lambda\leq 1$, $0<\delta\leq 1$, and $K>1$,
	there exists a sequence of instances $\gamma$ of increasing time horizon $T$ that satisfies Assumptions \ref{assumption: Adversarial Arrival Model(a)} and \ref{assumption: Adversarial Arrival Model(b)}, such that for any algorithm \textup{(ALG)}, at least one of the following holds:
	
	\begin{enumerate}
		\item[(1)] The arrivals are stochastic, and
		$$\textup{Regret}(\textup{ALG})=\Omega(T);$$
		\item[(2)] The arrivals are adversarial and $(\lambda,\delta)$-stationary, and $$\limsup_{T\to\infty}\left\{\frac{1}{T}\left( (1-\lambda')\max\left\{\frac{1}{\alpha^*}\textup{OPT}(\gamma), \textup{PRD}( \gamma)\right\} - R(\textup{ALG}\mid \gamma)\right)\right\} > K\delta\bar{r}.$$ 
	\end{enumerate}
\end{proposition}

\section{Algorithm and Proof of Main Result}\label{Section main algorithm}

In this section, we first present two algorithms that utilize the prediction in an optimal way for the two arrival models respectively. 
Then we combine these two algorithms to a single algorithm that is oblivious to both the prediction quality and the arrival model, which completely solves the Online Resource Allocation with Prediction.

Our algorithms for each arrival model utilize mirror descent, which take an initial dual variable, a step-size, and a reference function\footnote{For completeness, in Appendix \ref{Appendix A.2.5} we state the standard assumptions on choosing the reference function $h(\cdot)$ for mirror descent algorithms \citep{beck2003mirror,bubeck2015convex,lu2018relatively,lu2019relative}.} as inputs. At each time period $t$, the algorithms take the action induced by the current dual variable $\mu_t$, and performs a first-order update on the dual variable. With prediction $\hat{\mu}$, a natural initialization of the dual variable is to set $\mu_1=\hat{\mu}$, i.e., the algorithms start by assuming the prediction is accurate. Then, the algorithms use adaptive step sizes $\eta_t$ in mirror descent steps depending on the arrival model and the prediction's behavior. 

\subsection{Algorithm for the Stochastic Arrival Model}
Let $\hat{\mu}$ be a prediction with accuracy parameter $a$, i.e., $||\hat{\mu}-\mu^*||_1\leq\kappa T^{-a}$. By \cref{prop:stochastic lower}, no algorithm can achieve regret better than $O(\max\{T^{\frac{1}{2}-a},1\})$ even if $a$ is known. As a comparison, we can show that the optimal fixed step size for the Mirror Descent Algorithm 
is $\eta\sim T^{-\frac{1-a}{2}}$ using similar method as the proof of \cref{prop:balseiro}, which incurs $O(\max\{T^{\frac{1-a}{2}},1\})$ regret. Therefore, mirror descent with fixed step size is sub-optimal even if the prediction quality is known. This suggest us to use adaptive step sizes. The step size we use is drawn from \cite{carmon2022making} in their work in parameter-free optimization. It follows the line of work in the more general online learning problem of parameter-free regret minimization \citep{chaudhuri2009parameter,cutkosky2019artificial,cutkosky2017online,cutkosky2018black,mhammedi2020lipschitz,streeter2012no}.

We list some notations used in Algorithm \ref{alg:stochastic}, which follow notations in \cite{carmon2022making}. Given initial dual solution $\mu_1$ and step-size $\eta$:
\begin{itemize}
	\item[(a)] Let $x_t(\mu_1,\eta)$ and $\mu_{t}(\mu_1,\eta)$ be the action we take and the the dual variable we get after $t-1$ iterations of the Mirror Descent Algorithm with initial dual solution $\mu_1$, step-size $\eta$, and the same requests as the requests that Algorithm \ref{alg:stochastic} received so far;
	\item[(b)] Define $\theta_t(\mu_1,\eta):=\max_{s\leq t}||\mu_1-\mu_s(\mu_1,\eta)||_1$ to be the maximum $\ell_1$-distance from any updated dual variable used in the Mirror Descent Algorithm before time $t$ to the initial dual variable;
	\item[(c)] Define $\Phi_t(\mu_1,\eta):=\sum_{s=1}^{t}||-g_s(x_s(\mu_1,\eta))+\rho||^2_\infty$ to be the running sum of squared $\ell_\infty$-norms of the dual functions' sub-gradients.
\end{itemize}

Algorithm \ref{alg:stochastic}, which we call the \textit{Stochastic Arrival Algorithm} (SA), initializes the dual variable at the prediction $\hat{\mu}$ and updates the dual variable at each time period through mirror descent with fine-tuned step sizes. A high-level intuition behind the choices of step sizes is that, it is well-known \citep{francesco2020ICML} that the hindsight (asymptotically) optimal step size is $\eta$ that satisfies $$\eta=\frac{||\mu_1-\mu^*||_1}{\sqrt{\Phi_T(\mu_1,\eta)}}.$$ Because $||\mu_1-\mu^*||_1$ and $\Phi_T(\mu_1,\eta)$ are unknown a priori, at each time period $t$ we use $\theta_t(\mu_1,\eta)$ as an approximation of $||\mu_1-\mu^*||_1$ and use $\Phi_t(\mu_1,\eta)$ as an approximation of $\Phi_t(\mu_1,\eta)$, and these approximations can be proven to be accurate. Then we use bisection to find an approximate solution of the implicit function $$\eta_t=\frac{\theta_t(\mu_1,\eta_t)}{\sqrt{\alpha\Phi_t(\mu_1,\eta_t)+\beta}},$$ where $\alpha,\beta$ are damping parameters. For a more detailed explanation, see \cite{carmon2022making}. Note that the Stochastic Arrival Algorithm does not need to know the accuracy parameter $a$.

\begin{algorithm}
	\small
	\caption{Stochastic Arrival Algorithm (SA)}
	\vspace{0.5em}
	\SetKwFunction{RFB}{Root Finding Bisection}
	\SetKwFunction{MD}{Mirror Descent}
	{\bf Inputs:} Prediction $\hat{\mu}$, total time periods $T$, initial resources $G_1=\rho T$, reference function $h(\cdot): \mathbb{R}^m \rightarrow \mathbb{R}$, and initial step-size $\eta_1$\;
	Initialize $\mu_1\gets\hat{\mu}$\;
	\For{$t$ from $1$ to $T$}{
		Receive request $\left(r_t, g_t, \mathcal{X}_t\right)$\;
		Make the primal decision $x_t$ and update the remaining resources $G_t$:
		\centerline{$
			x_t\in\arg \max_{ x\in \mathcal{X}_t,g_t(x)\leq G_t}\left\{r_t(x)-\mu_t^{\top} g_t(x)\right\}$;}
		\centerline{$G_{t+1}\gets G_t-g_t\left(x_t\right).$}
		
		Obtain a sub-gradient of the dual function:
		\centerline{$\phi_t\gets -g_t\left(x_t\right)+\rho.$} 
		
		Update the dual variable by mirror descent:
		\centerline{$\mu_{t+1}\gets \arg \min _{\mu \in \mathbb{R}_{+}^m} \phi_t^{\top} \mu+\frac{1}{\eta_t} V_h\left(\mu, \mu_t\right) $,}
		where $V_h(x, y):=h(x)-h(y)-\nabla h(y)^{\top}(x-y)$ is the Bregman divergence.
		\vspace{0.5em}
		
		Tune the step size:\hfill
		
		\For{$k=2,4,8,16,\dots$}{
			$t_k\gets \lfloor t/2k\rfloor$\;
			$\alpha^{(k)}\gets 32^2C_t^{(k)}$, $\beta^{(k)}\gets (32C_t^{(k)}(\bar{g}+\bar{\rho}))^2$ where $C_t^{(k)}:=2k+\log\left(60T\log^2(6t)\right)$\;
			\If{$\RFB(\eta_t,2^{2^k}\eta_t;t_k,\alpha^{(k)},\beta^{(k)})<\infty$}{$\eta_{t+1}\gets \RFB(\eta_{t},2^{2^k}\eta_t;t_k,\alpha^{(k)},\beta^{(k)})$.}
		}
	}
	\SetKwProg{Fn}{Function}{:}{}
	\Fn{\RFB{$\eta_{\textup{lo}},\eta_{\textup{hi}};t',\alpha,\beta$}}{$\psi(\cdot):=\eta\to\theta_{t'}(\hat{\mu},\eta)/\sqrt{\alpha\Phi_{t'}(\hat{\mu},\eta)+\beta}$\;  
		\textbf{if} $\eta_{\textup{hi}}\leq\psi(\eta_{\textup{hi}})$ \textbf{then return} $\infty$\;
		\textbf{if} $\eta_{\textup{lo}}>\psi(\eta_{\textup{lo}})$ \textbf{then return} $\eta_{\textup{lo}}$\;
		\While{$\eta_{\textup{hi}}>2\eta_{\textup{lo}}$}{$\eta_{\textup{mid}}\gets \sqrt{\eta_{\textup{hi}}\eta_{\textup{lo}}}$\;
			\textbf{if} $\eta_{\textup{mid}}\leq\psi(\eta_{\textup{mid}})$ \textbf{then} $\eta_{\textup{lo}}\gets\eta_{\textup{mid}}$ \textbf{else} $\eta_{\textup{hi}}\gets\eta_{\textup{mid}}$.}
		\textbf{if} $\theta_{t'}(\hat{\mu},\eta_{\textup{hi}})\leq\theta_{t'}(\hat{\mu},\eta_{\textup{lo}})\frac{\psi(\eta_{\textup{hi}})}{\eta_{\textup{hi}}}$ \textbf{then return} $\eta_{\textup{hi}}$ \textbf{else return} $\eta_{\textup{lo}}$.
	}
	\textbf{End Function}
	\vspace{0.5em}
	\label{alg:stochastic}
\end{algorithm}

\begin{proposition}
	\label{thm:stochastic}
	Consider the Stochastic Arrival Algorithm \textup{(SA)} under the stochastic arrival model. Given a prediction $\hat{\mu}$ with (unknown) accuracy parameter $a$, it holds that: $$\textup{Regret}(\textup{SA})=\tilde{O}(\max\{T^{\frac{1}{2}-a},1\}).$$
\end{proposition}

The proof of \cref{thm:stochastic} can be found in Appendix \ref{Appendix B}. By \cref{prop:stochastic lower}, the Stochastic Arrival Algorithm achieves optimal worst-cast regret up to logarithm factors.

\subsection{Algorithm for the Adversarial Arrival Model}
Different from the stochastic arrival model, under the adversarial arrival model it is impossible to achieve sub-linear worst-case regret. Instead, we directly compare the reward obtained by our algorithm to the maximum reward of two natural benchmark algorithms: the Mirror Descent Algorithm (which is optimal when the prediction quality is low) and the Prediction Algorithm (which is optimal when the prediction quality is high). That is, for an algorithm $\text{ALG}$, we will analyze the following quantity:
$$\limsup_{T\to\infty}\sup_{\gamma\in\mathcal{S}^T}\left\{\frac{1}{T}\left(\max\left\{\frac{1}{\alpha^*}\text{OPT}(\gamma),\text{PRD}(\gamma)\right\} - R(\text{ALG}\mid \gamma)\right)\right\}.$$ 

We give Algorithm \ref{alg:adversarial} for the adversarial arrival model, which we call the \textit{Adversarial Arrival Algorithm} (AA). It performs mirror descent with  fixed step size $\eta\sim \epsilon(T)/T$.

\begin{algorithm}
	\small
	\caption{Adversarial Arrival Algorithm (AA)}
	\vspace{0.5em}
	{\bf Inputs:} Prediction $\hat{\mu}$, total time periods $T$, initial resources $G_1=\rho T$, reference function $h(\cdot): \mathbb{R}^m \rightarrow \mathbb{R}$, upper bound function $\epsilon(T)=o(1)$, and step-size $\eta\sim \epsilon(T)/T$\;
	Initialize $\mu_1\gets\hat{\mu}$\;
	\For{$t$ from $1$ to $T$}{
		Receive request $\left(r_t, g_t, \mathcal{X}_t\right)$\;
		Make the primal decision $x_t$ and update the remaining resources $G_t$:
		\centerline{$
			x_t\in\arg \max_{ x\in \mathcal{X}_t,g_t(x)\leq G_t}\left\{r_t(x)-\mu_t^{\top} g_t(x)\right\}$;}
		\centerline{$G_{t+1}\gets G_t-g_t\left(x_t\right).$}
		
		Obtain a sub-gradient of the dual function:
		\centerline{$\phi_t\gets -g_t\left(x_t\right)+\rho.$}
		
		Update the dual variable by mirror descent:
		\centerline{$\mu_{t+1}\gets \arg \min _{\mu \in \mathbb{R}_{+}^m} \phi_t^{\top} \mu+\frac{1}{\eta} V_h\left(\mu, \mu_t\right)$,}
		where $V_h(x, y):=h(x)-h(y)-\nabla h(y)^{\top}(x-y)$ is the Bregman divergence.}
	\label{alg:adversarial}
\end{algorithm}

\begin{proposition}
	\label{thm:adversarial}
	Assume that Assumptions \ref{assumption: Adversarial Arrival Model(a)} and \ref{assumption: Adversarial Arrival Model(b)} hold. Consider the Adversarial Arrival Algorithm \textup{(AA)} under the adversarial arrival model. Given a prediction $\hat{\mu}$ with (unknown) accuracy parameter $a$, it holds that: $$\limsup_{T\to\infty}\sup_{\gamma\in\mathcal{S}^T}\left\{\frac{1}{T}\left( \max\left\{\frac{1}{\alpha^*}\textup{OPT}(\gamma),\textup{PRD}( \gamma)\right\} - R(\textup{AA}\mid \gamma)\right)\right\}\leq 0.$$
\end{proposition}

The proof of \cref{thm:adversarial} can be found in Appendix \ref{Appendix B}. \cref{thm:adversarial} states that the Adversarial Arrival Algorithm achieves the maximum of the two benchmark algorithms without knowing the prediction quality. \cref{thm:adversarial} is tight by \cref{prop:adv lower}.

\subsection{Main Algorithm: Detection of Nonstationarity}

With the Stochastic Arrival Algorithm and the Adversarial Arrival Algorithm, we are ready to present our main algorithm, which merges the two algorithms and works for both arrival models. 

\begin{algorithm}
	\small
	\caption{Main Algorithm}
	\vspace{0.5em}
	{\bf Inputs:} Prediction $\hat{\mu}$, total time periods $T$, initial rewards $R_1=0$, initial resources $G_1=0$, reference function $h(\cdot): \mathbb{R}^m \rightarrow \mathbb{R}$, upper bound function $\epsilon(T)=o(1)$, constant $L$ which is specified in \cref{eqn:thm 3 constant} in Appendix \ref{Appendix C}, constant $0<\delta\leq 1$, and initial step-size $\eta_1$\;
	\For{$t$ from $1$ to $T$}{
		Receive request $\gamma_t=\left(r_t, g_t, \mathcal{X}_t\right)$\;
		
		\If{$t=k\lfloor \delta T\rfloor+1$ for some $k=0,\dots,\lceil 1/\delta\rceil-1$\vspace{0.3em}}{
			\uIf{$R_{t}+L\log(T)\sqrt{T}\geq\textup{OPT}_{t-1}(\gamma_1,\dots,\gamma_{t-1})$}
			{
				\vspace{0.3em} Release resources for the next $\lfloor \delta T\rfloor$ time periods: $G_{t}\gets G_t+\lfloor\delta T\rfloor\rho$\;  
				Take action $x_t$ given by the Stochastic Arrival Algorithm with the following inputs: total time periods $\lfloor \delta T\rfloor$, initial dual variable $\mu_{k\lfloor\delta T\rfloor+1}=\hat{\mu}$, initial resources $G_{k\lfloor\delta T\rfloor+1}$, reference function $h(\cdot)$, and initial step-size $\eta_1$\;
				Update resources: $G_{t+1}\gets G_{t}-g_t(x_t)$\;
				Update rewards: $R_{t+1}\gets R_{t}+r_t(x_t)$.
			}\Else{\bf break}
		}
		
		\vspace{0.5em} Release all resources: $G_t\gets G_t+\rho(T-t+1)$\;
		Use the Adversarial Arrival Algorithm with initial dual solution $\hat{\mu}$, remaining resources $G_t$, and step size $\eta\sim \epsilon(T)/T$.}
	\label{alg:main}
\end{algorithm}

The Main Algorithm starts by assuming the arrivals are stochastic and using the Stochastic Arrival Algorithm, while carefully releasing the resources to prevent the algorithm from over-consuming resources. Meanwhile, the algorithm keeps monitoring on the arrivals so far to see if the arrivals are truly stochastic. Intuitively, if the arrivals are stochastic, the reward we obtained so far should be similar to the optimal offline reward of drawn from the underlying distribution. If our reward is significantly lower than the reward of the optimal offline reward, we have evidence that with high probability the arrivals are not stochastic, and for the remaining time periods we switch to the Adversarial Arrival Algorithm. Note that if the arrivals are adversarial but they are relatively stationary (by our definition), the algorithm would not be able to detect that the arrivals are adversarial. However, because they 
are stationary, the Stochastic Arrival Algorithm would work well on these arrivals, so it is fine to not switch to the Adversarial Arrival Algorithm.

\setcounter{theorem}{0}
\begin{theorem}[Upper Bound]\label{thm:main}
	Consider the Main Algorithm \textup{(MainALG)}. Assume that Assumptions \ref{assumption: Adversarial Arrival Model(a)} and \ref{assumption: Adversarial Arrival Model(b)} hold. Given a prediction $\hat{\mu}$ with (unknown) accuracy parameter $a$ and given $0<\delta\leq 1$, it holds that: 
	\begin{enumerate}
		\item[(1)] If the arrivals are stochastic,
		$$\textup{Regret}(\textup{MainALG})=\tilde{O}(\max\{T^{\frac{1}{2}-a},1\});$$
		\item[(2)] If the arrivals are adversarial and $(\lambda,\delta)$-stationary, $$\limsup_{T\to\infty}\sup_{\gamma\in\mathcal{S}^T}\left\{\frac{1}{T}\left( (1-\lambda)\max\left\{\frac{1}{\alpha^*}\textup{OPT}(\gamma),\textup{PRD}( \gamma)\right\} - R(\textup{MainALG}\mid \gamma)\right)\right\}\leq \delta\bar{r}.$$
	\end{enumerate}
\end{theorem}

The proof of \cref{thm:main} appears in Appendix \ref{Appendix C}. \cref{thm:main} is tight by \cref{prop:lower bound of full}.

\section{Experiments}\label{Section experiments}

Finally, we describe a set of experiments, one on synthetic data and one on real data, that we performed to empirically evaluate our algorithm. 
The main takeaway is that our algorithm's performance is robust with respect to the quality of the predictions. Specifically, the rewards it obtains is consistently ``close'' to the higher of the rewards obtained by the Mirror Descent Algorithm (which recall is worst-case optimal without predictions) and the Prediction Algorithm.

In both sets of experiments, we used sequential assortment optimization as the application, acting as an online retailer making in-cart recommendations: when each customer checks out, we recommend a subset of products (with certain fixed cardinality). For each product, there is a customer-specific probability that the product will be purchased if recommended (we will discuss the way we obtained these probabilities later), hence generating some revenue. In the Online Resource Allocation with Predictions framework, an action is a choice of a subset of products to recommend, the resources are the inventories of the products, and the reward of an action is the expected profit obtained by recommending the chosen set of products. Our objective is to maximize the total reward. 

Each {\em instance} of our experiment represented a single problem with certain fixed initial inventories and an (online) arrival sequence. For each instance, we were given a prediction on the shadow price of each product. The predictions were generated with various qualities across instances.

Each instance yields three total rewards: one incurred by our algorithm (the ``Main Algorithm''), and two incurred by the benchmark algorithms (the Mirror Descent Algorithm and the Prediction Algorithm). The primary performance metric we report is a form of optimality gap. For some instance $I$, let $R(\textup{PRD}\mid I)$, $R(\textup{MDA}\mid I)$, and $R(\textup{MainALG}\mid I)$ represent the reward generated from instance $I$ using the Prediction algorithm, the Mirror Descent Algorithm, and the Main Algorithm, respectively. Then we can define the \textit{optimality gap} (GAP) of our algorithm as
\[
\mathrm{GAP}(I)=\frac{R(\textup{MainALG}\mid I)-\min\{R(\textup{PRD}\mid I),R(\textup{MDA}\mid I)\}}{\max\{R(\textup{PRD}\mid I),R(\textup{MDA}\mid I)\}-\min\{R(\textup{PRD}\mid I),R(\textup{MDA}\mid I)\}}
\]
If we think of the Main Algorithm as trying to achieve the maximum of the rewards obtained by the two benchmark algorithms, then GAP measures the rewards that the Main Algorithm obtains compared to this maximum, normalized so that $\mathrm{GAP}=1$ implies that the maximum has been obtained, and $\mathrm{GAP}=0$ implies that the minimum of the two rewards was obtained.\footnote{GAP may technically be outside of $[0,1]$.} As a baseline, randomly choosing between the Mirror Descent Algorithm and the Prediction Algorithm yields $\mathrm{GAP}=0.5$.

\subsection{Synthetic Experiment}
We began with a set of smaller, synthetic experiments with 25 products over 1000 time periods, and the task of recommending 2 products at a time. We assumed customers belong to one of 25 ``types.''
The process we used to randomly generate the product prices, the initial inventory levels, and the customer type-specific purchase probabilities, is described in Appendix \ref{Appendix Experiment}. Each time period corresponds to a single arriving customer (drawn uniformly from the 25 types), or no arrival. We used three types of arrival sequences, with the probability of a customer arrival changing over time: {\bf stationary} with a fixed arrival probability of 0.7, {\bf nonstationary} with an arrival probability linearly increasing from 0.4 to 1.0, and {\bf adversarial} with an arrival probability of 0.0 during the first 300 periods and 1.0 afterward. We randomly generated predictions of varying qualities by computing the optimal shadow prices, and adding mean-zero Gaussian noise with standard deviations of 500 ({\bf bad}), 5 ({\bf good}), and 0 ({\bf perfect}). 


\begin{table}[h!]
	\small
	\centering
	\begin{tabular}{@{}lrR{3em}R{3em}cR{3em}R{3em}cR{3em}R{3em}@{}} \toprule
		Median         & $1-\mathrm{CDF}(0.5)$        & \multicolumn{2}{r}{\bf Stochastic} && \multicolumn{2}{r}{\bf Nonstationary} && \multicolumn{2}{r}{\bf Adversarial} \\ \midrule
		\multicolumn{2}{l}{\bf Perfect Predictions} & 0.81            & 0.63           && 0.83              & 0.64             && 0.65            & 0.56            \\
		\multicolumn{2}{l}{\bf Good Predictions}    & 0.77            & 0.62           && 0.81              & 0.64             && 0.64            & 0.56            \\
		\multicolumn{2}{l}{\bf Bad Predictions}     & 0.54            & 0.52           && 0.45              & 0.49             && 0.22            & 0.40\\ \bottomrule 
	\end{tabular}
	
	\vspace{1em}
	
	\begin{tabular}{@{}lrR{3em}R{3em}cR{3em}R{3em}cR{3em}R{3em}@{}} \toprule
		Median         & $1-\mathrm{CDF}(0.5)$        & \multicolumn{2}{r}{\bf Stochastic} && \multicolumn{2}{r}{\bf Nonstationary} && \multicolumn{2}{r}{\bf Adversarial} \\ \midrule
		\multicolumn{2}{l}{\bf Perfect Predictions} & 0.71            & 0.65           && 0.72              & 0.66             && 0.64            & 0.60            \\
		\multicolumn{2}{l}{\bf Good Predictions}    & 0.67            & 0.63           && 0.72              & 0.67             && 0.52            & 0.54            \\
		\multicolumn{2}{l}{\bf Bad Predictions}     & 0.58            & 0.55           && 0.49              & 0.49             && 0.36            & 0.40\\ \bottomrule 
	\end{tabular}
	\vspace{1em}
	\caption{Summary of synthetic experiments. For each of three levels of prediction quality (the rows), and each of three generative arrival models (the columns), two summary statistics are reported over a random ensemble of instances: (left) the median GAP, and (right) the proportion of instances for which the GAP was at least 0.5. (Top) Results over all instances. (Bottom) Results over instances for which the rewards of the Mirror Descent and Prediction algorithms differ by at least 25\%.}
	\label{tab:synthetic-experiment}
\end{table}

The results are summarized in the top half of Table \ref{tab:synthetic-experiment}, which for nine random ensembles of instances (depending on arrival model and prediction quality) reports both the median GAP, and the proportion of instances for which the GAP was at least 0.5 (for both values, higher is better). Recall that no algorithm can be expected to achieve high GAP values (say above 0.5) for all nine ensembles simultaneously. We see that our algorithm generally performs better with higher prediction quality and higher stationarity. Now one issue with GAP as a performance metric is that it can be quite erratic when the Mirror Descent and Prediction algorithms generate similar rewards (as their difference is the denominator in GAP), and these are arguably the instances in which GAP ``matters'' the least from a practical standpoint. Thus, in the bottom half of Table \ref{tab:synthetic-experiment}, in which  instances for which the two rewards are within 25\% of each other have been removed, we see better overall performance. 

\subsection{H\&M Experiment}
We used a dataset from H\&M (a fast-fashion clothing retailer), which contains the online transactions of 105,542 products from 2018 to 2020, along with product features. Because most products have zero or few transactions in two years, we selected the products with the top 5000 number of total transactions for our experiment, which includes 13,697,790 transactions. Our task was to recommend three products.



Each instance runs across three month's transactions from the data. The time horizon for each instance was the maximum number of transactions per day (103,473) multiplied by the total number of days (90), so that each day contained 103,473 time periods, each having zero or one arriving customer. 
To estimate customer-specific purchase probabilities,
we used (customer, transaction time, product $A$, product $B$, price of product $A$, price of product $B$)-tuples and trained a random forest algorithm with the corresponding features of this tuple (a 209-dimensional vector after encoding) to estimate the probability that the customer, who brought product $A$ at that certain time period with the certain price, would also buy product $B$ if recommended. 


To generate predictions, we used three popular forecasting methods ranging from classical algorithms to the state-of-the art:
\begin{itemize}
	\item {\bf Prophet:} A recent algorithm developed by Facebook \citep{taylor2018forecasting} based on a (piecewise-linear) trend and seasonality decomposition, known to work well in practice with minimal tuning. Tuning parameters: software default.
	\item {\bf Exponential Smoothing (Holt Winters):} A classic algorithm based on a (linear) trend and seasonality decomposition, known for its simplicity and robust performance. It is frequently used as a benchmark in forecasting competitions \citep{makridakis2000m3}. Tuning parameters: seasonality of length.
	\item {\bf ARIMA:} Another classic algorithm that is rich enough to model a wide class of nonstationary time-series. Tuning parameters: $(p,q,r)$.
	
\end{itemize}
These experiments were run on a N2D Series machine on Google Cloud's Compute Engine, with 224 vCPUs and 896GBs of memory. The total compute time was around 140 hours.


The results are summarized in \cref{fig:gap}, which contains histograms of the GAPs across an ensemble of 100 instances (for varying three-month periods in the data), separately for each forecasting algorithm. The average GAP is 0.68 on instances with Prophet forecasts, 0.58 on instances with ARIMA forecasts, and 0.53 on instances with Exponential Smoothing forecasts. Because the average GAPs are large with all three forecasting methods, our algorithm performs close to the better one of the Prediction algorithm and the Mirror Descent Algorithm, showcasing its robustness to the unknown prediction accuracy.

\begin{figure}[h!]
	\captionsetup[subfigure]{font=scriptsize,labelfont=scriptsize}
	\centering
	\begin{subfigure}{0.32\textwidth}
		\centering
		\includegraphics[width=\linewidth]{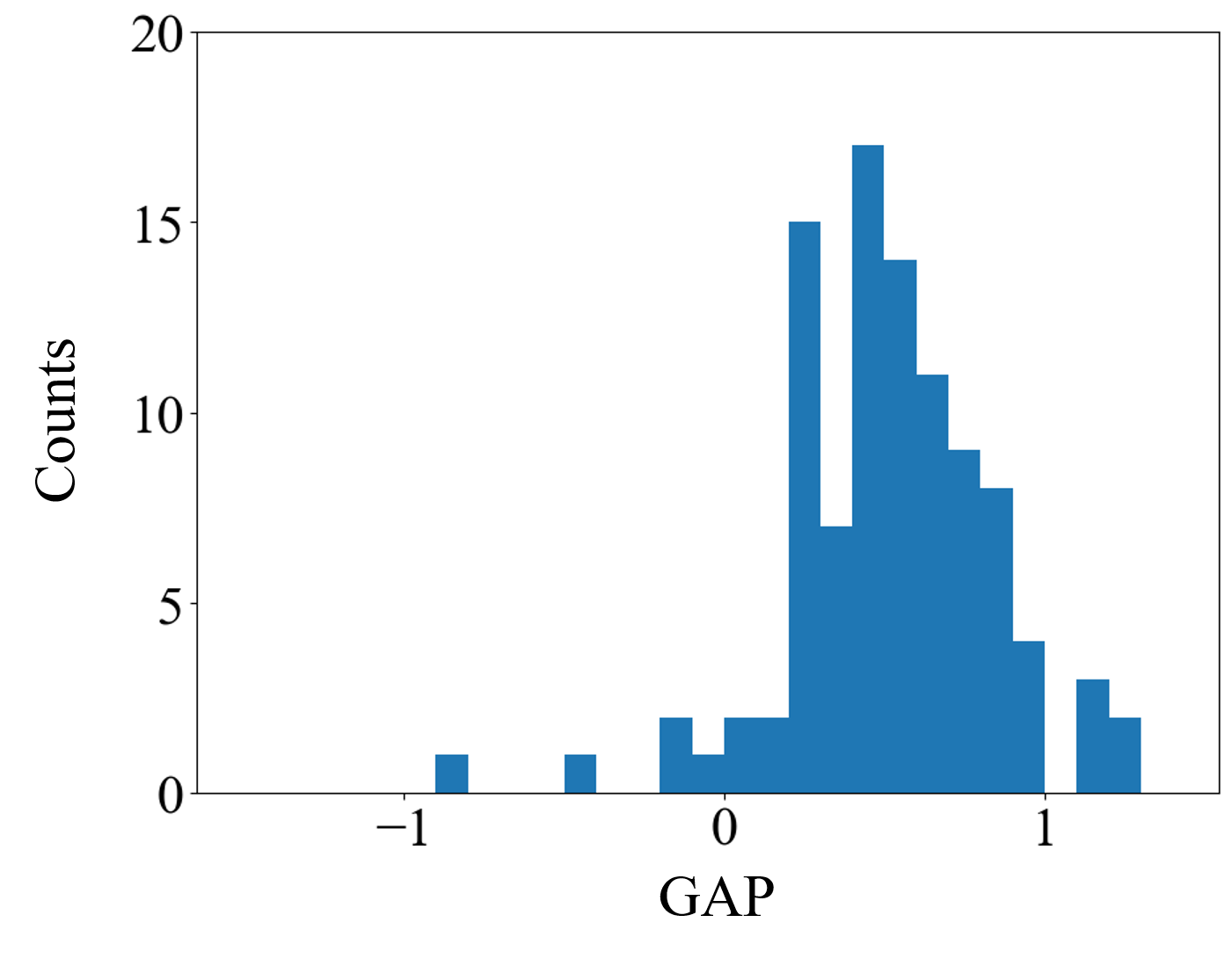}
		\caption{GAPs of Prophet}
	\end{subfigure}
	\begin{subfigure}{0.32\textwidth}
		\centering
		\includegraphics[width=\linewidth]{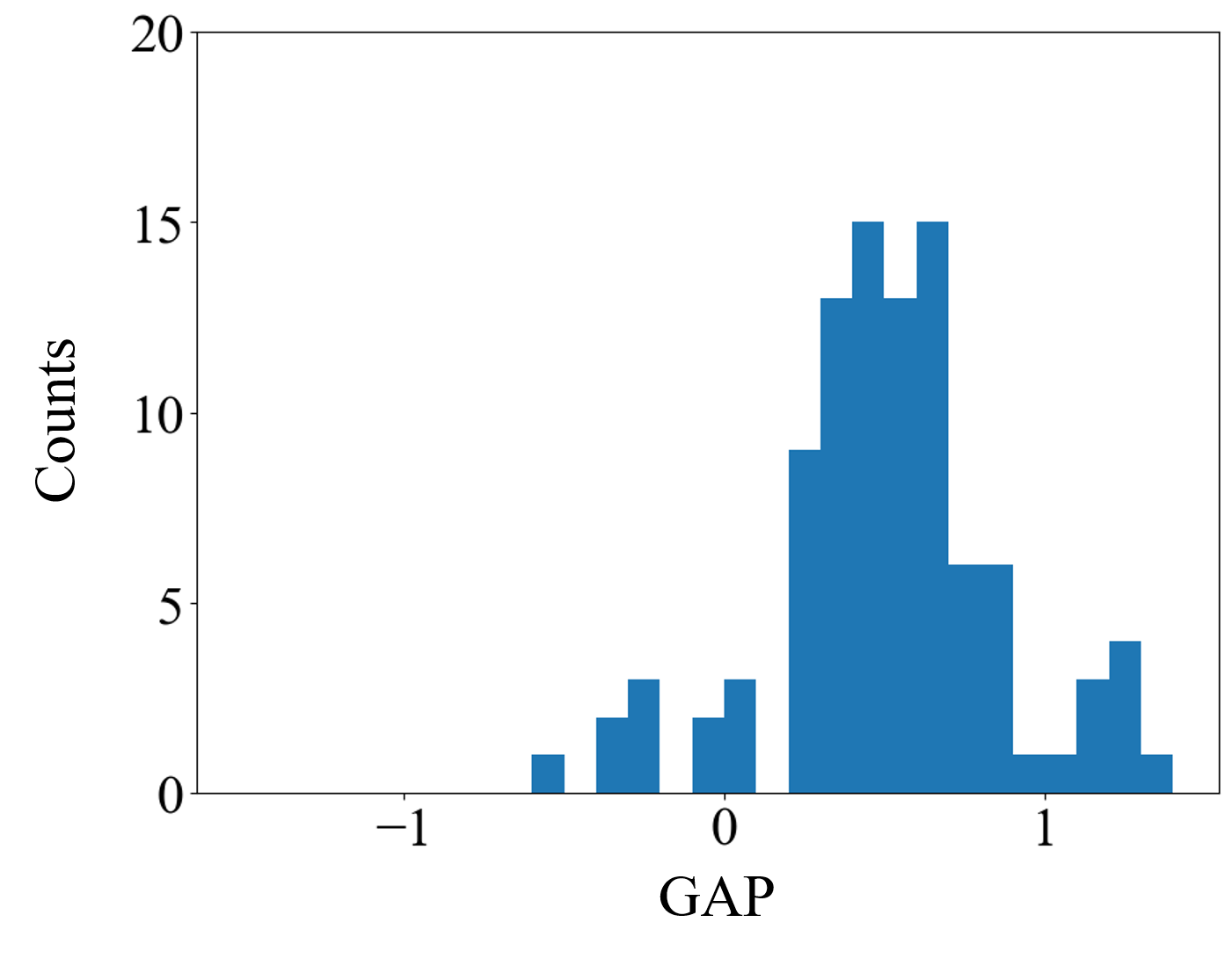}
		\caption{GAPs of ARIMA }
	\end{subfigure}
	\begin{subfigure}{0.32\textwidth}
		\centering
		\includegraphics[width=\linewidth]{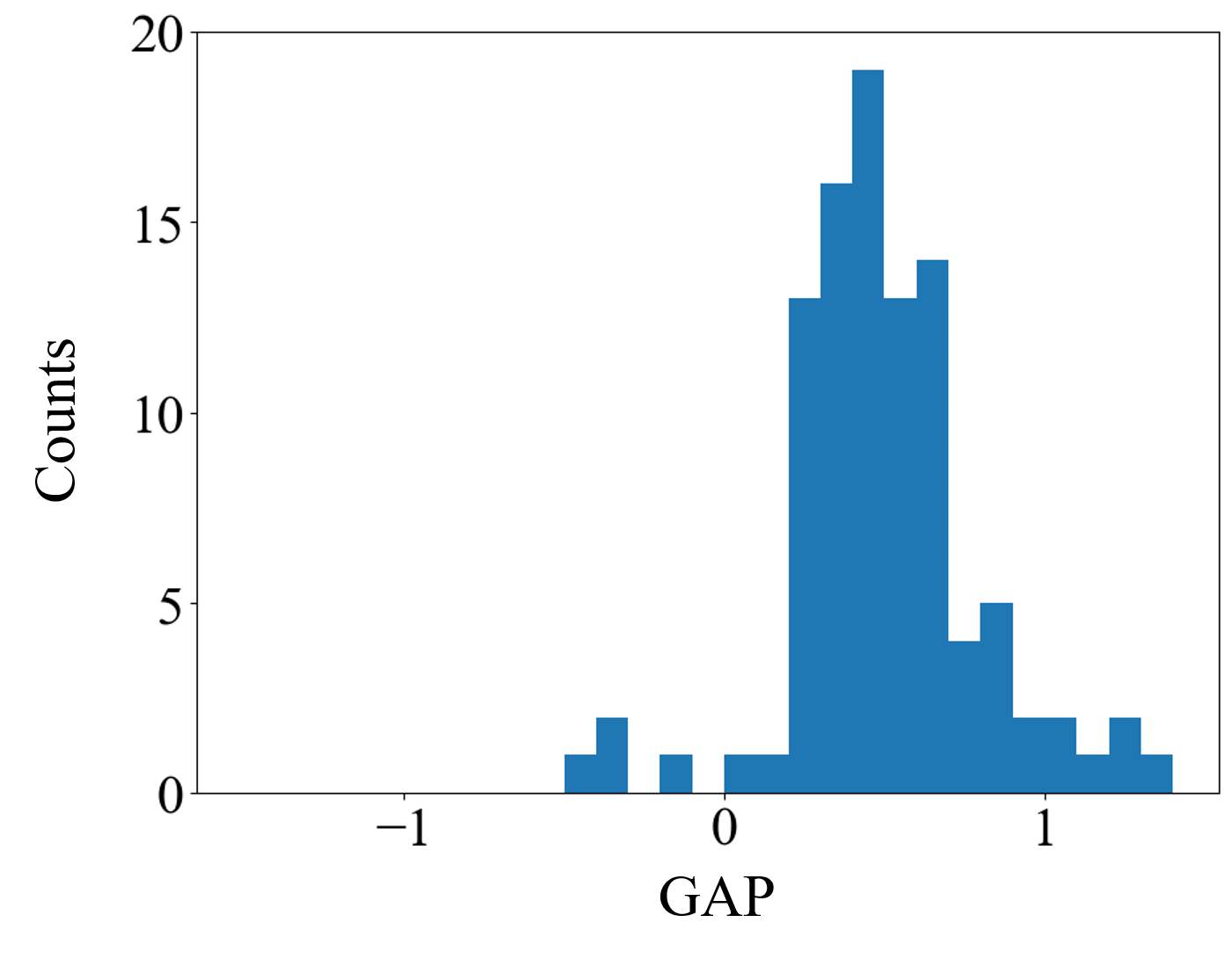}
		\caption{GAPs of Exponential Smoothing}
	\end{subfigure}  \caption{Histograms of GAPs with different forecasting methods, each containing 100 instances.}
	\label{fig:gap}
\end{figure}

\section{Conclusion}\label{Section conclusion}
In this paper, we proposed a new model incorporating predictions into the Online Resource Allocation Problem. With the notion of prediction, we first gave two separate algorithms for the stochastic arrival model and the adversarial arrival model. Under the stochastic arrival model the respective algorithm achieves nearly optimal minimax worst-cast regret, and under the adversarial arrival model the respective algorithm obtains nearly optimal amount of reward. Both algorithms do not need to know the prediction quality beforehand. We then built on these two algorithms and proposed our main algorithm, which achieves the above-mentioned performance under respective arrival models without knowing the underlying arrival model and the prediction quality a priori. The main idea behind our algorithm is to first assume the arrivals are stochastic, while keeps running hypothesis tests on the arrivals to see if the arrivals are adversarial instead.

\bibliographystyle{apalike}

\newpage

\appendix
	
\section{Preliminary Results and Proofs in \cref{Section formulation}}\label{Appendix A}

We first state two structural results regarding the duality of the offline problem.
\begin{lemma}[Weak Duality]\label{lemma:weak duality}
	$\textup{OPT}(\gamma)\leq D(\mu\mid\gamma)$ for every $\mu\in\mathbb{R}^m_+$.
\end{lemma}
\begin{lemma}[Duality Gap]\label{lemma:duality gap}
	$\min_{\mu\in\mathbb{R}^m_+}D(\mu\mid\gamma)\leq \textup{OPT}(\gamma)+(m+1)\overline{r}$.
\end{lemma}
\cref{lemma:weak duality} is the standard weak duality result. \cref{lemma:duality gap} states that, even without any convexity assumptions, the duality gap of our problem is upper bounded by a constant that is independent from the time horizon $T$. This can be shown via Shapley-Folkman Theorem (see Proposition 5.26 of \cite{bertsekas2014constrained} for a detailed explanation).
\proof{Proof of \cref{lemma:weak duality}.}
This proof appears in \cite{balseiro2023best}. We include it for the sake of completeness.
It holds for any $\mu \in \mathbb{R}_{+}^m$ that
$$
\begin{aligned}
	\text{OPT}(\gamma) & =\left[\arraycolsep=1.5pt\def\arraystretch{1.5}\begin{array}{cc}
		\max _{x_t \in \mathcal{X}_t} & \sum_{t=1}^T r_t\left(x_t\right) \\
		\text { s.t. } & \sum_{t=1}^T g_t\left(x_t\right) \leq \rho T 
	\end{array}\right] \\
	& \leq \max _{x_t \in \mathcal{X}}\left\{\sum_{t=1}^T r_t\left(x_t\right)+\mu^{\top} \rho T -\mu^{\top} \sum_{t=1}^T g_t\left(x_t\right)\right\} \\
	& =\sum_{t=1}^T r_t^*(\mu)+T \mu^{\top} \rho T \\
	& =D(\mu\mid \gamma),
\end{aligned}
$$
where the first inequality is because we relax the constraint $\sum_{t=1}^T g_t\left(x_t\right) \leq\rho T $ and $\mu \geq 0$, and the last equality utilizes the definition of $r^*(\cdot)$.
\endproof

\vspace{1em}
\proof{Proof of \cref{corollary:perfect dual}.}
Let $\text{conv}(\mathcal{X}_t)\subset\mathbb{R}^d_+$ denote the convex hull of $\mathcal{X}_t$. For each $t$, define the function $\tilde{r}_t:\text{conv}(\mathcal{X}_t)\to\mathbb{R}_+$ by $$\tilde{r}_t(\tilde{x})=\sup \left\{\sum_{k=1}^{d+1} \alpha^k r_t\left(x^k\right) \mid \tilde{x}=\sum_{k=1}^{d+1} \alpha^k x^k, x^k \in \mathcal{X}_t, \sum_{k=1}^{d+1} \alpha^k=1, \alpha^k \geq 0\right\}\qquad \forall \tilde{x}\in \text{conv}(\mathcal{X}_t).$$
$\tilde{r}_t$ is concave regardless of whether $r_t$ is concave or not, and it can be viewed as a ``concavification'' of $r_t$ on $\text{conv}(\mathcal{X}_t)$. Similarly, for each $t$, define the function $\tilde{g}_t:\text{conv}(\mathcal{X}_t)\to\mathbb{R}_+^m$ by $$\tilde{g}_t(\tilde{x})=\inf \left\{\sum_{k=1}^{d+1} \alpha^k g_t\left(x^k\right) \mid \tilde{x}=\sum_{k=1}^{d+1} \alpha^k x^k, x^k \in \mathcal{X}_t, \sum_{k=1}^{d+1} \alpha^k=1, \alpha^k \geq 0\right\}\qquad \forall \tilde{x}\in \text{conv}(\mathcal{X}_t).$$ $\tilde{g}_t$ is convex regardless of whether $g_t$ is convex or not. 

Let $(\text{P})$ denote the optimization problem in \cref{eqn:offline opt}. Consider the following convex relaxation $(\tilde{\text{P}})$ of the optimization problem in \cref{eqn:offline opt}:
\begin{equation*}
	\max_{x_t\in\text{conv}(\mathcal{X}_t)} \sum_{t=1}^T \tilde{r}_t(\tilde{x}_t) \quad\text{ s.t. } \quad \sum_{t=1}^T \tilde{g}_t(\tilde{x}_t)\leq \rho T;\end{equation*} 
and its Lagrangian dual problem $(\tilde{\text{D}})$:
\begin{equation*}
	\min_{\mu\in\mathbb{R}_+^m}\sum_{t=1}^T\tilde{r}^*_t(\tilde{x}_t)+\mu^\top\rho T \quad \text{ where }\quad r^*_t(\mu)=\sup_{x\in\text{conv}(\mathcal{X}_t)}\{r_t(\tilde{x})-\mu^\top \tilde{g}_t(\tilde{x})\}.
\end{equation*} 
Because $0\in\text{conv}(\mathcal{X}_t)$, $\tilde{g}_t(0)=0$ for all $t$, and $\rho>0$, $(\tilde{\text{P}})$ satisfies Slater's condition. Therefore by strong duality $\sup(\tilde{\text{P}})=\inf(\tilde{\text{D}})$. By an application of the Sharpley-Folkman Theorem (\cite{bertsekas2014constrained}, Proposition 5.26), there exists an optimal solution $\{\tilde{x}^*_t\}_{t=1}^T$ of $(\tilde{\text{P}})$ with the following property: let $I\subset [T]$ be the set of time periods where $\tilde{x}^*_t\notin\mathcal{X}_t$ for $t\in I$, then $|I|\leq m+1$ and $\sum_{t\in [T]\setminus I}g_t(\tilde{x}^*_t)\leq \rho T$.  Let $\tilde{\mu}^*$ be the optimal dual variable of $(\tilde{\text{D}})$ that induces $\{\tilde{x}^*_t\}_{t=1}^T$, and let $\{x^{\tilde{\mu}^*}_t\}_{t=1}^T$ be the actions induced by $\tilde{\mu}^*$ in the original primal $(\text{P})$. We prove that $\sum_{t=1}^Tr_t(x^{\tilde{\mu}^*}_t)\geq \sum_{t=1}^Tr_t(\tilde{x}^*_t)-(\bar{g}/\underline{g}+1)(m+1)\bar{r}$.

Let $S\subset [T]$ be the set of time periods such that $x^{\tilde{\mu}^*}_s\neq\tilde{x}^*_s$ for $s\in S$, and let $J=S\setminus I$. Then $J$ is the set of time periods where the resource constraint becomes active when choosing the action induced by $\tilde{\mu}^*$. Because $|I|\leq m+1$ and $x^{\tilde{\mu}^*}_t=\tilde{x}^*_t$ for $t\in [T]\setminus S$, $\sum_{t\in J}g_t(\tilde{x}^*_t)-\sum_{t\in J}g_t(x^{\tilde{\mu}^*}_t)\leq(m+1)\bar{g}$. Therefore $|\{t\in J:\tilde{x}^*_t\neq 0\}|\leq (m+1)\bar{g}/\underline{g}$, so $\sum_{t\in J}r_t(\tilde{x}^*_t)-\sum_{t\in J}r_t(x^{\tilde{\mu}^*}_t)\leq (m+1)\bar{r}\bar{g}/\underline{g}$. Further, we also have $\sum_{t\in I}r_t(\tilde{x}^*_t)-\sum_{t\in I}r_t(x^{\tilde{\mu}^*}_t)\leq (m+1)\bar{r}$ and $x^{\tilde{\mu}^*}_t=\tilde{x}^*_t$ for $t\in [T]\setminus S$. These together gives $\sum_{t=1}^Tr_t(x^{\tilde{\mu}^*}_t)\geq \sum_{t=1}^Tr_t(\tilde{x}^*_t)-(\bar{g}/\underline{g}+1)(m+1)\bar{r}$. Finally, since $(\tilde{\text{P}})$ is a relaxation of $(\text{P})$, $\sup(\tilde{\text{P}})=\sum_{t=1}^Tr_t(\tilde{x}^*_t)\geq\text{OPT}(\gamma)$, so we have $\max_{\mu\in\mathbb{R}_+^m}R(\textup{GRD}_{\mu}\mid\gamma)\geq \sum_{t=1}^Tr_t(x^{\tilde{\mu}^*}_t)\geq \text{OPT}(\gamma)-(\bar{g}/\underline{g}+1)(m+1)\bar{r}$. 
\endproof

We make the following observation, which follows since \cref{corollary:perfect dual} shows there exists a perfect dual variable for every arrival sequence.
\begin{observation}\label{obs:stationary}
	If an arrival sequence $\gamma$ is \textit{$(\lambda,\delta)$-stationary}, then \begin{equation*}
		\min_{k=1,\dots,\lfloor\frac{1}{\delta}\rfloor}(\text{OPT}(\gamma_{1:k\delta T})+\text{OPT}(\gamma_{k\delta T+1:T}))\geq (1-\lambda)\text{OPT}(\gamma).
	\end{equation*}
\end{observation}

\vspace{1em}
\proof{Proof of \cref{prop:iid stationary}.}
Fix any $\delta>0$. If $R(\textup{GRD}_{\mu}\mid\gamma)=o(T)$, then since the total amount of available resources $\rho T$ scales linearly in $T$ and every single action consumes constants amount of resources, the Dual-Adjusted Greedy Algorithm with dual variable $\mu$ never depletes resources. Therefore $R(\textup{GRD}_{\mu}\mid\gamma_{1:k\delta T})+R(\textup{GRD}_{\mu}\mid\gamma_{k\delta T+1:T})= R(\textup{GRD}_{\mu}\mid\gamma)$ for every $1\leq k\leq\lfloor\frac{1}{\delta}\rfloor$, which shows $\gamma$ is $(\delta,\lambda)$-stationary for every $\lambda>0$. From now on we assume that $R(\textup{GRD}_{\mu}\mid\gamma)=\Theta(T)$.

For any time periods $s,t$ and any amount of resources $\rho'T\in\mathbb{R}^m_+$, we use $R(\textup{GRD}_{\mu}\mid\gamma_{s:t},\rho'T)$ to denote the amount of reward obtained by the Dual-Adjusted Greedy Algorithm with dual variable $\mu$ on the $\gamma_{s:t}$-subproblem \textit{with available amount of resources $\rho' T$}. 

Fix an integer $1\leq k\leq\lfloor\frac{1}{\delta}\rfloor$. If $k\delta T= o(T)$, then \begin{eqnarray*}
	R(\textup{GRD}_{\mu}\mid\gamma_{k\delta T+1:T})&\geq &R(\textup{GRD}_{\mu}\mid\gamma_{k\delta T+1:T},\rho T)-k\delta T(m\bar{g}\bar{r}/\underline{g})\\
	&\geq &R(\textup{GRD}_{\mu}\mid\gamma)-k\delta T\bar{r}-k\delta T(m\bar{g}\bar{r}/\underline{g}),
\end{eqnarray*}
where the first inequality follows since any algorithm can consume at most $m\bar{g}$ amount of resources in $\ell_1$-norm in a single time period, which can be translated to at most $m\bar{g}\bar{r}/\underline{g}$ amount of reward; the second inequality follows since any algorithm can obtain at most $k\delta T\bar{r}$ amount of rewards in the first $k\delta T\bar{r}$ time periods. Since $k\delta T\in o(T)$, this shows $R(\textup{GRD}_{\mu}\mid\gamma_{1:k\delta T})+R(\textup{GRD}_{\mu}\mid\gamma_{k\delta T+1:T}))\geq R(\textup{GRD}_{\mu}\mid\gamma)-o(T)$. Similarly, we can also show this if $T-k\delta T= o(T)$.

Suppose $k\delta T, T-k\delta T=\Theta(T) $. Let $\rho'T\leq \rho T$ be the amount of resources that is consumed by the Dual-Adjusted Greedy Algorithm with dual variable $\mu$ and arrivals $\gamma$, then $R(\textup{GRD}_{\mu}\mid\gamma)=R(\textup{GRD}_{\mu}\mid\gamma,\rho'T)$ and $R(\textup{GRD}_{\mu}\mid\gamma_{1:k\delta T})=R(\textup{GRD}_{\mu}\mid\gamma_{1:k\delta T},k\delta\rho' T)$. Condition on $\gamma\sim\mathcal{P}^T$, for each time period $t$, let $X_t\in\mathbb{R}^m_+$ be the  amount of resources consumed at time period $t$. Then $X_t$'s are independent with $\mathbb{E}[X_t]=\rho'$ and $0\leq ||X_t||_{\infty}\leq \bar{g}$. By Hoeffding's inequality we have
\begin{eqnarray*}
	\mathbb{P}_{\gamma\sim\mathcal{P}^T}\left(\sum_{t=1}^{k\delta T} (X_t)_{m'}-k\delta\rho'_{m'} T\geq \bar{g}^2\sqrt{k\delta T\log T}\right)\leq \frac{1}{T^2},
\end{eqnarray*} where $(X_t)_{m'}$ denotes the $m'$-th coordinate of $X_t$. So by union bound we have \begin{equation}\label{eqn:prop2,1}
	\mathbb{P}_{\gamma\sim\mathcal{P}^T}\left(\sum_{t=1}^{k\delta T} X_t-k\delta\rho'T\geq \bar{g}^2\sqrt{k\delta T\log T}\right)\leq \frac{m}{T^2}.
\end{equation} Therefore \begin{eqnarray*}
	&\quad&k\delta\mathbb{E}_{\gamma\sim\mathcal{P}^T}\left[R\left(\textup{GRD}_{\mu}\mid\gamma\right)\right]-\mathbb{E}_{\gamma\sim\mathcal{P}^T}\left[R(\textup{GRD}_{\mu}\mid\gamma_{1:k\delta T})\right]\\
	&=&\mathbb{E}_{\gamma\sim\mathcal{P}^T}\left[R\left(\textup{GRD}_{\mu}\mid\gamma_{1:k\delta T},\sum_{t=1}^{k\delta T} X_t\right)\right]-\mathbb{E}_{\gamma\sim\mathcal{P}^T}\left[R(\textup{GRD}_{\mu}\mid\gamma_{1:k\delta T},k\delta\rho' T)\right]\\
	&\leq&\left(1-\frac{m}{T^2}\right)\mathbb{E}_{\gamma\sim\mathcal{P}^T}\left[R\left(\textup{GRD}_{\mu}\mid\gamma_{1:k\delta T},k\delta\rho' T+\bar{g}^2\sqrt{k\delta T\log T})-R(\textup{GRD}_{\mu}\mid\gamma_{1:k\delta T},k\delta\rho' T)\right)\right]+\frac{m}{T^2}\cdot \bar{r}k\delta T\\
	&\leq&\bar{r}\bar{g}^2\sqrt{k\delta T\log T})/\underline{g}+ m\bar{r}k\delta/T\\
	&=&o(T),
\end{eqnarray*}
where the first inequality follows by conditioning on two cases given by Eq. (\ref{eqn:prop2,1}) and noticing that any algorithm can obtain at most $k\delta T\bar{r}$ amount of rewards in the first $k\delta T\bar{r}$ time periods, and the second inequality follows since any algorithm can consume at most $m\bar{g}$ amount of resources in $\ell_1$-norm in a single time period, which can be translated to at most $m\bar{g}\bar{r}/\underline{g}$ amount of reward. Similarly, we also have \begin{equation*}
	(1-k\delta)\mathbb{E}_{\gamma\sim\mathcal{P}^T}\left[R\left(\textup{GRD}_{\mu}\mid\gamma\right)\right]-\mathbb{E}_{\gamma\sim\mathcal{P}^T}\left[R(\textup{GRD}_{\mu}\mid\gamma_{k\delta T+1:T})\right]=o(T).
\end{equation*}
Hence $$\mathbb{E}_{\gamma\sim\mathcal{P}^T}\left[R\left(\textup{GRD}_{\mu}\mid\gamma\right)-(R(\textup{GRD}_{\mu}\mid\gamma_{1:k\delta T})+R(\textup{GRD}_{\mu}\mid\gamma_{k\delta T+1:T}))\right]=o(T).$$

Note that $R\left(\textup{GRD}_{\mu}\mid\gamma\right)-(R(\textup{GRD}_{\mu}\mid\gamma_{1:k\delta T})+R(\textup{GRD}_{\mu}\mid\gamma_{k\delta T+1:T}))$ is a function from $\mathcal{P}^T$ to $\mathbb{R}$ such that each $\gamma_t$ is drawn independently. Moreover, it satisfies the bounded differences property with bound $m\bar{g}\bar{r}/\underline{g}$ since any algorithm can consume at most $m\bar{g}$ amount of resources in $\ell_1$-norm in a single time period, which can be translated to at most $m\bar{g}\bar{r}/\underline{g}$ amount of reward. Therefore by McDiarmid's inequality we have \begin{eqnarray*}
	&\mathbb{P}_{\gamma\sim\mathcal{P}^T}&\left[R\left(\textup{GRD}_{\mu}\mid\gamma\right)-(R(\textup{GRD}_{\mu}\mid\gamma_{1:k\delta T})+R(\textup{GRD}_{\mu}\mid\gamma_{k\delta T+1:T}))\geq o(T)+(m\bar{g}\bar{r}/\underline{g})\sqrt{2T\log T/3}\right]\\&\leq&\frac{1}{T^3}.
\end{eqnarray*}
This implies $R\left(\textup{GRD}_{\mu}\mid\gamma\right)-(R(\textup{GRD}_{\mu}\mid\gamma_{1:k\delta T})+R(\textup{GRD}_{\mu}\mid\gamma_{k\delta T+1:T}))=o(T)$ with probability at least $1-\frac{1}{T^3}$ for any $1\leq k\delta T\leq T$. Since $k\delta T$ can take at most $T$ values, by union bound $\min_{k=1,\dots,\lfloor\frac{1}{\delta}\rfloor}(R(\textup{GRD}_{\mu}\mid\gamma_{1:k\delta T})+R(\textup{GRD}_{\mu}\mid\gamma_{k\delta T+1:T}))\geq R(\textup{GRD}_{\mu}\mid\gamma)-o(T)$ with probability at least $1-\frac{1}{T^2}$. Because $R(\textup{GRD}_{\mu}\mid\gamma)=\Theta(T)$, this implies $\gamma$ is $(\delta,\lambda)$-stationary for every $\lambda>0$ with probability at least $1-\frac{1}{T^2}$.
\endproof

\section{Details in Section \ref{Section main results}.1}\label{Appendix A.2.5}

For completeness, we discuss the \textit{Mirror Descent Algorithm} (MDA) given in \cite{balseiro2023best}. 

\begin{algorithm}\small
	\caption{Mirror Descent Algorithm (MDA)}
	\vspace{0.5em}
	{\bf Inputs:} Initial dual solution $\mu_1$, total time periods $T$, initial resources $G_1=\rho T$, reference function $h(\cdot): \mathbb{R}^m \rightarrow \mathbb{R}$, and step-size $\eta$\;
	\For{$t$ from $1$ to $T$}{
		Receive request $\left(r_t, g_t, \mathcal{X}_t\right)$\;
		Make the primal decision $x_t$ and update the remaining resources $G_t$:
		\centerline{$
			x_t\in\arg \max_{ x\in \mathcal{X}_t,g_t(x)\leq G_t}\left\{r_t(x)-\mu_t^{\top} g_t(x)\right\}$;}
		\centerline{$G_{t+1}\gets G_t-g_t\left(x_t\right).$}
		
		Obtain a sub-gradient of the dual function:
		\centerline{$\phi_t\gets -g_t\left(x_t\right)+\rho.$}
		
		Update the dual variable by mirror descent:
		\begin{equation}
			\mu_{t+1}\gets \arg \min _{\mu \in \mathbb{R}_{+}^m} \phi_t^{\top} \mu+\frac{1}{\eta} V_h\left(\mu, \mu_t\right),
			\label{eqn:MDA update dual}
		\end{equation}
		where $V_h(x, y):=h(x)-h(y)-\nabla h(y)^{\top}(x-y)$ is the Bregman divergence.}
	\label{alg:without predictions}
\end{algorithm}

The Mirror Descent Algorithm takes an initial dual variable, a step-size, and a {\em reference function} as inputs. At each time period $t$, 
the algorithm takes the action induced by the current dual variable $\mu_t$, and performs a first-order update on the dual variable. 
For the updating step, note we can write the dual function in \cref{eqn:dual} as $D(\mu\mid\gamma):=\sum_{t=1}^TD_t(\mu\mid\gamma)$ where the $t$-th term of the dual function is given by $D_t(\mu\mid\gamma)= r_t^*(\mu)+\mu^\top \rho $. Then it follows that $\phi_t:=-g_t\left(x_t\right)+\rho$ is a sub-gradient of $D_t(\mu\mid\gamma)$ at $\mu_t$ under our assumptions by Danskin's Theorem (see, e.g., Proposition B.25 in \cite{bertsekas1997nonlinear}), and the algorithm uses $\phi_t$ to update the dual variable by performing a mirror descent step in \cref{eqn:MDA update dual} with step-size $\eta$ and reference function $h(\cdot)$. Intuitively, the Mirror Descent Algorithm tries to find dual variables via gradient information such that these dual variables induce actions with good primal performances. For more  on mirror descent algorithms in general, see \cite{nemirovskij1983problem,beck2003mirror,hazan2016introduction,lu2018relatively}.

We state the standard assumptions on choosing the reference function $h(\cdot)$ for mirror descent algorithms \citep{beck2003mirror,bubeck2015convex,lu2018relatively,lu2019relative}. These assumptions are applicable to all algorithms in our paper.
\begin{itemize}
	\item[(a)] $h(\mu)$ is either differentiable or essentially smooth \citep{bauschke2001essential} and Lipschitz in $\mathbb{R}^m_+$;
	\item[(b)] $h(\mu)$ is $\sigma$-strongly convex with respect to the $\ell_1$-norm in $\mathbb{R}^m_+$, i.e., $h(\mu_1)\geq h(\mu_2)+\nabla h(\mu_2)^\top(\mu_1-\mu_2)+\frac{\sigma}{2}||\mu_1-\mu_2||_1^2$ for all $\mu_1,\mu_2\in\mathbb{R}^m_+$.
	\item[(c)] $h(\mu)$ coordinately-wise separable, i.e., $h(\mu)=\sum_{j=1}^m h_j(\mu_j)$ where $h_j:\mathbb{R}_+\to\mathbb{R}$ is an univariate function. Moreover, for every resource $j$ the function $h_j$ is $\sigma'$-strongly convex with respect to the $\ell_1$-norm over $[0,\mu_j^{\max}]$ where $\mu_j^{\max}:=\bar{r}/\rho_j+1$.
\end{itemize}

\section{Proofs in Section \ref{Section main results}.3}\label{Appendix A.5}
\proof{Proof of \cref{prop:lower bound of full}.}
Let $c$ be a positive integer such that $c>\max\{\frac{K}{\lambda-\lambda'},\frac{1-\lambda}{\delta}\}$. Set $\rho=1$, $\alpha^*=1/c\delta$, and $\bar{r}=(\alpha^*-1)/(1-\lambda-1/\alpha^*)$, then by our choice of $c$ we have $\lambda<1-1/\alpha^*$ and $\bar{r}>\alpha^*$. Consider two different types of arrivals $\gamma^1=(r^1,g^1,\mathcal{X})$ and $\gamma^2=(r^2,g^2,\mathcal{X})$, where $\mathcal{X}=\{0,1\}$ (one can think of this as \{reject, accept\}). Set $r^1(1)=1, g^1(1)=1, r^2(1)=\bar{r}$, and $g^2(1)=\alpha^*$.
Let $\hat{\mu}=1+1/\log(T)$ be the prediction, then following $\hat{\mu}$ means taking action 0 for $\gamma^1$ and taking action 1 for $\gamma^2$. Because $\bar{r}/\alpha^*>\hat{\mu}$, one can verify that Assumptions \ref{assumption: Adversarial Arrival Model(a)} and \ref{assumption: Adversarial Arrival Model(b)} are satisfied.

Consider the following two instances.
\begin{itemize}
	\item Instance one: the arrivals are stochastic where the state space is $\mathcal{S}=\{\gamma^1\}$, i.e., $\gamma_t=\gamma^1$ for every $t=1,\dots, T$. In this instance the optimum is to take action 1 for all arrivals. Note that following $\hat{\mu}$ would take action 0 for all arrivals, which means the prediction has bad quality.
	
	\item Instance two: the arrivals are adversarial where  $\gamma_t=\gamma^1$ for $t=1,\dots, \frac{\alpha^*-1}{\alpha^*}T$ and $\gamma_t=\gamma^2$ for $t=\frac{\alpha^*-1}{\alpha^*}T+1,\dots, T$. In this instance the optimum is to take action 0 for $\gamma^1$ and take action 1 for $\gamma^2$. We have $\textup{PRD}( \gamma)=\textup{OPT}(\gamma)=\frac{\bar{r}}{\alpha^*}T$, which means the prediction is perfect. Moreover, since $\delta=1/c\alpha^*$, one can verify that \begin{eqnarray*}
		&\quad&\min_{k=1,\dots,\lfloor\frac{1}{\delta}\rfloor}\textup{OPT}(\gamma_{1:k\delta T})+\textup{OPT}(\gamma_{k\delta T+1:T})\\
		&=&\textup{OPT}(\gamma_{1:\frac{\alpha^*-1}{\alpha^*}T})+\textup{OPT}(\gamma_{\frac{\alpha^*-1}{\alpha^*}T+1:T})\\
		&=&\frac{\alpha^*-1}{\alpha^*}T+\frac{\bar{r}}{{\alpha^*}^2}T.
	\end{eqnarray*}
	Then we get \begin{eqnarray*}
		&\quad&1-\left(\textup{OPT}(\gamma_{1:\frac{\alpha^*-1}{\alpha^*}T})+\textup{OPT}(\gamma_{\frac{\alpha^*-1}{\alpha^*}T+1:T})\right)/\textup{OPT}(\gamma)\\
		&=& 1-\left(\frac{\alpha^*-1}{\alpha^*}T+\frac{\bar{r}}{{\alpha^*}^2}T\right)/\frac{\bar{r}}{\alpha^*}T\\
		&=&1-\frac{1}{\alpha^*}-\frac{\alpha^*-1}{\bar{r}}\\
		&=&\lambda,
	\end{eqnarray*} where the last equality follows by our choice of $\bar{r}=(\alpha^*-1)/(1-\lambda-1/\alpha^*)$. Therefore $\gamma$ is $\lambda$-nonstationary with respect to $\delta$.
\end{itemize}

Note that no algorithm can distinguish instance one and instance two before time period $t=\frac{\alpha^*-1}{\alpha^*}T+1$. For any algorithm, assume in instance one it satisfies $\textup{Regret}(\textup{ALG})=o(T)$, then since the optimum is to take action 1 for all arrivals, at time period $t=\frac{\alpha^*-1}{\alpha^*}T+1$ the amount of resources left is at most $\frac{1}{\alpha^*}T-o(T)$. Therefore in instance two the algorithm can take action 1 for at most $\frac{1}{{\alpha^*}^2}T+o(1)$ time periods, so the total rewards gained in instance two satisfies $R(\textup{ALG}\mid \gamma)=\frac{\alpha^*-1}{\alpha^*}T+\frac{\bar{r}}{{\alpha^*}^2}T+o(T)$. Because $\textup{PRD}( \gamma)=\frac{\bar{r}}{\alpha^*}T$, in instance two we have \begin{eqnarray*}
	&\quad&\limsup_{T\to\infty}\left\{\frac{1}{T}\left( (1-\lambda')\max\left\{\frac{1}{\alpha^*}\textup{OPT}(\gamma),\textup{PRD}( \gamma)\right\} - R(\textup{ALG}\mid \gamma)\right)\right\}\\
	&=&\limsup_{T\to\infty}\left\{\frac{1}{T}\left( (\lambda-\lambda')\frac{\bar{r}}{\alpha^*}T - o(T)\right)\right\}\\
	&=& (\lambda-\lambda')\frac{\bar{r}}{\alpha^*}\\
	&>&K\delta\bar{r},
\end{eqnarray*} where the last inequality follows since $c>K/(\lambda-\lambda')$ and $\delta=\alpha^*/c$. 
\endproof

\section{Proofs in Section \ref{Section main algorithm}.1 and \ref{Section main algorithm}.2}\label{Appendix B}
\proof{Proof of \cref{thm:stochastic}.} The proof technique is similar to the proof of Theorem 1 in \cite{balseiro2023best}, which we largely borrow. We break down the proof in three steps.
\\
\text{ }\text{ }\text{ }  \textbf{Step 1 (Primal performance.)} First, we define the stopping time $\tau_A$ of Algorithm \ref{alg:stochastic} as the first time less than $T$ that there exists resource $j$ such that $\sum_{t=1}^{\tau_A}\left(g_t\left(x_t\right)\right)_j+\bar{g} \geq \rho_j T$. Notice that $\tau_A$ is a random variable, and moreover, we will not violate the resource constraints before the stopping time $\tau_A$. We here study the primal-dual gap until the stopping-time $\tau_A$. Notice that before the stopping time $\tau_A$, Algorithm \ref{alg:stochastic} performs the mirror descent steps on the dual function with fine-tuned step sizes.

Consider a time $t \leq \tau_A$ so that actions are not constrained by resources. Then the algorithm takes the action $x_t \in \arg \max _{x \in \mathcal{X}_t}\left\{r_t(x)-\right.$ $\left.\mu_t^{\top} g_t(x)\right\}$, so we have that
$$
r_t\left(x_t\right)=r_t^*\left(\mu_t\right)+\mu_t^{\top} g_t\left(x_t\right).
$$
Let $\bar{D}(\mu \mid \mathcal{P})=\frac{1}{T} \mathbb{E}_{\gamma \sim \mathcal{P}^T}[D(\mu \mid \gamma)]=\mathbb{E}_{(r, g,\mathcal{X}) \sim \mathcal{P}}\left[r^*\left(\mu_t\right)\right]+\mu_t^{\top} \rho$ be the expected dual objective at $\mu$ when requests are drawn i.i.d. from $\mathcal{P} \in \Delta(\mathcal{S})$. Let $\xi_t=\left\{\gamma_0, \ldots, \gamma_t\right\}$ and $\sigma\left(\xi_t\right)$ be the sigma-algebra generated by $\xi_t$. Adding the last two equations and taking expectations conditional on $\sigma\left(\xi_{t-1}\right)$ we obtain, because $\mu_t \in \sigma\left(\xi_{t-1}\right)$ and $\left(r_t, g_t,\mathcal{X}_t\right) \sim \mathcal{P}$, that
\begin{eqnarray}\label{eqn:proof of thm 1}
	\mathbb{E}\left[r_t\left(x_t\right) \mid \sigma\left(\xi_{t-1}\right)\right] & =&\mathbb{E}_{(r, g,\mathcal{X}) \sim \mathcal{P}}\left[f^*\left(\mu_t\right)\right]+\mu_t^{\top} \rho+\mu_t^{\top}\left(\mathbb{E}\left[g_t\left(x_t\right) \mid \sigma\left(\xi_{t-1}\right)\right]-\rho\right) \nonumber\\
	& =&\bar{D}\left(\mu_t \mid \mathcal{P}\right)-\mathbb{E}\left[\mu_t^{\top}\left(\rho-g_t\left(x_t\right)\right) \mid \sigma\left(\xi_{t-1}\right)\right]
\end{eqnarray}
where the second equality follows the definition of the dual function.

Consider the process $Z_t=\sum_{s=1}^t \mu_s^{\top}\left(a_s-b_s\left(x_s\right)\right)-\mathbb{E}\left[\mu_s^{\top}\left(a_s-b_s\left(x_s\right)\right) \mid \sigma\left(\xi_{s-1}\right)\right]$, which is martingale with respect to $\xi_t$ (i.e., $Z_t \in \sigma\left(\xi_t\right)$ and $\mathbb{E}\left[Z_{t+1} \mid \sigma\left(\xi_t\right)\right]=Z_t$ ). Since $\tau_A$ is a stopping time with respect to $\xi_t$ and $\tau_A$ is bounded, the Optional Stopping Theorem implies that $\mathbb{E}\left[Z_{\tau_A}\right]=0$. Therefore,
$$
\mathbb{E}\left[\sum_{t=1}^{\tau_A} \mu_t^{\top}\left(\rho-g_t\left(x_t\right)\right)\right]=\mathbb{E}\left[\sum_{t=1}^{\tau_A} \mathbb{E}\left[\mu_t^{\top}\left(\rho-g_t\left(x_t\right)\right) \mid \sigma\left(\xi_{t-1}\right)\right]\right] .
$$
Using a similar martingale argument for $f_t\left(x_t\right)$ and summing Eq. (\ref{eqn:proof of thm 1}) from $t=1, \ldots, \tau_A$ we obtain that
\begin{eqnarray}
	\label{eqn:slackness}
	\mathbb{E}\left[\sum_{t=1}^{\tau_A} r_t\left(x_t\right)\right] & =&\mathbb{E}\left[\sum_{t=1}^{\tau_A} \bar{D}\left(\mu_t \mid \mathcal{P}\right)\right]-\mathbb{E}\left[\sum_{t=1}^{\tau_A} \mu_t^{\top}\left(\rho-g_t\left(x_t\right)\right)\right] \nonumber\\
	& \geq &\mathbb{E}\left[\tau_A \bar{D}\left(\bar{\mu}_{\tau_A} \mid \mathcal{P}\right)\right]-\mathbb{E}\left[\sum_{t=1}^{\tau_A} \mu_t^{\top}\left(\rho-g_t\left(x_t\right)\right)\right] .
\end{eqnarray}
where the inequality follows from denoting $\bar{\mu}_{\tau_A}=\frac{1}{\tau_A} \sum_{t=1}^{\tau_A} \mu_t$ to be the average dual variable and using that the dual function is convex.

\textbf{Step 2 (Complementary slackness).} Consider the sequence of functions $w_t(\mu)=\mu^{\top}\left(\rho-g_t\left(x_t\right)\right)$, which capture the complementary slackness at time $t$. The sub-gradients are given by $\nabla_\mu w_t(\mu)=\rho-g_t\left(x_t\right)$, which are bounded as follows $\left\|\nabla_\mu w_t(\mu)\right\|_{\infty} \leq\left\|g_t\left(x_t\right)\right\|_{\infty}+\|\rho\|_{\infty} \leq \bar{g}+\bar{\rho}$. Therefore, Algorithm \ref{alg:stochastic} applies online mirror descent to the sequence of functions $w_t(\mu)$ with the fine-tuned step sizes. To analyze the performance, we use the following lemma from \cite{carmon2022making}.

\begin{lemma}[Theorem 4 in \cite{carmon2022making}]\label{lemma:carmon}
	Under the assumptions and notations of our paper, the online mirror descent in Algorithm \ref{alg:stochastic} with the proposed step sizes satisfies, with probability at least $1-\frac{1}{T}$, that $$\sum_{t=1}^{\tau_A}\left( w_t\left(\mu_t\right)-w_t(\mu^*) \right)\leq CT^{\frac{1}{2}}||\mu_1-\mu^*||_1\cdot\textup{polylog}(T)\footnote{\textup{Polylog}(T) hides logarithmic terms in $T$. For explicit expressions see \cite{carmon2022making}.}$$ where $C>0$ is some constant.
\end{lemma}

Because $||\mu_1-\mu^*||_1=||\hat{\mu}-\mu^*||_1\leq\kappa T^{-a}$, Lemma \ref{lemma:carmon} states that $\sum_{t=1}^{\tau_A} w_t\left(\mu_t\right)-w_t(\mu^*) \leq \kappa CT^{\frac{1}{2}-a}\cdot\textup{polylog}(T)$ with probability at least $1-\frac{1}{T}$.

\textbf{Step 3 (Putting it all together).} For any $\mathcal{P} \in \Delta(\mathcal{S})$ and $\tau_A \in[0, T]$ we have that
\begin{eqnarray}\label{eqn:all together}
	\mathbb{E}_{\gamma \sim \mathcal{P}^T}[\mathrm{OPT}(\gamma)]=\frac{\tau_A}{T} \mathbb{E}_{\gamma \sim \mathcal{P}^T}[\mathrm{OPT}(\gamma)]+\frac{T-\tau_A}{T} \mathbb{E}_{\gamma \sim \mathcal{P}^T}[\mathrm{OPT}(\gamma)] \leq \tau_A \bar{D}\left(\bar{\mu}_{\tau_A} \mid \mathcal{P}\right)+\left(T-\tau_A\right) \bar{r},
\end{eqnarray}
where the inequality uses \cref{lemma:weak duality} and the fact that $\text{OPT}(\gamma) \leq \bar{r}T $. Therefore, with probability at least $1-\frac{1}{T}$,
\begin{eqnarray}\label{eqn:clubsuit}
	\text{Regret}(\text{SA} \mid \mathcal{P}) & =&\mathbb{E}_{\gamma \sim \mathcal{P}^T}[\text{OPT}(\gamma)-R(\text{SA} \mid \gamma)]\nonumber \\&\leq& \mathbb{E}_{\gamma \sim \mathcal{P}^T}\left[\text{OPT}(\gamma)-\sum_{t=1}^{\tau_A} r_t\left(x_t\right)\right]\nonumber\\& \leq& \mathbb{E}_{\gamma \sim \mathcal{P}^T}\left[\text{OPT}(\gamma)-\tau_A D\left(\bar{\mu}_{\tau_A} \mid \mathcal{P}\right)+\sum_{t=1}^{\tau_A}w_t\left(\mu_t\right)\right]\nonumber \\
	& \leq& \mathbb{E}_{\gamma \sim \mathcal{P}^T}\left[\textup{OPT}(\gamma)-\tau_A D\left(\bar{\mu}_{\tau_A} \mid \mathcal{P}\right)+\sum_{t=1}^{\tau_A} w_t(\mu^*)+CT^{\frac{1}{2}-a}\cdot\textup{polylog}(T)\right] \nonumber\\&\leq &\mathbb{E}_{\gamma \sim \mathcal{P}^T}\underbrace{\left[\left(T-\tau_A\right) \cdot \bar{r}+\sum_{t=1}^{\tau_A} w_t(\mu^*)+CT^{\frac{1}{2}-a}\cdot\textup{polylog}(T)\right]}_{\clubsuit}
\end{eqnarray}
where the first inequality follows from using that $\tau_A \leq T$ together with $r_t(\cdot) \geq 0$ to drop all requests after $\tau_A$; the second is from Eq. (\ref{eqn:slackness}); the third follows from \cref{lemma:carmon}; and the last from Eq. (\ref{eqn:all together}).

Note that $\sum_{t=1}^{\tau_A} w_t(\mu^*)\leq \sum_{t=1}^{\tau_A} w_t(\mu)$ for every $\mu\in \mathbb{R}_{+}^m$. We now discuss the choice of $\mu \in \mathbb{R}_{+}^m$ in order to upper bound $\sum_{t=1}^{\tau_A} w_t(\mu^*)$. If $\tau_A=T$, then set $\mu=0$ to obtain that $\clubsuit \leq CT^{\frac{1}{2}-a}\cdot\textup{polylog}(T)$. If $\tau_A<T$, then there exists a resource $j \in[m]$ such that $\sum_{t=1}^{\tau_A}\left(g_t\left(x_t\right)\right)_j+\bar{g} \geq \rho_j T $. Set $\mu=\left(\bar{r} / \rho_j\right) e_j$ with $e_j$ being the $j$-th unit vector. This yields
\begin{eqnarray*}
	&\quad&\sum_{t=1}^{\tau_A} w_t(\mu^*)\leq\sum_{t=1}^{\tau_A} w_t(\mu)=\sum_{t=1}^{\tau_A} \mu^{\top}\left(\rho-g_t\left(x_t\right)\right)\\ &=&\frac{\bar{r}}{\rho_j} \sum_{t=1}^{\tau_A}\left(\rho_j-\left(g_t\left(x_t\right)\right)_j\right) \leq \frac{\bar{r}}{\rho_j}\left(\tau_A \rho_j-\rho_j T +\bar{g}\right)=\frac{\bar{r}}{\rho_j} \bar{g}-\bar{r}\left(T-\tau_A\right),
\end{eqnarray*}
where the inequality follows because of the definition of the stopping time $\tau_A$. Therefore, using that $\rho_j \geq \underline{\rho}$ for every resource $j \in[m]$, we have
$$
\clubsuit \leq \frac{\bar{r}\bar{g}}{\underline{\rho}}+CT^{\frac{1}{2}-a}\cdot\textup{polylog}(T).
$$
Therefore $\text{Regret}(\text{SA})\leq \frac{\bar{r}\bar{g}}{\underline{\rho}}+CT^{\frac{1}{2}-a}\cdot\textup{polylog}(T)$ with probability at least $1-\frac{1}{T}.$ We conclude by noting that $\text{Regret}(\text{SA} )\leq \mathbb{E}_{\gamma \sim \mathcal{P}^T}[\text{OPT}(\gamma)]\leq \bar{r}T$, so we have $\text{Regret}(\text{SA})\leq \frac{\bar{r}\bar{g}}{\underline{\rho}}+\bar{r}+CT^{\frac{1}{2}-a}\cdot\textup{polylog}(T)\in\tilde{O}(\max\{T^{\frac{1}{2}-a},1\})$.
\endproof

\vspace{1em}

\proof{Proof of \cref{thm:adversarial}.} By Assumption \ref{assumption: Adversarial Arrival Model(b)}, there exists a function $\psi(T)$ such that $||\hat{\mu}-\mu^*||_1\leq\psi (T)$ and $\psi(T)=o(\epsilon(T))$. We break down the proof into two lemmas, which compares $R(\textup{AA}\mid \gamma)$ with $\frac{1}{\alpha^*}\textup{OPT}(\gamma)$ and $R(\textup{PRD}\mid \gamma)$ separately. 
\begin{lemma}\label{lemma:adversarial_1}
	Consider the Adversarial Arrival Algorithm \textup{(AA)} under the adversarial arrival model. Given a prediction $\hat{\mu}$ with accuracy parameter $a$, it holds that: $$\limsup_{T\to\infty}\sup_{\gamma\in\mathcal{S}^T}\left\{\frac{1}{T}\left( \frac{1}{\alpha^*}\textup{OPT}(\gamma)- R(\textup{AA}\mid \gamma)\right)\right\}\leq 0.$$
\end{lemma}

\proof{Proof of \cref{lemma:adversarial_1}.}

The proof is drawn from the proof of Theorem 2 in \cite{balseiro2023best}. The proof contains three steps, which is similar to the proof of \cref{thm:stochastic}.
\\
\text{ }\text{ }\text{ } \textbf{Step 1 (Primal performance.)}
Fix an arrival sequence $\gamma\in\mathcal{S}^T$ and let $x^*\in\mathcal{X}_t$ be an optimal action in $\text{OPT}(\gamma)$ at time $t$. Let $\tau_A$ be the stopping time of Algorithm \ref{alg:adversarial}, which is defined similarly as in the proof of \cref{thm:stochastic}, then for $t\leq\tau_A$ we have $x_t\in\arg \max_{ x\in \mathcal{X}_t}\left\{r_t(x)-\mu_t^{\top} g_t(x)\right\}$, and thus $r_t\left(x_t\right) \geq r_t\left(x_t^*\right)-\mu_t^{\top}\left(g_t\left(x_t^*\right)-g_t\left(x_t\right)\right)$ and $0=r_t(0) \leq r_t\left(x_t\right)-\mu_t^{\top} g_t\left(x_t\right)$. Therefore
\begin{eqnarray*}
	\alpha^* r_t\left(x_t\right) & =&r_t\left(x_t\right)+(\alpha^*-1)r_t\left(x_t\right) \\
	& \geq& r_t\left(x_t^*\right)+\mu_t^{\top} g_t\left(x_t\right)-\mu_t^{\top} g_t\left(x_t^*\right)+(\alpha^*-1)\left(\mu_t^{\top} g_t\left(x_t\right)\right) \\
	& =&r_t\left(x_t^*\right)-\alpha^* \mu_t^{\top}\left(\rho-g_t\left(x_t\right)\right)+\alpha^* \mu_t^{\top} \rho-\mu_t^{\top} g_t\left(x_t^*\right) \\
	& \geq& r_t\left(x_t^*\right)-\alpha^*\mu_t^{\top}\left(\rho-g_t\left(x_t\right)\right),
\end{eqnarray*}
where the second inequality is because $\alpha^* \mu_t^{\top} \rho-\mu_t^{\top} g_t\left(x_t^*\right) \geq 0$ by our definition of $\alpha^*$ and the fact that $\mu_t \geq 0$. Summing up over $t=1, \ldots, \tau_A$ yields
\begin{eqnarray}\label{eqn:proof of thm2 slackness}
	\alpha^* \sum_{t=1}^{\tau_A} r_t\left(x_t\right) \geq \sum_{t=1}^{\tau_A} r_t\left(x_t^*\right)-\alpha^* \sum_{t=1}^{\tau_A} \mu_t^{\top}\left(\rho-g_t\left(x_t\right)\right) .
\end{eqnarray}

\textbf{Step 2 (Complementary slackness).} Denoting, as before, $w_t(\mu)=\mu^{\top}\left(\rho-b_t\left(x_t\right)\right)$. As we have seen in the step 2 in the proof of \cref{thm:stochastic}  (the analysis is deterministic in nature), Algorithm \ref{alg:adversarial} performs online mirror descent to the sequence of functions $w_t(\mu)$ with step size $\eta=c\epsilon(T)/T$ where $c>0$ is an arbitrary scaling constant. By our assumption that the reference function $h(\cdot)$ is Lipschitz, there exists a constant $L>0$ such that $V_h(\mu',\mu'')\leq L||\mu'-\mu''||_1$ for all $\mu',\mu''\in\mathbb{R}^m_+$. By a standard result on online mirror descent (see, e.g., Appendix G of \cite{balseiro2023best}), we have \begin{eqnarray}\label{eqn:proof of thm2 oco}
	\sum_{t=1}^{\tau_A} w_t(\mu_t)&\leq& \sum_{t=1}^{\tau_A} w_t(\mu^*)+\frac{(\bar{g}+\bar{\rho})^2\eta}{2\sigma}\tau_A+\frac{1}{\eta}V_h(\mu^*,\mu_1) \nonumber \\
	&\leq& \sum_{t=1}^{\tau_A}w_t(\mu^*)+\frac{c(\bar{g}+\bar{\rho})^2}{2\sigma}\epsilon(T)+\frac{\kappa L\psi(T)T}{c\epsilon(T)},
\end{eqnarray}
where the first inequality is the standard online mirror descent result, and the second inequality follows by the step size $\eta=c\epsilon(T)/T$ and the fact that $||\mu_1-\mu^*||_1=||\hat{\mu}-\mu^*||_1\leq\psi(T)$.

\textbf{Step 3 (Putting it all together).}
We have
$$
\begin{aligned}
	\operatorname{OPT}(\gamma)-\alpha^* R(\text{AA} \mid \gamma) & \leq \sum_{t=1}^T r_t\left(x_t^*\right)-\alpha^* \sum_{t=1}^{\tau_A} r_t\left(x_t\right) \\
	& \leq \sum_{t=\tau_A+1}^T r_t\left(x_t^*\right)+\alpha^* \sum_{t=1}^{\tau_A} w_t\left(\mu_t\right) \\
	& \leq \left(T-\tau_A\right) \cdot \bar{r}+\alpha^* \sum_{t=1}^{\tau_A}w_t(\mu^*)+\alpha^*\left(\frac{c(\bar{g}+\bar{\rho})^2}{2\sigma}\epsilon(T)+\frac{\kappa L\psi(T)T}{c\epsilon(T)} \right),
\end{aligned}
$$
where the first inequality follows because $\tau_A \leq T$ and $r_t(\cdot) \geq 0$, the second inequality is from Eq. (\ref{eqn:proof of thm2 slackness}), and the third inequality utilizes $r_t\left(x_t^*\right) \leq \bar{r}$ and Eq. (\ref{eqn:proof of thm2 oco}). Similar to the proof of \cref{thm:stochastic}, we note that $\sum_{t=1}^{\tau_A} w_t(\mu^*)\leq \sum_{t=1}^{\tau_A} w_t(\mu)$ for every $\mu\in \mathbb{R}_{+}^m$ and discuss the choice of $\mu \in \mathbb{R}_{+}^m$ in order to upper bound $\sum_{t=1}^{\tau_A} w_t(\mu^*)$.
If $\tau_A=T$, then set $\mu=0$, and the result follows. If $\tau_A<T$, then there exists a resource $j \in[m]$ such that $\sum_{t=1}^{\tau_A}\left(g_t\left(x_t\right)\right)_j+\bar{g} \geq \rho_j T $. Set $\mu=\left(\bar{r} /\left(\alpha^* \rho_j\right)\right) e_j$ where $e_j$ is the $j$-th unit vector and repeat the steps of the stochastic arrivals case to obtain:
$$
\operatorname{OPT}(\gamma)-\alpha^* R(\text{AA} \mid \gamma) \leq \frac{\bar{r}\bar{g}}{\underline{\rho}}+\alpha^*\left(\frac{c(\bar{g}+\bar{\rho})^2}{2\sigma}\epsilon(T)+\frac{\kappa L\psi(T)T}{c\epsilon(T)}\right),
$$
which finishes the proof by noticing that $\epsilon(T)$ and $\psi(T)T/\epsilon(T)$ are both sub-linear in $T$.
\endproof
\vspace{1em}

\begin{lemma}\label{lemma:adversarial_2}
	Consider the Adversarial Arrival Algorithm \textup{(AA)} under the adversarial arrival model. Given a prediction $\hat{\mu}$ with accuracy parameter $a$, it holds that: $$\limsup_{T\to\infty}\sup_{\gamma\in\mathcal{S}^T}\left\{\frac{1}{T}\left( \textup{PRD}(\gamma)- R(\textup{AA}\mid \gamma)\right)\right\}\leq 0.$$
\end{lemma}

\proof{Proof of \cref{lemma:adversarial_2}.}

Recall the updating rule $\mu_{t+1}\in \arg \min _{\mu \in \mathbb{R}_{+}^m} \phi_t^{\top} \mu+\frac{1}{\eta} V_h\left(\mu, \mu_t\right)$ where $\phi_t=-g_t(x_t)+\rho$. Note that $\phi_t^{\top} \mu+\frac{1}{\eta} V_h\left(\mu, \mu_t\right)$ is convex in $\mu$, and set its gradient of $\mu$ to zero yields $\phi_t+\frac{1}{\eta}(\nabla h(\mu)-\nabla h(\mu_t))=0$, where $h(\cdot)$ is the reference function. Because $||\phi_t||_\infty\leq ||g_t(x_t)||_\infty+||\rho||_\infty\leq \bar{g}+\bar{\rho}$ and by our assumption $h(\cdot)$ is $\sigma$-strongly convex with respect to the $\ell_1$-norm in $\mathbb{R}^m_+$, we have $||\mu_{t+1}-\mu_t||_1\leq\frac{\eta}{\sigma}||\phi_t||_\infty\leq \frac{c(\bar{g}+\bar{\rho})\epsilon(T)}{\sigma T}$. Therefore $||\mu_{t}-\hat{\mu}||_1=||\mu_{t}-\mu_1||_1\leq\sum_{s=2}^t||\mu_{s}-\mu_{s-1}||_1\leq \frac{c(\bar{g}+\bar{\rho})\epsilon(T)}{\sigma T}t$, and hence \begin{equation}\label{eqn: proof of thm2 bound of mu}
	\sum_{t=1}^T ||\mu_{t}-\hat{\mu}||_1\leq \left(\frac{c(\bar{g}+\bar{\rho})}{2\sigma}+1\right)\epsilon(T).
\end{equation}

Let $x_t^{\hat{\mu}}$ be the actions taken by the Prediction Algorithm at time $t$, then $\textup{PRD}(\gamma)=\sum_{t=1}^T r_t(x_t^{\hat{\mu}})$. 
Because $\zeta>0$ is a constant and $\epsilon(T)\in o(1)$, $||\mu_{t}-\hat{\mu}||_1\leq \frac{c(\bar{g}+\bar{\rho})\epsilon(T)}{\sigma T}t\leq\zeta$ for all $t$ as $T\to\infty$. Therefore $\mu_1,\dots,\mu_T$ is a sequence of dual variables that satisfies Assumption \ref{assumption: Adversarial Arrival Model(b)}. Let $\tau_A^j$ be the depletion time of resources $j$ of Algorithm \ref{alg:adversarial} and $\tau_P^j$ be the depletion time of resources $j$ of Algorithm \ref{alg:prediction}. Then by Assumption \ref{assumption: Adversarial Arrival Model(b)} we have $|\tau_P^j-\tau_A^j|\in o(T)$ for all resource $j$. Moreover, since there are $m$ resources, outside of all times between each $\tau_A^j$ and $\tau_P^j$, $T$ is partitioned into at most $m+1$ consecutive time blocks, say $B_1,\dots,B_k$ for some $k\leq m+1$. Note that the set of feasible actions $\{x\mid x\in\mathcal{X}_t,g_t(x)\leq \text{amount  of remaining resources}\}$ at time period $t$ is the same for Algorithm \ref{alg:adversarial} and Algorithm \ref{alg:prediction} for all $t\in \cup_{k'=1}^k B_{k'}$. Therefore both algorithms perform online mirror descent during time periods $B_{1},\dots,B_k$. Therefore similar to Eq. (\ref{eqn:proof of thm2 oco}) we have 
\begin{equation}\label{eqn: proof of thm2 bound of w}
	\sum_{t\in B_{k'}} w_t(\mu_t)\leq \sum_{t\in B_{k'}}w_t(\hat{\mu})+\frac{c(\bar{g}+\bar{\rho})^2}{2\sigma}\epsilon(T)+\frac{\kappa L\psi(T)T}{c\epsilon(T)}
\end{equation} for each $B_{k'}$. Also, because $x_t\in \arg \max _{ x\in \mathcal{X}_t,g_t(x)\leq G_t}\left\{r_t(x)-\mu_t^{\top} g_t(x)\right\}$, for $t\in \cup_{k'=1}^k B_{k'}$ we have \begin{equation}\label{eqn: proof of thm2 bound of r}
	r_t(x_t)-\mu_t^{\top} g_t(x_t)\geq r_t(x_t^{\hat{\mu}})-\mu_t^{\top} g_t(x_t^{\hat{\mu}}).
\end{equation} Because $w_t(\mu_t)=\mu_t^\top(\rho-g_t(x_t))$ and $w_t(\hat{\mu})=\hat{\mu}^\top(\rho-g_t(x_t^{\hat{\mu}}))$, for each $B_k'$ we get
\begin{eqnarray*}
	\sum_{t\in B_{k'}} (r_t(x_t^{\hat{\mu}})-r_t(x_t))&\leq&\sum_{t\in B_{k'}} (\mu_t^{\top} g_t(x_t^{\hat{\mu}})-\mu_t^{\top} g_t(x)) \\
	&= &\sum_{t\in B_{k'}} \mu_t^{\top}(\rho- g_t(x))-\sum_{t\in B_{k'}}\hat{\mu}^{\top} (\rho-g_t(x_t^{\hat{\mu}}))+\sum_{t\in B_{k'}}(\hat{\mu}-\mu_t)^{\top} (\rho-g_t(x_t^{\hat{\mu}}))\\
	&\leq& \sum_{t\in B_{k'}} w_t(\mu_t)-\sum_{t\in B_{k'}}w_t(\hat{\mu})+\bar{\rho}\sum_{t\in B_{k'}}||\mu_{t}-\hat{\mu}||_1\\
	&\leq& \frac{c(\bar{g}+\bar{\rho})^2}{2\sigma}\epsilon(T)+\frac{\kappa L\psi(T)T}{c\epsilon(T)}+\left(\frac{c\bar{\rho}(\bar{g}+\bar{\rho})}{2\sigma}+\bar{\rho}\right)\epsilon(T),
\end{eqnarray*}
where the first inequality follows from Eq. (\ref{eqn: proof of thm2 bound of r}), the second inequality is because by H\"older's inequality $(\hat{\mu}-\mu_t)^{\top} (\rho-g_t(x_t^{\hat{\mu}}))\leq ||\hat{\mu}-\mu_t||_1||\rho-g_t(x_t^{\hat{\mu}})||_\infty\leq \bar{\rho}||\hat{\mu}-\mu_t||_1$, and the third inequality follows from Eq. (\ref{eqn: proof of thm2 bound of w}) and Eq. (\ref{eqn: proof of thm2 bound of mu}). Therefore
\begin{eqnarray*}
	\textup{PRD}(\gamma)- R(\textup{AA}\mid \gamma)&=& \sum_{t\in\cup_{k'=1}^k B_{k'}} (r_t(x_t^{\hat{\mu}})-r_t(x_t))+\sum_{t\in [T]\setminus\cup_{k'=1}^k B_{k'}} (r_t(x_t^{\hat{\mu}})-r_t(x_t))\\
	&\leq& k\left(\frac{c(\bar{g}+\bar{\rho})^2}{2\sigma}\epsilon(T)+\frac{\kappa L\psi(T)T}{c\epsilon(T)}+\left(\frac{c\bar{\rho}(\bar{g}+\bar{\rho})}{2\sigma}+\bar{\rho}\right)\epsilon(T)\right)\\&\quad&+\bar{r}\cdot|[T]\setminus\cup_{k'=1}^k B_{k'}|\\
	&=&o(T),
\end{eqnarray*}
where the first inequality is because $r_t(x_t^{\hat{\mu}})-r_t(x_t)\leq\bar{r}$ for each $t$, and the second inequality is by noting that $k\leq m$, $\epsilon(T)= o(T),\psi(T)T/\epsilon(T)=o(T)$, and $|[T]\setminus\cup_{k'=1}^k B_{k'}|\leq\sum_{j=1}^m|\tau_P^j-\tau_A^j|$. This shows $\limsup_{T\to\infty}\sup_{\gamma\in\mathcal{S}^T}\left\{\frac{1}{T}\left(\textup{PRD}(\gamma)- R(\textup{AA}\mid \gamma)\right)\right\}\leq 0.$ 
\endproof
Combine \cref{lemma:adversarial_1} and \cref{lemma:adversarial_2} gives \cref{thm:adversarial}. 
\endproof

\section{Proofs in Section \ref{Section main algorithm}.3}\label{Appendix C}

\proof{Proof of \cref{thm:main}.}

We divide the proof into three cases. The first case is that the underlying arrival model is stochastic and the algorithm never switches to the Adversarial Arrival Algorithm (i.e., the ``for'' loop in the algorithm is never broken), and in this case we show that $\textup{Regret}(\textup{MainALG})=\tilde{O}(\max\{T^{\frac{1}{2}-a},1\})$. The second case is that the underlying arrival model is stochastic and yet the algorithm switches to the Adversarial Arrival Algorithm at some point, and we prove that this case happens with low probability. The third case is that the underlying arrival model is adversarial, and in this case we show that $\limsup_{T\to\infty}\sup_{\gamma\in\mathcal{S}^T}\left\{\frac{1}{T}(1-\lambda)\left(\max\left\{\frac{1}{\alpha^*}\textup{OPT}(\gamma),\textup{PRD}(\gamma)\right\} - R(\textup{MainALG}\mid \gamma)\right)\right\}\leq \delta\bar{r}$, regardless of whether the algorithm switches to the Adversarial Arrival Algorithm or not. To simplify the notation, throughout the proof we will assume $\delta T$ and $1/\delta$ are integers. The roundings $\lfloor \delta T \rfloor$ and $\lceil 1/\delta \rceil$ in our algorithm will not affect the result of our analysis.

\noindent \textbf{Case 1:}

Suppose the underlying arrival model is stochastic where each arrival $\gamma_t$ is drawn i.i.d. from an underlying probability distribution $\mathcal{P}\in\Delta(\mathcal{S})$, and the algorithm never switches to the Adversarial Arrival Algorithm. Then the algorithm decomposes $T$ time periods into $1/\delta$ time blocks, where each time block contains $\delta T$ time periods and has at least $\delta T \rho$ amount of resources available. During each time block the algorithm performs the Stochastic Arrival Algorithm. Therefore, by our definition of $\text{OPT}_s(\gamma_1,\dots,\gamma_s)$ and the performance guarantee of the Stochastic Arrival Algorithm given by \cref{thm:stochastic}, we have \begin{eqnarray}\label{eqn:proof of thm 3,0}
	\mathbb{E}_{\gamma\sim\mathcal{P}^T}\left[\sum_{k=0}^{1/\delta-1}\text{OPT}(\gamma_{k\delta T+1:(k+1)\delta T})-R(\textup{MainALG}\mid \gamma)\right]&=& \tilde{O}\left(\frac{1}{\delta}\max\left\{(\delta T)^{\frac{1}{2}-a},1\right\}\right)\nonumber\\&=&\tilde{O}\left(\max\left\{T^{\frac{1}{2}-a},1\right\}\right).
\end{eqnarray}

For each time period $t$, let $D_t(\mu\mid \gamma_t):=r^*_t(\mu)+\mu^\top\rho$ be the $t$-term of the Lagrangian dual function \cref{eqn:dual}, then every $D_t(\mu\mid \gamma_t)$ is also i.i.d. By \cref{lemma:weak duality} and \cref{lemma:duality gap}, for every arrival sequence $\gamma$ we have \begin{equation}\label{eqn:proof of thm 3,2}
	\text{OPT}(\gamma_{1:s})\leq\sum_{t=1}^s D_t(\mu\mid \gamma_t) \quad \forall \mu \in\mathbb{R}^m_+\end{equation} and 
\begin{equation}\label{eqn:proof of thm 3,3}
	\min_{\mu \in\mathbb{R}^m_+}\sum_{t=1}^s D_t(\mu\mid \gamma_t)\leq \text{OPT}(\gamma_{1:s})+(m+1)\bar{r}
\end{equation} for every time period $s$.
Setting $s=T$, taking $\mu \in\mathbb{R}^m_+$ to be the minimizer, and taking the expected value of Eq. (\ref{eqn:proof of thm 3,2}) gives 
\begin{equation}\label{eqn:proof of thm 3,2'}
	\mathbb{E}_{\gamma\sim\mathcal{P}^T}\left[\text{OPT}(\gamma)\right]\leq\mathbb{E}_{\gamma\sim\mathcal{P}^T}\left[\min_{\mu \in\mathbb{R}^m_+}\sum_{t=1}^T D_t(\mu\mid \gamma_t)\right]. \end{equation}
Taking expected value on both sides of Eq. (\ref{eqn:proof of thm 3,3}) yields \begin{equation}\label{eqn:proof of thm 3,4}
	\mathbb{E}_{(\gamma_{1:s})\sim\mathcal{P}^s}\left[\min_{\mu \in\mathbb{R}^m_+}\sum_{t=1}^s D_t(\mu\mid \gamma_t)\right]\leq \mathbb{E}_{(\gamma_{1:s})\sim\mathcal{P}^s}[\text{OPT}(\gamma_{1:s})]+(m+1)\bar{r}.
\end{equation} 
Therefore
\begin{equation}\label{eqn:proof of thm 3,4'}
	\mathbb{E}_{\gamma\sim\mathcal{P}^T}\left[\sum_{k=0}^{1/\delta-1}\min_{\mu \in\mathbb{R}^m_+}\sum_{t=k\delta T+1}^{(k+1)\delta T} D_t(\mu\mid \gamma'_t)\right]\leq\mathbb{E}_{\gamma\sim\mathcal{P}^T}\left[\sum_{k=0}^{1/\delta-1}\text{OPT}(\gamma_{k\delta T+1:(k+1)\delta T})\right]+(m+1)\bar{r}/\delta.
\end{equation}
Combine Eq. (\ref{eqn:proof of thm 3,2'}) and Eq. (\ref{eqn:proof of thm 3,4'}) we have
\begin{eqnarray}\label{eqn:proof of thm 3,4''}
	\mathbb{E}_{\gamma\sim\mathcal{P}^T}\left[\text{OPT}(\gamma_{1:T})\right]&\leq&\mathbb{E}_{\gamma\sim\mathcal{P}^T}\left[\min_{\mu \in\mathbb{R}^m_+}\sum_{t=1}^T D_t(\mu\mid \gamma_t)\right]\nonumber\\
	&\leq &\mathbb{E}_{\gamma\sim\mathcal{P}^T}\left[\sum_{k=0}^{1/\delta-1}\min_{\mu \in\mathbb{R}^m_+}\sum_{t=k\delta T+1}^{(k+1)\delta T} D_t(\mu\mid \gamma'_t)\right]\nonumber\\
	&\leq &\mathbb{E}_{\gamma\sim\mathcal{P}^T}\left[\sum_{k=0}^{1/\delta-1}\text{OPT}(\gamma_{k\delta T+1:(k+1)\delta T})\right]+(m+1)\bar{r}/\delta.
\end{eqnarray}
We conclude the proof of case 1 by combining Eq. (\ref{eqn:proof of thm 3,0}) and Eq. (\ref{eqn:proof of thm 3,4''})and noting that $(m+1)\bar{r}/\delta$ is a constant:
\begin{eqnarray*}
	\textup{Regret}(\textup{MainALG})=\mathbb{E}_{\gamma\sim\mathcal{P}^T}\left[\text{OPT}(\gamma
	)-R(\textup{MainALG}\mid \gamma)\right]=\tilde{O}(\max\{T^{\frac{1}{2}-a},1\}).
\end{eqnarray*}

\noindent \textbf{Case 2:}

Suppose the underlying arrival model is stochastic where each arrival $\gamma_t$ is drawn i.i.d. from an underlying probability distribution $\mathcal{P}\in\Delta(\mathcal{S})$. We show that the probability that the algorithm switches to the Adversarial Arrival Algorithm is low. More specifically, we show that this probability is no more than $\frac{3+\delta}{\delta^2 T}$.

First we prove a Chernoff-like bound for sums with stopping times.
\begin{lemma}[Stopping Time Chernoff]\label{lemma:Chernoff}
	Consider a discrete-time random sequence with states $S_1,S_2,\dots$ where each state $S_t$ determines two values $x_t$ and $y_t$ with $x_t,y_t\in[0,c]$ for some constant $c>0$. Suppse $\mathbb{E}[x_ t \mid S_{t-1}]\leq\mathbb{E}[y_ t \mid S_{t-1}]$. Then for every $0<\epsilon<1$ and every $\mu>0$ we have
	$$\mathbb{P}\left(\exists\text{ } \tau\text{ such that }\sum_{t=1}^\tau x_t/(1+\epsilon) - \sum_{t=1}^\tau y_t/(1-\epsilon) \geq \epsilon \mu c\right)<\exp(-\epsilon^2\mu).$$
\end{lemma}

\proof{Proof of \cref{lemma:Chernoff}.}

Let $\phi_0=1$, and for $\tau=1,2,\dots$ let $\phi_\tau = (1+\epsilon)^{\sum_{t=1}^\tau x_t/c}
(1-\epsilon)^{\sum_{t=1}^\tau y_t/c}.$ Then $\phi_0,\phi_1,\dots$ is a non-negative super-martingale. Indeed, for $\tau\geq 1$ we have $\phi_\tau/\phi_{\tau-1}
=(1+\epsilon)^{x_\tau/c}(1-\epsilon)^{y_\tau/c}
\leq(1 +\epsilon x_\tau/c) (1-\epsilon y_\tau/c)
\leq 1+\epsilon x_\tau/c - \epsilon y_\tau/c$, where the first inequality is because $x_t/c,y_t/c\in[0,1]$ for every $t$. Because $\mathbb{E}[x_ t \mid S_{t-1}]\leq\mathbb{E}[y_ t \mid S_{t-1}]$, we get $\mathbb{E}[\phi_\tau/\phi_{\tau-1}\mid S_{\tau-1}]\leq 1$, which shows $\phi_0,\phi_1,\dots$ is a non-negative super-martingale.

If the event in the statement happens at some $\tau$, then $\exp(\sum_{t=1}^\tau \epsilon x_t/c(1+\epsilon) - \sum_{t=1}^\tau \epsilon y_t/c(1-\epsilon))\geq \exp(\epsilon^2\mu)$. Using $e^{\epsilon/(1-\epsilon)}<1+\epsilon$ we get $\phi_\tau=(1+\epsilon)^{\sum_{t=1}^\tau x_t/c}
(1-\epsilon)^{\sum_{t=1}^\tau y_t/c}> \exp(\epsilon^2\mu)$. Therefore 
\begin{eqnarray*}
	\mathbb{P}\left(\exists\text{ } \tau\text{ such that }\sum_{t=1}^\tau x_t/(1+\epsilon) - \sum_{t=1}^\tau y_t/(1-\epsilon) \geq \epsilon \mu c\right)&\leq&
	\mathbb{P}\left(\exists\text{ } \tau\text{ such that }\phi_\tau >\exp(\epsilon^2\mu)\right)
	\\&<&\exp(-\epsilon^2\mu),
\end{eqnarray*} where the second inequality follows by Doob's martingale inequality.
\endproof
\vspace{1em}

To analyze the reward obtained so far by the algorithm at a certain time period, we revisit the proof of \cref{thm:stochastic} and inherit all notations are ed from the proof of \cref{thm:stochastic}. Recall in Eq. (\ref{eqn:proof of thm 1}) we have $\mathbb{E}_{\gamma\sim\mathcal{P}^T}\left[r_t\left(x_t\right) \mid \sigma\left(\xi_{t-1}\right)\right] =\mathbb{E}_{\gamma\sim\mathcal{P}^T}\left[\bar{D}\left(\mu_t \mid \mathcal{P}\right) -w_t(\mu_t)\mid \sigma\left(\xi_{t-1}\right)\right]$. For any $x_t\in\mathcal{X}_t$ and $\mu_t\in\mathbb{R}_+^m$ we have $0\leq r_t(x_t), \bar{D}\left(\mu_t \mid \mathcal{P}\right) -w_t(\mu_t)\leq\bar{r}$. Therefore, for the stopping time $\tau_A$ defined in the proof of \cref{thm:stochastic}, we can apply \cref{lemma:Chernoff} on $\bar{D}\left(\mu_t \mid \mathcal{P}\right) -w_t(\mu_t)$ and $r_t(x_t)$, which gives
\begin{eqnarray*}
	&\quad& \mathbb{P}_{\gamma\sim\mathcal{P}^T}\left(\sum_{t=1}^{\tau_A}(\bar{D}\left(\mu_t \mid \mathcal{P}\right) -w_t(\mu_t))-\sum_{t=1}^{\tau_A}r_t(x_t)\geq 2\epsilon'\bar{r}T/(1-\epsilon')+(1+\epsilon')\epsilon'\mu'\bar{r}\right)\\&\leq& \mathbb{P}_{\gamma\sim\mathcal{P}^T}\left(\sum_{t=1}^{\tau_A}(\bar{D}\left(\mu_t \mid \mathcal{P}\right) -w_t(\mu_t))-\sum_{t=1}^{\tau_A}r_t(x_t)(1+\epsilon')/(1-\epsilon')\geq (1+\epsilon')\epsilon'\mu'\bar{r}\right)
	\\&\leq& \mathbb{P}_{\gamma\sim\mathcal{P}^T}\left(\sum_{t=1}^{\tau_A}(\bar{D}\left(\mu_t \mid \mathcal{P}\right) -w_t(\mu_t))/(1+\epsilon')-\sum_{t=1}^{\tau_A}r_t(x_t)/(1-\epsilon')\geq \epsilon'\mu'\bar{r}\right)
	\\&<&\exp(-\epsilon'^2\mu'),
\end{eqnarray*}
where the first inequalities follows because $r_t(x_t)\leq\bar{r}$ and $\tau_A\leq T$, so $2\sum_{t=1}^{\tau_A}r_t(x_t)/(1-\epsilon')\leq 2\epsilon'\bar{r}T/(1-\epsilon')$; the second inequality is obtained by dividing $1+\epsilon'$ on both sides of the inequality; the third inequality utilizes \cref{lemma:Chernoff}. Plug in $\epsilon'=T^{-1/2}$ and $\mu=T\log (T)$ yields
\begin{equation}\label{eqn:bound of R}
	\mathbb{P}_{\gamma\sim\mathcal{P}^T}\left(\sum_{t=1}^{\tau_A}(\bar{D}\left(\mu_t \mid \mathcal{P}\right) -w_t(\mu_t))-\sum_{t=1}^{\tau_A}r_t(x_t)\geq  (4\bar{r}+2\bar{r}\log (T))\sqrt{T} \right)<\frac{1}{T}.
\end{equation} We will use Eq. (\ref{eqn:bound of R}) later in bounding the concentration of $R(\text{SA}\mid\gamma)$.

Then we look to bound the concentration of $\text{OPT}(\gamma)$. Because $0\leq r^*_t(\mu)\leq \bar{r}$ for every $\mu\in\mathbb{R}^m_+$, by Hoeffding's inequality we have
\begin{equation}\label{eqn:proof of thm 3,1}
	\mathbb{P}_{\gamma\sim\mathcal{P}^T}\left(\sum_{t=1}^T D_t(\mu\mid \gamma_t)-\mathbb{E}_{\gamma\sim\mathcal{P}^T}\left[\sum_{t=1}^T D_t(\mu\mid \gamma'_t)\right]>y\right)\leq \exp{\left(-\frac{2y^2}{\bar{r}^2 T}\right)}
\end{equation} 
and 
\begin{equation}\label{eqn:proof of thm 3,1.5}
	\mathbb{P}_{\gamma\sim\mathcal{P}^T}\left(\mathbb{E}_{\gamma'\sim\mathcal{P}^T}\left[\sum_{t=1}^s D_t(\mu\mid \gamma'_t)\right]-\sum_{t=1}^s D_t(\mu\mid \gamma_t)>y\right)\leq \exp{\left(-\frac{2y^2}{\bar{r}^2 s}\right)} \quad\forall s
\end{equation}
for every $\mu \in\mathbb{R}^m_+$ and $y> 0$. 
Apply Eq. (\ref{eqn:proof of thm 3,2}) and Eq. (\ref{eqn:proof of thm 3,4}) to Eq. (\ref{eqn:proof of thm 3,1}) and take $\mu$ to be the minimizer on the left hand side of Eq. (\ref{eqn:proof of thm 3,4}) gives \begin{equation}\label{eqn:proof of thm 3,2.0.5}
	\mathbb{P}_{\gamma\sim\mathcal{P}^T}\left(\text{OPT}(\gamma)-\mathbb{E}_{\gamma\sim\mathcal{P}^T}\left[\text{OPT}(\gamma)\right]>y+(m+1)\bar{r}\right)\leq \exp{\left(-\frac{2y^2}{\bar{r}^2T}\right)}.
\end{equation}
Take $y=\sqrt{\bar{r}^2T\log(T)/2}$ yields
\begin{equation}\label{eqn:proof of thm 3,2.1}
	\mathbb{P}_{\gamma\sim\mathcal{P}^T}\left(\text{OPT}(\gamma)-\mathbb{E}_{\gamma\sim\mathcal{P}^T}\left[\text{OPT}(\gamma)\right]>\sqrt{\bar{r}^2T\log(T)/2}+(m+1)\bar{r}\right)\leq \frac{1}{T}.
\end{equation}

Recall in the steps of Eq. (\ref{eqn:clubsuit}), we have $R(\text{SA}\mid\gamma)\geq \sum_{t=1}^{\tau_A}r_t(x_t)$ and $\bar{D}\left(\mu_t \mid \mathcal{P}\right)\geq\tau_A\bar{D}\left(\bar{\mu}_{\tau_A} \mid \mathcal{P}\right)$. Combine Eq. (\ref{eqn:bound of R}) and Eq. (\ref{eqn:proof of thm 3,2.1}) gives that, for every $z>0$,
\begin{eqnarray*}
	&\quad&\mathbb{P}_{\gamma\sim\mathcal{P}^T}\left(\text{OPT}(\gamma)-R(\text{SA}\mid\gamma)\geq z+(4\bar{r}+2\bar{r}\log (T))\sqrt{T}+\sqrt{\bar{r}^2T\log(T)/2}+(m+1)\bar{r}\right)\\
	&\overset{(a)}{\leq}&\mathbb{P}_{\gamma\sim\mathcal{P}^T}\left(\mathbb{E}_{\gamma\sim\mathcal{P}^T}\left[\text{OPT}(\gamma)\right]-R(\text{SA}\mid\gamma)\geq z+(4\bar{r}+2\bar{r}\log (T))\sqrt{T}\right)+\frac{1}{T}\\
	&\overset{(b)}{\leq}&\mathbb{P}_{\gamma\sim\mathcal{P}^T}\left(\mathbb{E}_{\gamma\sim\mathcal{P}^T}\left[\text{OPT}(\gamma)\right]-\sum_{t=1}^{\tau_A}r_t(x_t)\geq z+(4\bar{r}+2\bar{r}\log (T))\sqrt{T}\right)+\frac{1}{T}\\
	&\overset{(c)}{\leq}&\mathbb{P}_{\gamma\sim\mathcal{P}^T}\left(\mathbb{E}_{\gamma\sim\mathcal{P}^T}\left[\text{OPT}(\gamma)\right]-\sum_{t=1}^{\tau_A}(\bar{D}\left(\mu_t \mid \mathcal{P}\right) -w_t(\mu_t))\geq z\right)+\frac{2}{T}\\
	&\overset{(d)}{\leq}&\mathbb{P}_{\gamma\sim\mathcal{P}^T}\left(\mathbb{E}_{\gamma\sim\mathcal{P}^T}\left[\text{OPT}(\gamma)\right]-\tau_A\bar{D}\left(\bar{\mu}_{\tau_A}\mid \mathcal{P}\right)+\sum_{t=1}^{\tau_A}w_t(\mu_t)\geq z\right)+\frac{2}{T}\\
	&\overset{(e)}{\leq}&\mathbb{P}_{\gamma\sim\mathcal{P}^T}\left(\left(T-\tau_A\right) \cdot \bar{r}+\sum_{t=1}^{\tau_A} w_t(\mu^*)+CT^{\frac{1}{2}-a}\cdot\textup{polylog}(T)\geq z\right)+\frac{2}{T}\\
	&\overset{(f)}{\leq}&\mathbb{P}_{\gamma\sim\mathcal{P}^T}\left(\frac{\bar{r}\bar{g}}{\underline{\rho}}+CT^{\frac{1}{2}-a}\cdot\textup{polylog}(T)\geq z\right)+\frac{3}{T}.
\end{eqnarray*}
Here $(a)$ follows by Eq. (\ref{eqn:proof of thm 3,2.1}); $(b)$ is because $R(\text{SA}\mid\gamma)\geq \sum_{t=1}^{\tau_A}r_t(x_t)$; $(c)$ follows by Eq. (\ref{eqn:bound of R}); $(d)$ is because $\bar{D}\left(\mu_t \mid \mathcal{P}\right)\geq\tau_A\bar{D}\left(\bar{\mu}_{\tau_A} \mid \mathcal{P}\right)$; $(e)$ holds since the last three steps of Eq. (\ref{eqn:clubsuit}) is deterministic in nature; $(f)$ follows from the last paragraph of the proof of \cref{thm:stochastic}. Take $z=\frac{\bar{r}\bar{g}}{\underline{\rho}}$ and note that $z+(4\bar{r}+2\bar{r}\log (T))\sqrt{T}+\sqrt{\bar{r}^2T\log(T)/2}+(m+1)\bar{r}\in O(\log(T)\sqrt{T})$, i.e., there exists a constant $C'>0$ such that $z+(4\bar{r}+2\bar{r}\log (T))\sqrt{T}+\sqrt{\bar{r}^2T\log(T)/2}+(m+1)\bar{r}<C'\log(T)\sqrt{T}$. This gives 
\begin{equation}\label{eqn:chernoff on SA}
	\mathbb{P}_{\gamma\sim\mathcal{P}^T}\left(\text{OPT}(\gamma)-R(\text{SA}\mid\gamma)>C'\log(T)\sqrt{T}\right)\leq\frac{3}{T}.
\end{equation}

Suppose the algorithm does not switch to the Adversarial Arrival Algorithm before time period $k'\delta T$ for some $k'\in\{0,\dots,1/\delta-1\}$. For $k=0,1,\dots,k'-1$, the algorithm performs the Stochastic Arrival Algorithm during each time block between time periods $k\delta T+1$ and $(k+1)\delta T$.
Apply Eq. (\ref{eqn:proof of thm 3,3}) to Eq. (\ref{eqn:proof of thm 3,1.5}) over each time block gives that, for every $y>0$,
\begin{equation}\label{eqn:proof of thm 3,2.2}
	\mathbb{P}_{\gamma\sim\mathcal{P}^T}\left(\mathbb{E}_{\gamma'\sim\mathcal{P}^T}\left[\sum_{t=k\delta T+1}^{(k+1)\delta T} D_t(\mu\mid \gamma'_t)\right]-\text{OPT}(\gamma_{k\delta T+1:(k+1)\delta T})>y+(m+1)\bar{r}\right)\leq \exp{\left(-\frac{2y^2}{\bar{r}^2 \delta T}\right)}.
\end{equation}
Let $X_k$ be the random variable such that $$X_k=\mathbb{E}_{\gamma'\sim\mathcal{P}^T}\left[\sum_{t=k\delta T+1}^{(k+1)\delta T} D_t(\mu\mid \gamma'_t)\right]-\text{OPT}(\gamma_{k\delta T+1:(k+1)\delta T})-(m+1)\bar{r}$$ where $\gamma_{k\delta T+1:(k+1)\delta T}\sim\mathcal{P}^{\delta T}$. Then by Eq. (\ref{eqn:proof of thm 3,2.2}) each $X_k$ is an independent sub-Gaussian random variable with parameter $\sqrt{2/\bar{r}^2\delta T}$. Therefore $\sum_{k=0}^{k'-1} X_k$ is also a sub-Gaussian random variable with parameter at most $\sqrt{2/\bar{r}^2\delta T}$. Hence we get
\begin{eqnarray}\label{eqn:proof of thm 3,2.3}
	&\qquad &\mathbb{P}_{\gamma\sim\mathcal{P}^T}\left(\mathbb{E}_{\gamma'\sim\mathcal{P}^T}\left[\sum_{t=1}^{k'\delta T} D_t(\mu\mid \gamma'_t)\right]-\sum_{k=0}^{k'-1}\text{OPT}(\gamma_{k\delta T+1:(k+1)\delta T})>y+(m+1)\bar{r}/\delta\right)\nonumber\\
	&=&\mathbb{P}_{\gamma\sim\mathcal{P}^T}\left(\sum_{k=0}^{k'-1} X_k>y\right)\nonumber\\
	&\leq& \exp{\left(-\frac{2y^2}{\bar{r}^2\delta^3 T}\right)}.
\end{eqnarray}
Note that $\mathbb{E}_{\gamma'\sim\mathcal{P}^T}\left[\sum_{t=1}^{k'\delta T} D_t(\mu\mid \gamma'_t)\right]/k'\delta=\mathbb{E}_{\gamma'\sim\mathcal{P}^T}\left[\sum_{t=1}^T D_t(\mu\mid \gamma'_t)\right]$, so combining Eq. (\ref{eqn:proof of thm 3,2.0.5}) from time period $t=1$ to time period $t=k'\delta T$ and Eq. (\ref{eqn:proof of thm 3,2.3}) and using union bound we get 
\begin{eqnarray*}
	\mathbb{P}_{\gamma\sim\mathcal{P}^T}\left(\text{OPT}(\gamma_{1:k'\delta T})-\sum_{k=0}^{k'-1}\text{OPT}(\gamma_{k\delta T+1:(k+1)\delta T})>2y+(m+1)\bar{r}/\delta\right) \le& \exp{\left(-\frac{2y^2}{\bar{r}^2\delta^3 T}\right)}.
\end{eqnarray*}
Take $y=\sqrt{\bar{r}^2\delta^3 T\log(T)/2}$ yields

\begin{equation}
	\label{eqn:proof of thm 3,5}\mathbb{P}_{\gamma\sim\mathcal{P}^T}\left(\text{OPT}(\gamma_{1:k'\delta T})-\sum_{k=0}^{k'-1}\text{OPT}(\gamma_{k\delta T+1:(k+1)\delta T})>\sqrt{2\bar{r}^2\delta^3 T\log(T)}+(m+1)\bar{r}/\delta\right)\leq \frac{1}{T}.
\end{equation}

For $k=0,1,\dots,k'-1$, let $R(\text{SA}\mid \gamma_{k\delta T+1:(k+1)\delta T})$ denote the reward obtained by the Stochastic Arrival Algorithm during each time block between time periods $k\delta T+1$ and $(k+1)\delta T$, then $R_{k'\delta T+1}=\sum_{k=0}^{k'-1}R(\text{SA}\mid \gamma_{k\delta T+1:(k+1)\delta T})$, where $R_t$ is the total amount of reward obtained between time periods $1$ and $t-1$ as defined in the algorithm. Apply Eq. (\ref{eqn:chernoff on SA}) on each time block 
shows that for each $k$ we have \begin{equation*}
	\mathbb{P}_{\gamma\sim\mathcal{P}^T}\left(\text{OPT}(\gamma_{k\delta T+1:(k+1)\delta T})-R_{\delta T+1}(\text{SA}\mid \gamma_{k\delta T+1},\dots, \gamma_{(k+1)\delta T})>C'\log(T)\sqrt{T}\right)\leq\frac{3}{T},
\end{equation*}
and therefore \begin{equation}\label{eqn:proof of thm 3,6}
	\mathbb{P}_{\gamma\sim\mathcal{P}^T}\left(\sum_{k=0}^{k'-1}\text{OPT}(\gamma_{k\delta T+1:(k+1)\delta T})-R_{k'\delta T+1}>C'\log(T)\sqrt{T}\right)\leq\frac{3k'}{T}\leq\frac{3}{\delta T}.
\end{equation}  

Let $L$ be a constant such that \begin{equation}\label{eqn:thm 3 constant}
	L\log(T)\sqrt{T}>C'\log(T)\sqrt{T}+\sqrt{2\bar{r}^2\delta^3 T\log(T)}+(m+1)\bar{r}/\delta.
\end{equation} Combine Eq. (\ref{eqn:proof of thm 3,5}) and Eq. (\ref{eqn:proof of thm 3,6}) gives
\begin{equation*}
	\mathbb{P}_{\gamma\sim\mathcal{P}^T}\left(\text{OPT}(\gamma_{1:k'\delta T})-R_{k'\delta T+1}>L\log(T)\sqrt{T}\right)\leq \frac{1}{T}+\frac{3}{\delta T}=\frac{3+\delta}{\delta T}.
\end{equation*}
Therefore
\begin{eqnarray*}
	&\quad&\mathbb{P}(\text{the algorithm switches to the Adversarial Arrival Algorithm incorrectly})\\
	&=&\sum_{k'=0}^{1/\delta -1}\mathbb{P}(\text{the algorithm switches at time period $k'\delta T+1$ incorrectly})\\
	&\leq&\sum_{k'=0}^{1/\delta -1}\mathbb{P}_{\gamma\sim\mathcal{P}^T}\left(\text{OPT}(\gamma_{1:k'\delta T})-R_{k'\delta T+1}>L\log(T)\sqrt{T}\right)\\
	&\leq&\frac{3+\delta}{\delta^2 T}.
\end{eqnarray*}

\noindent \textbf{Case 3:}

Suppose the underlying arrival model is adversarial and the algorithm switches to the Adversarial Arrival Algorithm at time period $k'\delta T+1$ for some $k'\in\{0,1, \dots,1/\delta-1,1/\delta\}$. Here, to simplify the notation, we set $k'=1/\delta$ if the algorithm never switches to the Adversarial Arrival Algorithm. For time periods $t_1,t_2$, let $R(\textup{MainALG}\mid \gamma)[t_1,t_2]$ be the amount of rewards that the algorithm obtained between time periods $t_1$ and $t_2$.



Because the algorithm does not switch at time period $(k'-1)\delta T+1$, we have \begin{eqnarray}\label{eqn:main case 3, 1}
	&\quad& R(\textup{MainALG}\mid \gamma)[1,(k'-1)\delta T]+L\log(T)\sqrt{T}\nonumber \\&\geq& \text{OPT}(\gamma_{1:(k'-1)\delta T})\nonumber\\&\geq& \max\left\{\frac{1}{\alpha^*}\text{OPT}(\gamma_{1:(k'-1)\delta T}),\textup{PRD}(\gamma_{1:(k'-1)\delta T})\right\}\nonumber\\
	&\geq &\max\left\{\frac{1}{\alpha^*}\text{OPT}(\gamma_{1,k'\delta T}),\textup{PRD}(\gamma_{1:(k'-1)\delta T})\right\}-\delta\bar{r} T,
\end{eqnarray} where the last inequality follows since the total rewards obtained in $\delta T$ time periods is upper bounded by $\delta\bar{r} T$.

Because the algorithm releases the remaining $\rho (T-k'\delta T)$ amount of resources for the remaining $T-k'\delta T$ time periods and performs the Adversarial Arrival Algorithm, by Theorem \ref{thm:adversarial} \begin{equation}\label{eqn:main case 3, 2}
	\max\left\{\frac{1}{\alpha^*}\text{OPT}(\gamma_{k'\delta T+1,T}),\textup{PRD}(\gamma_{k'\delta T+1,T})\right\} - R(\textup{MainALG}\mid \gamma)[k'\delta T+1,T]= o(T).
\end{equation}
Combining Eq. (\ref{eqn:main case 3, 1}) and Eq. (\ref{eqn:main case 3, 2}) gives
\begin{eqnarray*}
	&\quad& R(\textup{MainALG}\mid \gamma)+\delta\bar{r} T\\
	&=& R(\textup{MainALG}\mid \gamma)[1,(k'-1)\delta T] +R(\textup{MainALG}\mid \gamma)[k'\delta T+1,T]\\
	&\geq &\max\left\{\frac{1}{\alpha^*}\text{OPT}(\gamma_{1:k'\delta T}),\textup{PRD}(\gamma_{1:k'\delta T})\right\}-L\log(T)\sqrt{T}\\&\quad&+\max\left\{\frac{1}{\alpha^*}\text{OPT}(\gamma_{k'\delta T+1:T}),\textup{PRD}(\gamma_{k'\delta T+1:T})\right\}-o(T)\\
	&\geq &(1-\lambda)\max\left\{\frac{1}{\alpha^*}\textup{OPT}(\gamma),\textup{PRD}( \gamma)\right\}-o(T),
\end{eqnarray*}
where the last inequality follows since $\gamma$ is $(\lambda,\delta)$-stationary (\cref{def:stationary} and \cref{obs:stationary}). Hence $\limsup_{T\to\infty}\sup_{\gamma\in\mathcal{S}^T}\left\{\frac{1}{T}\left((1-\lambda)\max\left\{\frac{1}{\alpha^*}\textup{OPT}(\gamma),\textup{PRD}( \gamma)\right\} - R(\textup{MainALG}\mid \gamma)\right)\right\}\leq \delta\bar{r}.$

\noindent \textbf{Putting it all together.}

If the underlying arrival model is stochastic, combining case 1 and case 2 gives
\begin{eqnarray*}
	&&\textup{Regret}(\textup{MainALG})\\&=&\quad\mathbb{E}_{\gamma\sim\mathcal{P}^T}\left[\text{OPT}(\gamma)-R(\textup{MainALG}\mid \gamma)\mid \text{never switches}\right]\mathbb{P}(\text{never switches})\\
	&\quad&\text{}+ \mathbb{E}_{\gamma\sim\mathcal{P}^T}\left[\text{OPT}(\gamma)-R(\textup{MainALG}\mid \gamma)\mid \text{switches}\right]\mathbb{P}(\text{switches}).
\end{eqnarray*}
By case 1, $\mathbb{E}_{\gamma\sim\mathcal{P}^T}\left[\text{OPT}(\gamma)-R(\textup{MainALG}\mid \gamma)\mid \text{never switches}\right]\in\tilde{O}(\max\{T^{\frac{1}{2}-a},1\})$. By case 2, $\mathbb{P}(\text{switches})\leq\frac{3+\delta}{\delta T}$. Since $\text{OPT}(\gamma)\in O(T)$, we have $$\textup{Regret}(\textup{MainALG})= \tilde{O}(\max\{T^{\frac{1}{2}-a},1\}).$$

If the underlying arrival model is adversarial, case 3 shows $$\limsup_{T\to\infty}\sup_{\gamma\in\mathcal{S}^T}\left\{\frac{1}{T}\left( (1-\lambda)\max\left\{\frac{1}{\alpha^*}\textup{OPT}(\gamma),\textup{PRD}( \gamma)\right\} - R(\textup{MainALG}\mid \gamma)\right)\right\}\leq \delta\bar{r}.$$ This completes the proof.
\endproof

\section{Experiment Details}\label{Appendix Experiment}
\subsection{Synthetic Experiment}

The detailed setups were the following. There were 25 products, where each product was randomly assigned a unique integer price in the range of $[1, 25]$ and an embedding that lied randomly in $\mathbb{S}^{4}$. There were $26$  types of customers, consisting of $25$ customers that each corresponded to exactly one unique product, and one no-customer type, corresponding to no product being selected in that time interval. For each customer type $i$ (apart from the no-customer type), the probability that it would buy product $j$ if recommended was $\text{sigmoid}(e_i^\top \cdot e_j)/10$, where $e_i$ and $e_j$ were the $d$-dimensional embeddings for products $i$ and $j$, respectively. For the no-customer type, the probabilities were zero - we could not recommend anything. 

One instance contained $T=1000$ time periods. To build the arrival sequence, we had a function $N$ that maps each time period $t$ to a probability of observing a no-customer type at that time. If a customer did arrive, we chose its type uniformly at random. The initial inventory level was controlled by $\rho$. Modeling inventory shortages or excess inventory can be done by changing $\rho$. The price of each product was fixed at the start of the experiment and is held constant. To generate predictions, we first calculated at the true item counts in the demand sequence for each product. Depending on the arrival model, we applied various amounts of zero-mean Guassian noise with variance $\sigma$ to each of the true item counts. We then took these counts, compared them to our inventory, and determined the predicted shadow prices for each product. By changing $\sigma$, we were able to simulate predictions of different qualities.

The synthetic experiment was run on a MacBook Pro equipped with Apple's M2 Chip. The total compute time was under 20 hours. All offine optimization problems in the algorithms were solved by Gurobi.

We list all the (hyper)parameters used:
\begin{itemize}
	\item Low inventory level: $\rho = .015$, medium inventory level: $\rho = .03$, and high inventory level: $\rho = .06$;
	\item Root finding bisection parameters: $\alpha = 10^6$, $\beta = 0$, $lo = 10^{-4}$, $hi = 1$;
	\item Perfect predictions: $\sigma = 0$, good predictions: $\sigma = 5$, and bad predictions $\sigma = 500$;
	\item Stochastic arrivals: $N(t)=0.7$, nonstationary arrivals: $N(t) = .4+\frac{3t}{5000}$, and adversarial arrivals: $N(t) = \mathbbm{1}(t > 300)$;
	\item Parameters for the Main Algorithm (Algorithm $\ref{alg:main}$): $\delta = \frac{1}{20}$ and $L = 7$.
\end{itemize}

\subsection{H\&M Experiment}
\subsubsection{Background}
The H\&M dataset contains two years of online purchase data from H\&M customers, consisting of dates, purchase prices, customer IDs, and product ID. For each product, there are basic categorical information about its type, appearance, and department. For computational reasons, we only considered the 5000 most purchased products during this experiment. Our goal was to simulate 90 days of the online marketplace where when a customer selects a product, we recommend three other products in return. Encoding the days using a start day $s$, we started by building a sequence of customer/no-customer arrivals for the 90 day window: Let $R := \max_{0 \leq j \leq 89}\{\text{Amount of customers in day }s+j\}$. We initialized an empty array of size $R\cdot 90$. For a day $s+j$, for every product that was purchased in that day, we randomly placed this product in the array between indices $jR$ and $(j+1)R-1$. We call this sequence of customer/no-customer interactions our demand sequence. Note that each entry in the tuple contained the product and the price for which it was purchased.

A product’s price on a given day was set to be the price of that product purchased by some customer on a given day. To ensure that this process was deterministic, as there could be multiple customers purchasing the same product for different prices, we defined the product’s price on that day to be the first time that product was purchased by a customer on that day. If no customer purchased that product, we made the assumption that the product was unavailable and took this under consideration when recommending products during the experiment, as we could not recommend a product that is not available. In order to facilitate this experiment, we buitd an accurate model which took in two products, along with their prices, and determined the probability that those two products were bought together. We did this using sklearn’s Random Forest model. First, we created a 50-dimensional embedding for each product. This was done by creating a matrix where each $(i,j)$ entry represented that the $i$-th product was bought by the $j$-th customer. Using a matrix factorization collaborative filtering algorithm, we were able to obtain a 50-dimensional embedding for each product. Next, for each product, we created a one-hot vector for ``product\_group\_name", ``graphical\_appearance\_no", ``perceived\_colour\_value\_id", ``perceived\_colour\_master\_id", ``index\_code", ``index\_group\_no", and ``garment\_group\_no", and concatenated these one-hot vectors to form a vector of length $102$ that contains exactly 7 ones. Given two products, $p_1$ and $p_2$, we created the final 207-dimensional vector we fed into the Random Forest model by concatenating $p_1$ and $p_2$'s one-hot vectors, adding in the dot product similarity metric between the $p_1$ and $p_2$'s embeddings, and finally adding the prices for both items on that specific day. To train this model, we generated 100,000 positive instances, meaning a customer bought products $p_1$ and $p_2$ together on the same day, and 1,000,000 negative instances, where we randomly selected a product $p_1$ purchased by customer $u$ and find a product $p_2$ that was available on that day but not bought by $u$. The trained model had an AUC of $0.78$. For any two products, $p_1$ and $p_2$, that also contained correct price information for that day, we referred to this probability function as $f_{prob}(p_1, p_2)$, giving us the probability that items $p_1$ and $p_2$ were bought together on that specific day. 

In executing the Main Algorithms, we modelled the random nature of recommending products to customers, that is, we did not know whether or not a customer would select the products we recommended. To remedy this, we performed the following procedure to closely model real-world customer decision making. At each time step, we either saw a no-customer, which we would recommend no products, or we observed a product that the customer selected, say $p_{customer}$. Then the value of recommending some product $p_{rec}$ was given by $f_{prob}(p_{customer}, p_{rec})\cdot(r_t(p_{rec})-\mu_{rec} g_t(p_{rec}))$, where $r_t(p_{rec})$ represented the current price of the recommended product, $g_t(p_{rec})$ was treated as being $1$, since the customer would only consume one unit of the recommended product, and $\mu_{rec}$ was the current shadow price of the recommended item as predicted by the dual variable at time period $t$. We then recommended the top three products according to this above metric that also satisfied the inventory constraint. Note that if no three products existed to recommend, then we recommended no products. Once we recommended three products to the customer, the customer would pick each product with probability $f_{prob}(p_{customer}, p_{rec})$ and we in turn received the product's value along with the decrease in inventory only if the customer ended up buying the product. The customer could select anywhere from none to all of the recommended products, and the selections were assumed to be independent of each other.

To generate the prediction for each instance, we used $365$ days of data before the starting day of our testing window. For every $5$ day span, we added up all the products that were purchased within that interval. We took this and converted it into counts for the embedding vectors. This gave $50$ streams of $73$ data points each (one stream per embedding dimension and one data point for each of the $365/5$ combined points). From here, we run our prediction algorithm (FB Prophet, ARIMA, and Exponential Smoothing) to generate 18 more data points (as $5\times 18$ gives the full 90 days) for each of the $50$ streams, converted these back into counts for the products themselves, and determined the shadow price for each product using these predicted demands for each product. The vector of shadow prices becomes our prediction.

When running the Main Algorithm (Algorithm \ref{alg:main}), we performed sequential hypothesis testing to determine whether or not the arrival sequence was stochastic or not. We began by assuming the arrival sequence was stochastic. Then we performed the following offline hypothesis test: after allowing for a burn-in period of $20$ days, for every $t \in \{25, 30, \dots, 85\}$ we performed a one-sided one sample t-test on the number of arrivals in $[t-4, t]$ compared to the average number of arrivals in $[0, t-5]$. For sufficiently low $p$-value, chosen to be $.05$, the algorithm switched to be adversarial. Additionally, due to the large amounts of data used, using the bisection algorithm as written in the Stochastic Arrival Algorithm 
(Algorithm \ref{alg:stochastic}) was too computationally inefficient, so instead we used an approximation of this by selecting a set of $\eta$'s, $H$, and computing
\[
\eta_t := \argmin_{\eta \in H} \left\lvert \eta - \frac{\theta_t(\mu_1,\eta)}{\sqrt{\alpha\Phi_t(\mu_1,\eta)+\beta}}\right\rvert
\]
This method allowed quicker computation, as we were only running a constant with respect to $T$ versions of the Mirror Descent Algorithm for each instance. The larger our set $H$ was, the closer we would get to the solution outputted by the root finding bisection algorithm. 

We list all the (hyper)parameters used:
\begin{itemize}
	\item Prophet and Exponential Smoothing: default;
	\item ARIMA parameters: $p = 5$, $q = 2$, $d = 1$; 
	\item Random Forest classifier: $n\_$estimators = $100$, max\_depth = $18$;
	\item $H = \{10^{-10(1 - i/30)} \mid i \in [30]\}$;
	\item Stochastic Arrival Algorithm parameters: $\alpha = 1$, $\beta = 0$.
\end{itemize}
	
\end{document}